\documentclass{amsart}
\usepackage{amsmath,amsthm,amssymb,latexsym,fullpage,setspace,graphicx,float,xcolor,hyperref,verbatim,bm,soul}
\usepackage[foot]{amsaddr}
\usepackage{tikz-cd,tikz,pgfplots}
\usetikzlibrary{decorations.markings,math}
\usepackage[utf8]{inputenc}
\usepackage[english]{babel}

\theoremstyle{plain}
\newtheorem{thm}{Theorem}[section]
\newtheorem{cor}[thm]{Corollary}
\newtheorem{lem}[thm]{Lemma}
\newtheorem{alg}[thm]{Algorithm}

\newtheorem{prop}[thm]{Proposition}

\newtheorem*{remark}{Remark}
\newtheorem{eg}[thm]{Example}

\theoremstyle{definition}
\newtheorem{defn}{Definition}[section]

\makeatletter
\@namedef{subjclassname@2020}{%
  \textup{2020} Mathematics Subject Classification}
\makeatother

\begin{document}
\tikzset{->-/.style={decoration={markings,mark=at position #1 with {\arrow{>}}},postaction={decorate}}}

\title{Constructing lattice surfaces with prescribed Veech groups: an algorithm}
\author{Slade Sanderson}
\address{Department of Mathematics, Utrecht University, P.O.~Box 80010, 3508TA Utrecht, The Netherlands}
\email{s.b.sanderson@uu.nl}
\date{\today}

\begin{abstract}
The Veech group of a translation surface is the group of Jacobians of orientation-preserving affine automorphisms of the surface.  We present an algorithm which constructs all translation surfaces with a given lattice Veech group in any given stratum.  In developing this algorithm, we give a new proof of a finiteness result of Smillie and Weiss, namely that there are only finitely many unit-area translation surfaces in any stratum with the same lattice Veech group.

Our methods can be applied to obtain obstructions of lattices being realized as Veech groups in certain strata; in particular, we show that the square torus is the only translation surface in any minimal stratum whose Veech group is all of $\mathrm{SL}_2\mathbb{Z}$.
\end{abstract}

\maketitle

\section{Introduction}\label{Introduction}

The study of translation surfaces may be approached from numerous different angles---complex analysis, differential geometry, algebraic topology and number theory, for instance, each offer unique insights into these objects.  These different perspectives lead to a wealth of questions regarding translation surfaces and their dynamics.  One such avenue of inquiry is the well-known action of $\mathrm{SL}_2\mathbb{R}$ on the stratum $\mathcal{H}_1(d_1,\dots,d_\kappa)$ of unit-area translation surfaces with singularities of orders $d_1,\dots,d_\kappa$.  Stabilizers (i.e. Veech groups) and orbits of translation surfaces under this action have proven to be of great interest.  Analogues of Ratner's theorems concerning unipotent flows on homogeneous spaces (\cite{Ratner}) are proven in the celebrated work of Eskin, Mirzakhani and Mohammadi; in particular, $\mathrm{SL}_2\mathbb{R}$-orbit closures are affine invariant submanifolds of $\mathcal{H}_1(d_1,\dots,d_\kappa)$ (\cite{EMM}).  Despite having been well-studied since as early as the 1980's, some immediate questions regarding Veech groups have been non-trivial to address.  General methods of computing Veech groups were unknown until recently (\cite{Bowman}, \cite{BrJ}, \cite{Edwards}, \cite{ESS}, \cite{Mukamel}, \cite{Veech11}, and the special cases of \cite{Freidinger} and \cite{Schmithusen}), and while some universal properties of Veech groups are known---for instance, they are necessarily discrete and non-cocompact (\mbox{\cite{Veech89}})---a complete classification of which subgroups of $\mathrm{SL}_2\mathbb{R}$ are realized as Veech groups remains an open problem (\cite{HMSZ}).  

The first non-arithmetic examples were discovered by Veech, who proved that the odd index Hecke triangle groups and an index two subgroup of the even index Hecke triangle groups are realized as Veech groups (\mbox{\cite{Veech89}}).  More general triangle groups are studied in \mbox{\cite{BM}} and \mbox{\cite{Ho}}, where the authors show that there are Veech groups commensurable to all triangle groups $\Delta(m,n,\infty),\ 2\le m<n<\infty$.  On the other hand, triangle groups whose orientation preserving subgroups are never contained in a Veech group are also given in \mbox{\cite{Ho}}.  Infinitely generated Veech groups are shown to exist in \mbox{\cite{HS}} and \mbox{\cite{Mc03}}, as are non-trivial Veech groups with no parabolic elements in \mbox{\cite{HL}}.  While this brief survey is far from an exhaustive overview of the literature, it does serve to highlight that the question of realizability of Veech groups has spanned decades and remains a large and interesting problem in the field.

Of particular interest within the space of all translation surfaces are so-called \textit{lattice surfaces}, or \textit{Veech surfaces}, which are those whose Veech groups are lattices (have finite covolume) in $\mathrm{SL}_2\mathbb{R}$.  These surfaces admit especially nice dynamics (\cite{Veech89}), but the list of known families of lattice surfaces is relatively short (\cite{DaPaU}).  Lattice surfaces are also interesting as they are precisely those surfaces whose $\mathrm{SL}_2\mathbb{R}$-orbits are closed with respect to the analytic topology on the stratum (\cite{Veech95}, \cite{Smillie-Weiss04}), and the projection of a closed orbit to the moduli space of Riemann surfaces is an algebraic curve---called a \textit{Teichm\"uller curve}---isometrically immersed with respect to the Teichm\"uller metric (see, say, \cite{Wright}).  

We present an explicit algorithm which constructs all translation surfaces with a given lattice Veech group in any given stratum:
\begin{alg}\label{the_algorithm}
Input: Non-negative integers $d_1\le\dots\le d_\kappa$ whose sum is even and a finite set of generators of a lattice $\Gamma\le\mathrm{SL}_2\mathbb{R}$.

Output:  All translation surfaces $(X,\omega)\in\mathcal{H}_1(d_1,\dots,d_\kappa)$ with Veech group $\mathrm{SL}(X,\omega)=\Gamma$.
\end{alg}

An implicit algorithm is given in \cite{Smillie-Weiss10} which enumerates all affine equivalence classes of lattice surfaces in terms of a parameter measuring the infimum of areas of triangles within the surface.  Algorithm \mbox{\ref{the_algorithm}} similarly returns lattice surfaces, but it fundamentally differs both in its goal of finding surfaces with specific, prescribed Veech groups and in its method of doing so.

In developing Algorithm \ref{the_algorithm}, we obtain a new proof of a finiteness result of Smillie and Weiss (Corollary 1.7, \cite{Smillie-Weiss10_fin}; see also \cite{McMullen}):
\begin{thm}[Smillie-Weiss]\label{finiteness_thm}
There are at most finitely many unit-area lattice surfaces with a given Veech group in any given stratum.
\end{thm}

Our methods essentially reverse the algorithm for computing Veech groups presented in \cite{Edwards} and \cite{ESS}.  There the authors associate to each stratum a canonical (infinite area, disconnected in general) flat surface $\mathcal{O}$, within which any closed, connected translation surface $(X,\omega)$ of the stratum may be naturally represented.  Data regarding the translation surface and its saddle connections are recorded via pairs of points in $\mathcal{O}$ termed \textit{orientation-paired marked segments}, and the authors show that $(X,\omega)$ may be recovered from a certain finite subset of these marked segments---the \textit{marked Voronoi staples} (\S3.4 of \cite{ESS}).  Furthermore, a classification of the Veech group of $(X,\omega)$ in terms of its orientation-paired marked segments and affine automorphisms of $\mathcal{O}$ is given in Proposition 17 of \cite{ESS}.  

Rather than beginning with a lattice surface $(X,\omega)$ and using its corresponding orientation-paired marked segments to compute elements of the Veech group $\Gamma=\mathrm{SL}(X,\omega)$, Algorithm \ref{the_algorithm} begins with a lattice $\Gamma\le\mathrm{SL}_2\mathbb{R}$, simulates normalized subsets of orientation-paired marked segments in $\mathcal{O}$ using $\Gamma$, and constructs candidate translation surfaces with Veech group $\Gamma$ from finite unions of appropriately scaled versions of these `simulations.'  Theorem \ref{finiteness_thm} follows from the facts that 
\begin{enumerate}
\item[(i)] a lattice $\Gamma$ determines only finitely many such simulations (Lemma \mbox{\ref{Sims_G_finite}}),
\item[(ii)] the orientation-paired marked segments of any unit-area translation surface with Veech group $\Gamma$ are realized as a finite union of scaled simulations (Theorem \mbox{\ref{marked_segments_as_scaled_simulations}}), and 
\item[(iii)] for any finite collection of simulations there are only finitely many scalars for which the union of scaled simulations could possibly coincide with the orientation-paired marked segments of a unit-area translation surface.
\end{enumerate}
For this latter point, we introduce the technical but important notions of \textit{permissible triples} and \textit{permissible scalars} of particular subsets of $\mathcal{O}$, prove finiteness results regarding them (Proposition \ref{scalars_unique} and Lemma \ref{permissible_triples_finite}), and apply these results to orientation-paired marked segments and marked Voronoi staples (Proposition \ref{vor_staples_det_perm_triples_prop}).  Our novel proof differs from that of Smillie and Weiss, who use finiteness results of Markov partitions corresponding to hyperbolic affine automorphisms of surfaces (\mbox{\cite{Smillie-Weiss10_fin}}).

This work should be viewed in light of the aforementioned open question regarding the realization of subgroups of $\mathrm{SL}_2\mathbb{R}$ as Veech groups: a full implementation\footnote{A partial implementation has been written by the author using \mbox{\cite{SageMath}} to confirm some low genus examples with small generating sets, though an efficient version of the algorithm which handles more interesting examples awaits implementation.} of Algorithm \ref{the_algorithm} could lead to the discovery of new lattice surfaces and experimental conjectures as to which groups are realized as Veech groups.  Moreover, further work to determine a halting criterion for Algorithm \ref{the_algorithm} would give a general, explicit procedure to determine whether or not a given lattice is a Veech group in any particular stratum (see comments at the end of \S\ref{The algorithm}).  At present, these ideas may be used in special cases to obtain obstructions for lattices being realized as Veech groups.  In particular, we prove:
\begin{thm}\label{square_torus_thm}
The square torus is the only translation surface in $\cup_{g>0}\mathcal{H}(2g-2)$ with Veech group $\mathrm{SL}(X,\omega)=\mathrm{SL}_2\mathbb{Z}$.
\end{thm}

The paper is organized as follows.  Section \ref{Background} gives a brief background on translation surfaces (\S\ref{Translation surfaces}) and their Veech groups (\S\ref{Veech groups}).  Section \ref{A canonical surface for each stratum} concerns the canonical surface $\mathcal{O}$ studied in \cite{Edwards} and \cite{ESS}: definitions and immediate results regarding $\mathcal{O}$ are given in \S\ref{Canonical surface}, permissible triples and scalars are introduced and studied in \S\ref{Permissible triples}, and in \S\ref{Marked segments and Voronoi staples} we recall definitions and results regarding marked segments and Voronoi staples of a translation surface.  Section \ref{Fanning groups, simulations and finiteness of lattices in strata} further develops the machinery for Algorithm \ref{the_algorithm} and uses these ideas to give a new proof of Theorem \ref{finiteness_thm}: a characterization of lattice groups in terms of directions announced by linear transformations is given in \S\ref{Fanning groups and lattices}; subsections \ref{The group Aff^+_O(X,omega) and its action on marked segments} and \ref{Simulating orbits} explore a group action on marked segments and how a lattice group may be used to simulate orbits of this action; subsection \ref{Marked Voronoi staples determine permissible triples} relates marked Voronoi staples to permissible triples; and a new proof of Theorem \ref{finiteness_thm} is given in \S\ref{Finiteness of lattice surfaces with given Veech groups}.  Algorithm \ref{the_algorithm} is presented in \S\ref{The algorithm} and Theorem \ref{square_torus_thm} is proven in \S\ref{The modular group in minimal strata}.

\section{Background}\label{Background}

\subsection{Translation surfaces}\label{Translation surfaces}

We begin with basic definitions and notation; see, say, surveys of \cite{Wright} and \cite{Zorich} for more details.

\subsubsection{Translation surfaces: various perspectives}

A \textit{translation surface} $(X,\omega)$ is a Riemann surface $X$ together with a non-zero, holomorphic one-form $\omega$.  Equivalently, a translation surface is defined as a real surface $X$ such that off of a finite set $\Sigma$, the surface $X\backslash \Sigma$ is equipped with a translation atlas, and the resulting flat structure extends to all of $X$ to give conical singularities at points of $\Sigma$ whose angles are positive integral multiples of $2\pi$.  Any translation surface may be \textit{polygonally represented} as a collection of polygons in the plane together with an identification of edges in pairs, such that identified edges are parallel, equal length and of opposite orientation.  A translation surface is \textit{closed} if it is compact and has no boundary.

In these three perspectives---which we move fluidly between---the zeros of $\omega$, elements of $\Sigma$ of cone angles $2\pi (d+1),\ d>0,$ and, after identifying edges, the vertices of angles $2\pi (d+1),\ d>0,$ of a polygonally represented translation surface each correspond to the same set of  \textit{non-removable singularities} of the surface.  Throughout, the \textit{singular set} $\Sigma$ denotes these singularities together with possibly finitely many \textit{removable singularities} (or \textit{marked points}), i.e. points of cone angle $2\pi$.  We refer to any element of $\Sigma$ as a singularity (of cone angle $2\pi(d+1),\ d\ge 0$), bearing in mind that it may in fact be removable.  Points in the complement $X\backslash\Sigma$ are called \textit{regular}.

\subsubsection{Metric, Lebesgue measure, saddle connections and holonomy vectors}\label{Metric, Lebesgue measure, saddle connections and holonomy vectors}

A translation surface $(X,\omega)$ comes equipped with a metric defined as follows: given $p,q\in X$, the distance from $p$ to $q$ is given by
\[d(p,q):=\inf_{\gamma}\left|\int_\gamma \omega\right|,\]
where the infimum is over all piecewise-smooth curves $\gamma\subset X$ from $p$ to $q$.  The infimum is realized by a straight line segment (with respect to the flat structure) on $(X,\omega)$ or by a union of such segments meeting at singularities.  Lebesgue measure in the plane pulls back via coordinate maps of the translation atlas of $(X,\omega)$ to give a measure on $X\backslash \Sigma$; we extend this measure to all of $X$ by declaring $\Sigma$ to be a null set.

A \textit{separatrix} is a straight line segment on $(X,\omega)$ emanating from a singularity and having no singularities in its interior.  A \textit{saddle connection} is a separatrix (of positive length) which also ends at a singularity.  The \textit{holonomy vector} of a saddle connection $s$ is defined as $\text{hol}(s):=\int_s\omega$.  The collection of all holonomy vectors of $(X,\omega)$ forms a subset of the plane with no limit points and whose directions are dense in $S^1$ (Proposition 3.1, \cite{Vorobets}).

\subsubsection{Voronoi decomposition}

Every translation surface $(X,\omega)$ has a (unique) \textit{Voronoi decomposition} subordinate to its singular set $\Sigma$, which is described as follows (see also \cite{Bowman}, \cite{Masur-Smillie}).  For any $x\in X$, let $d(x,\Sigma)$ denote the minimum of the set $\{d(x,\sigma)\}_{\sigma\in\Sigma}$.  To each $\sigma\in\Sigma$ there is associated an open, connected 2-cell, $\mathcal{C}_\sigma$, comprised of the points $x\in X$ for which $d(x,\Sigma)$ is realized by a unique length-minimizing path ending at $\sigma$.  The boundary of $\mathcal{C}_\sigma$ is the union of 1-cells of points $x$ for which $d(x,\Sigma)$ is realized by precisely two length-minimizing paths, at least one of which ends at $\sigma$, and of 0-cells of points $x$ for which $d(x,\Sigma)$ is realized by three or more length-minimizing paths, at least one of which ends at $\sigma$.

\subsection{Veech groups}\label{Veech groups}

\subsubsection{Affine diffeomorphisms and Veech groups}\label{affine_diffeomorphisms...}

An \textit{affine diffeomorphism} $f:(X_1,\omega_1)\to(X_2,\omega_2)$ between (not necessarily closed nor connected) translation surfaces is a diffeomorphism from $X_1$ to $X_2$ sending the singular set $\Sigma_1$ of $X_1$ into that of $X_2$, and which---on the complement of $\Sigma_1$---is locally an affine map of the plane of constant linear part\footnote{The translation structures guarantee that the linear part of a locally affine map is constant on components of $(X_1,\omega_1)$; this latter statement in the definition is thus intended for disconnected translation surfaces.} (i.e. a map of the form $v\mapsto Av+b$ for some global $A\in\mathrm{GL}_2\mathbb{R}$ and local $b\in\mathbb{R}^2$).  An \textit{affine automorphism} is an affine diffeomorphism from a translation surface $(X,\omega)$ to itself.  The set of all affine automorphisms of $(X,\omega)$, denoted $\text{Aff}(X,\omega)$, forms a group under composition, as does the subset $\text{Aff}^+(X,\omega)\subset\text{Aff}(X,\omega)$ of orientation-preserving elements.  The map $\text{der}:\text{Aff}(X,\omega)\to\mathrm{GL}_2\mathbb{R}$ sending an orientation-preserving affine automorphism to its linear part in local coordinates gives a group homomorphism.  The image of $\text{Aff}^+(X,\omega)$ under this map, denoted $\mathrm{SL}(X,\omega)$, is called the \textit{Veech group} of $(X,\omega)$, and its kernel is the group of translations denoted by $\text{Trans}(X,\omega)$.  Veech shows in \mbox{\cite{Veech89}} that for closed, connected $(X,\omega)$, the Veech group $\mathrm{SL}(X,\omega)$ is a discrete subgroup of $\mathrm{SL}_2\mathbb{R}$---and thus its image in $\text{PSL}_2\mathbb{R}$ is a Fuchsian group---and, moreover, $\mathrm{SL}(X,\omega)$ is always non-cocompact.

\subsubsection{Strata, a $\mathrm{GL}_2\mathbb{R}$-action and lattices}

We call $(X_1,\omega_1)$ and $(X_2,\omega_2)$ \textit{translation equivalent} if there is an affine diffeomorphism between them with trivial linear part.  For non-negative integers $d_1\le\dots\le d_\kappa$, let $\mathcal{H}(d_1,\dots,d_\kappa)$ denote the \textit{stratum} of all closed, connected translation surfaces---up to translation equivalence---with singularities of cone angles $2\pi(d_1+1)\le\dots\le 2\pi(d_\kappa+1)$ and no other singularities. (Recall that we allow for `removable singularities,' so some $d_i$ may be zero, and with our definitions, $\mathrm{SL}(X,\omega)$ depends on these marked points in $\Sigma$.  In particular, Algorithm \mbox{\ref{the_algorithm}} applies not only to standard strata where each $d_i>0$, but it also allows for the construction of lattice surfaces $(X,\omega)$ with a specified number of marked points which remain invariant under $\text{Aff}^+(X,\omega)$.)  Abusing notation, we denote an element of a stratum by any of its translation-equivalent representatives.  By the Riemann-Roch Theorem, every translation surface $(X,\omega)\in\mathcal{H}(d_1,\dots,d_\kappa)$ has the same genus $g>0$, and the sum of the $d_i$ equals $2g-2$.

There is a natural action of $\mathrm{GL}_2\mathbb{R}$ on each stratum, given by post-composing the coordinate charts of $(X,\omega)\in\mathcal{H}(d_1,\dots,d_\kappa)$ with the usual action of a matrix in the plane via a linear transformation (one can verify that this action is well-defined with respect to translation-equivalence).  Any translation surface may be normalized by the action of a diagonal matrix so as to have unit-area; we denote the collection of unit-area translation surfaces in $\mathcal{H}(d_1,\dots,d_\kappa)$ by $\mathcal{H}_1(d_1,\dots,d_\kappa)$.  There is a corresponding action of $\mathrm{SL}_2\mathbb{R}$ on $\mathcal{H}_1(d_1,\dots,d_\kappa)$, and the stabilizer of $(X,\omega)$ under this $\mathrm{SL}_2\mathbb{R}$-action is isomorphic to the Veech group $\mathrm{SL}(X,\omega)$ as defined in \S\ref{affine_diffeomorphisms...}.  We call $(X,\omega)$ a \textit{lattice surface} if $\mathrm{SL}(X,\omega)$ is a lattice group, i.e. $\mathrm{SL}(X,\omega)$ has finite covolume in $\mathrm{SL}_2\mathbb{R}$.  

\begin{remark}
Recall from \S\mbox{\ref{affine_diffeomorphisms...}} that a Veech group is always non-cocompact, and thus a lattice Veech group is non-uniform (i.e. it is a non-cocompact discrete group of finite covolume).  Other necessary properties of lattice Veech groups are known; for example, the trace field of $\mathrm{SL}(X,\omega)$ is a totally real number field of degree at most $g$, where $g$ is the genus of $(X,\omega)$ (\mbox{\cite{GJ}},\mbox{\cite{HL}},\mbox{\cite{KS}}).  Throughout, when speaking of a lattice $\Gamma\le\mathrm{SL}_2\mathbb{R}$ the reader may wish to impose such known restrictions, though they are not strictly necessary for our statements.  By a lattice we simply mean a discrete subgroup of $\mathrm{SL}_2\mathbb{R}$ of finite covolume.
\end{remark}

\section{A canonical surface for each stratum}\label{A canonical surface for each stratum}

\subsection{Canonical surface}\label{Canonical surface}

Here we recall the canonical (infinite area and disconnected in general) flat surface associated to each stratum $\mathcal{H}(d_1,\dots,d_\kappa)$ studied in \cite{Edwards} and \cite{ESS} and explore some immediate results.  For each $1\le i\le\kappa$, let $\mathcal{O}_i:=(\mathbb{C},z^{d_i}dz)$, and set 
\[\mathcal{O}=\mathcal{O}(d_1,\dots,d_\kappa):=\bigsqcup_{i=1}^\kappa \mathcal{O}_i.\]
Each component $\mathcal{O}_i\subset\mathcal{O}$ is an infinite translation surface with a sole singularity of cone angle $2\pi(d_i+1)$ at the origin, which we denote by $0\in\mathcal{O}_i$ (the context will be clear as to which component an origin $0$ belongs).  
Intuitively, we may think of $\mathcal{O}_i$ as $d_i+1$ copies of the plane, denoted $c_0^i,\dots,c_{d_i}^i$, each of which is slit along the non-negative real axis and glued so that the bottom edge, $b_j^i$, of the slit of $c_j^i$ is identified with the top edge, $t_{j+1}^i$, of the slit of $c_{j+1}^i$ for each $j\in\mathbb{Z}_{d_i+1}$; see Figure \ref{O_i_figure}.  After making these identifications, we (arbitrarily) consider $t_{j+1}^i\sim b_j^i$ as part of $c_{j+1}^i$ and not of $c_{j}^i$.  In particular, $c_j^i\cap c_k^i=\{0\}$ for distinct $j,k\in\mathbb{Z}_{d_i+1}$.

\begin{figure}[t]
\centering
\begin{tikzpicture}
\draw [->,dashed,line width=1] (0,0) -- (1.5,0) node[midway,above] {$t_0^i$} node[midway,below] {$b_0^i$};
\draw [<-,line width=1] (-1.5,0) -- (0,0);
\draw [<->,line width=1] (0,-1.5) -- (0,1.5);
\draw (1.4,1.4) circle (0pt) node {$c_0^i$};
\draw [->,dashed,line width=1] (4.5,0) -- (6,0) node[midway,above] {$t_1^i$} node[midway,below] {$b_1^i$};
\draw [<-,line width=1] (3,0) -- (4.5,0);
\draw [<->,line width=1] (4.5,-1.5) -- (4.5,1.5);
\draw (5.9,1.4) circle (0pt) node {$c_1^i$};
\draw (7,0) circle (0.025) [fill=black];
\draw (7.25,0) circle (0.025) [fill=black];
\draw (7.5,0) circle (0.025) [fill=black];
\draw [->,dashed,line width=1] (10,0) -- (11.5,0) node[midway,above] {$t_{d_i}^i$} node[midway,below] {$b_{d_i}^i$};
\draw [<-,line width=1] (8.5,0) -- (10,0);
\draw [<->,line width=1] (10,-1.5) -- (10,1.5);
\draw (11.4,1.4) circle (0pt) node {$c_{d_i}^i$};
\end{tikzpicture}
\caption{The surface $\mathcal{O}_i$, where the bottom of the slit, $b_{j}^i$, is identified with the top of the slit, $t_{j+1}^i$, for each $j\in\mathbb{Z}_{d_i+1}$.}
\label{O_i_figure}
\end{figure}
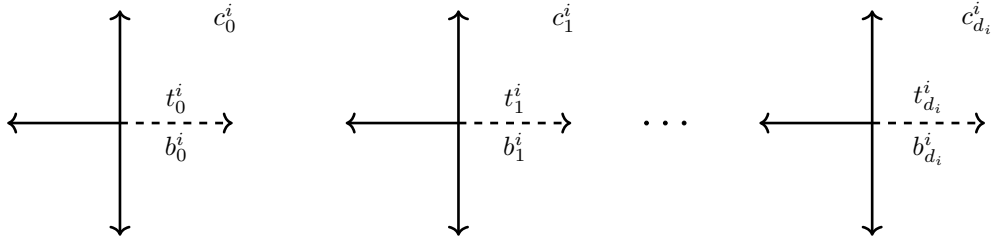

Define ${\bm\pi}_i:\mathcal{O}_i\to\mathbb{C}$ by ${\bm\pi}_i(p):=\int_\gamma z^{d_i}dz$, where $\gamma$ is any piecewise-smooth curve from $0$ to $p$ (note that the integral is independent of path due to Cauchy's integral theorem), and extend to a map ${\bm\pi}:\mathcal{O}\to\mathbb{C}$ by setting ${\bm\pi}|_{\mathcal{O}_i}:={\bm\pi}_i$ for each $i\in\{1,\dots,\kappa\}$.  Observe that ${\bm\pi}_i(\cdot)$ is a degree $d_i+1$ branched covering with a single branch point $0\in\mathbb{C}$, and the extended map ${\bm\pi}(\cdot)$ is a degree $\sum_{i=1}^\kappa (d_i+1)=2g-2+\kappa$ branched covering with sole branch point $0
\in\mathbb{C}$.  Throughout, we denote by $z_x$ and $z_y$ the real and imaginary parts, respectively of $z\in\mathbb{C}$; hence ${\bm\pi}(p)={\bm\pi}(p)_x+i{\bm\pi}(p)_y$.
Note that on sufficiently small neighborhoods in $\mathcal{O}$ (and away from singularities), the map ${\bm\pi}(\cdot)$ restricts to give a coordinate chart for the translation atlas associated to $\mathcal{O}$.  Moreover, each $\mathcal{O}_i$ is naturally equipped with generalized polar coordinates: every point $p\in c_j^i\subset \mathcal{O}_i$ is determined by a magnitude $|p|:=d(0,p)=|{\bm\pi}(p)|$ and, for $p\neq 0$, a principal argument value $\text{arg}(p):=\text{arg}({\bm\pi}(p))+2\pi j\in[2\pi j, 2\pi(j+1))$, where $|z|$ and $\text{arg}(z)\in [0,2\pi)$ are the Euclidean norm and principal value of the argument function, respectively, of $z\in\mathbb{C}$.

For each $i\in\{1,\dots,\kappa\}$, define an affine automorphism $\rho_i:\mathcal{O}_i\to\mathcal{O}_i$ which acts as a counterclockwise rotation of angle $2\pi$ about the origin; note that $\rho_i$ belongs to $\text{Trans}(\mathcal{O}_i)$ since, in local coordinates given by ${\bm\pi}(\cdot)$, $\rho_i$ has trivial linear part.  We extend $\rho_i$ to $\mathcal{O}$ by acting as the identity on all other components and observe that $\rho_i\in\text{Trans}(\mathcal{O})$.  For each integer $d\ge 0$, let $n(d)=\#\{1\le i\le\kappa\ |\ d_i=d\}$.  Having fixed $d$, let $i_1,\dots,i_{n(d)}$ be the distinct indices of the $d_1,\dots,d_\kappa$ which equal $d$, and let $S_{n(d)}$ denote the symmetric group on the set $\{i_1,\dots,i_{n(d)}\}$.  Extend $S_{n(d)}$ to act as the identity on all other indices, and for each $\alpha\in S_{n(d)}$, let $f_\alpha:\mathcal{O}\to\mathcal{O}$ permute the indices of $\mathcal{O}_1,\dots,\mathcal{O}_\kappa$ by $\alpha$, respecting polar coordinates.  We again have that $f_\alpha\in\text{Trans}(\mathcal{O})$ for each $\alpha\in S_{n(d)}$ with $n(d)\ge 1$.  

\begin{lem}\label{Trans(O)}
The group $\text{Trans}(\mathcal{O})$ is in bijective correspondence with
\[\left(\prod_{\substack{d\ge 0,\\ n(d)\ge 1}}S_{n(d)}\right)\times\left(\prod_{i=1}^\kappa C_{d_i+1}\right),\]
where $C_{d_i+1}$ is the cyclic group of order $d_i+1$.
\end{lem}
\begin{proof}[Sketch]
Lemma 2 of \cite{ESS} shows that $\text{Trans}(\mathcal{O})$ is generated by an action of 
\[\prod_{\substack{d\ge 0,\\ n(d)\ge 1}}S_{n(d)}\] 
together with an action of $\prod_{i=1}^\kappa C_{d_i+1}$ on $\mathcal{O}$, as naturally defined from the definitions of $f_\alpha$ and $\rho_i$ above.  That is, any element $\tau\in\text{Trans}(\mathcal{O})$ may be written as a finite composition of various $f_\alpha$'s and $\rho_i$'s.  We may rearrange the order of this composition (possibly altering the indices of some $\rho_i$'s) to collect all $f_\alpha$'s with $\alpha$ belonging to the same $S_{n(d)}$ and all $\rho_i$'s corresponding to the same $\mathcal{O}_i$.  That $\tau$ may be uniquely written in such a way is verified by considering the indices of images of $2\pi$-sectors $c_j^i$.
\end{proof}

\begin{lem}\label{pi*f=A*pi}
Let $f\in\text{Aff}(\mathcal{O})$ with $\text{der}(f)=A$.  Then 
\[{\bm\pi}\circ f=A\cdot{\bm\pi},\]
where the notation on the right denotes the usual action of a matrix $A$ on the plane as a linear transformation.
\end{lem} 
\begin{proof}
Since ${\bm\pi}(\cdot)$ restricts on sufficiently small neighborhoods to give a (bijective) coordinate map for the translation atlas of $\mathcal{O}$, we have---by definition of $f$---that for each $v\in\mathbb{R}^2$ in the image of such a neighborhood,
\[{\bm\pi}\circ f\circ{\bm\pi}^{-1}(v)=Av+b\]
for some $b\in\mathbb{R}^2$.  This composition agrees on the images of the intersections of such neighborhoods; in particular, $b$ depends only on the component $\mathcal{O}_i$ to which the neighborhood belongs.  Choose some neighborhood as above which contains $0\in\mathcal{O}_i$ on its boundary.  The above composition (and each of the maps comprising it) extends continuously to the closure of its domain.  Since $f$ sends singularities to singularities, setting $v=0$ gives $b=0$.  As $\mathcal{O}_i$ was arbitrary, the result follows from the previous observations.  
\end{proof}

In Lemma 3 of \cite{ESS}, it is shown that $\text{der}(\text{Aff}(\mathcal{O}))=\mathrm{GL}_2\mathbb{R}$.  To each $A\in\mathrm{GL}_2^+\mathbb{R}$ we associate a canonical $f_A\in\text{Aff}^+(\mathcal{O})$ with $\text{der}(f_A)=A$ as follows.  Fix some $g_A\in\text{Aff}^+(\mathcal{O})$ with $\text{der}(g_A)=A$, and choose $\tau\in\text{Trans}(\mathcal{O})$ so that the composition $f_A:=\tau\circ g_A$ satisfies for each $1\le i\le \kappa$ both (i) $f_A(\mathcal{O}_i)=\mathcal{O}_i$ and (ii) for each point $p\in\mathcal{O}_i$, the angle measured counterclockwise from $p$ to $f_A(p)$ is non-negative and less than $2\pi$.  Existence of $\tau$ follows from Lemma \ref{Trans(O)}.  Notice that $\text{der}(f_A)=A$ since $\text{der}(\cdot)$ is a group homomorphism, $\tau\in\text{Trans}(\mathcal{O})=\text{ker}(\text{der})$ and $\text{der}(g_A)=A$.  Furthermore, conditions (i) and (ii) guarantee that $f_A$ is unique, regardless of the initial choice of $g_A$ (and subsequent choice of $\tau$).

We also note that any $g_A\in\text{Aff}^+(\mathcal{O})$ with $\text{der}(g_A)=A\in\mathrm{GL}_2^+\mathbb{R}$ may be written uniquely as $g_A=\tau'\circ f_A$ for some $\tau'\in\text{Trans}(O)$ (namely $\tau'=\tau^{-1}$ for $\tau$ as above), with $f_A$ the canonical affine automorphism associated to $A$.

\begin{defn}
With notation as above, let 
\[\text{Aff}_{\text{C}}^+(\mathcal{O}):=\{f_A\in\text{Aff}^+(\mathcal{O})\ |\ A\in\mathrm{GL}_2^+\mathbb{R}\}\]
be the collection of canonical affine automorphisms of $\mathcal{O}$.
For any $r\in\mathbb{R}_+$ and $p\in\mathcal{O}$, let 
\[rp:=f_{D(r)}(p),\]
where $f_{D(r)}\in\text{Aff}_{\text{C}}^+(\mathcal{O})$ with $D(r)=(\begin{smallmatrix}r & 0\\ 0 & r\end{smallmatrix})$.
\end{defn}

In particular, if $p\in \mathcal{O}_i$ with polar coordinates $(|p|,\theta)$, then $rp\in\mathcal{O}_i$ with polar coordinates $(r|p|,\theta)$.  Note then that ${\bm\pi}(rp)=r{\bm\pi}(p)$.  

While $\text{Trans}(\mathcal{O})$ is generally non-abelian, we have the following commutativity result:
\begin{lem}\label{centralizer_lemma}
The subgroup $\text{Trans}(\mathcal{O})\le\text{Aff}(\mathcal{O})$ belongs to the centralizer of $\text{Aff}_{\text{C}}^+(\mathcal{O})$, i.e.
\[\tau\circ f_A=f_A\circ\tau\]
for each $\tau\in\text{Trans}(\mathcal{O})$ and $f_A\in\text{Aff}_{\text{C}}^+(\mathcal{O})$.
\end{lem}
\begin{proof}
Let $\tau\in\text{Trans}(\mathcal{O}),\ f_A\in\text{Aff}_{\text{C}}^+(\mathcal{O})$ and $p\in\mathcal{O}$.  We must show 
\[\tau \circ f_A(p)=f_A\circ \tau(p).\]
Since $\text{der}(\tau\circ f_A)=\text{der}(f_A\circ\tau)=A$, both sides of the previous line are sent under ${\bm\pi}(\cdot)$ to the same point $z\in\mathbb{C}$ (Lemma \ref{pi*f=A*pi}).  Suppose $p\in c_k^i$ and $\tau(\mathcal{O}_i)=\mathcal{O}_j$ with $\tau(c_k^i)=c_\ell^j$.  Let $r_k^i$ denote the set of points in $\mathcal{O}_i$ with argument $2\pi k$ (i.e. $r_k^i$ is the ray along the positive real axis in $c_k^i$).  Note that ${\bm\pi}(\cdot)$ sends each of $\tau\circ f_A(c_k^i)$ and $f_A\circ \tau(c_k^i)$ bijectively onto $\mathbb{C}$, and these sets are completely determined by the images $\tau\circ f_A(r_k^i)$ and $f_A\circ \tau(r_k^i)$ in $\mathcal{O}_j$: the former are the sets of points within angle $2\pi$ counterclockwise of the respective images of the ray $r_k^i$.  Since there is only one point in each $2\pi$-sector which maps to $z\in\mathbb{C}$ under ${\bm\pi}(\cdot)$, it suffices to show that $\tau\circ f_A(r_k^i)=f_A\circ \tau(r_k^i)$.  By definition of $f_A$ and $r_k^i$, the image $f_A(r_k^i)$ belongs to $c_k^i$, and thus $\tau\circ f_A(r_k^i)$ belongs to $c_\ell^j$.  But also $\tau(r_k^i)$ belongs to $c_\ell^j$ by assumption, and again by definition of $f_A$, the image $f_A\circ \tau(r_k^i)$ belongs to $c_\ell^j$ as well.  Since the images $\tau\circ f_A(r_k^i)$ and $f_A\circ \tau(r_k^i)$ are rays emanating from $0$ in $c_\ell^j$ and pointing in the same direction, we have $\tau\circ f_A(r_k^i)=f_A\circ \tau(r_k^i)$ and the result follows.
\end{proof}

\subsection{Permissible triples}\label{Permissible triples}

Here we introduce terminology and present results regarding particular subsets of the canonical surface $\mathcal{O}=\mathcal{O}(d_1,\dots,d_\kappa)$.  While the material of this subsection becomes technical, the underlying notions are rooted in elementary Euclidean geometry.  As we shall see in \S\ref{Marked Voronoi staples determine permissible triples} and \S\ref{Finiteness of lattice surfaces with given Veech groups}, the results proven here will be crucial for our proof of Theorem \ref{finiteness_thm}.  We begin with definitions and notation.

\begin{defn}\label{half-space}
The \textit{half-space} determined by $p\in\mathcal{O}_i$ is 
\[H(p):=\{q\in\mathcal{O}_i\ |\ d(0,q)\le d(p,q)\}.\]
\end{defn}

Note that $H(p)$ is convex in the sense that the length-minimizing path between any two points in $H(p)$ is also contained in $H(p)$.  Unless otherwise noted, a (closed or open) \textit{$\theta$-sector} of $\mathcal{O}_i$ is a (closed or open) sector of infinite radius, centered at $0\in\mathcal{O}_i$, and of angle $\theta$.  For the following definition and subsequent discussion, see Figure \mbox{\ref{circumcircle_etc}}.

\begin{figure}[t]
\begin{minipage}{.49\textwidth}
\centering
\begin{tikzpicture}[scale=1.5]
\draw [->,dashed,line width=1] (0,0) -- (2,0); 
\draw [<-,line width=1] (-2,0) -- (0,0);
\draw [<->,line width=1] (0,-.8) -- (0,2); 
\draw [dashed] (-.6,.8) circle (1);
\filldraw (0,0) circle (1pt) node[below right] {$0$};
\filldraw (-.6,.8+1) circle (1pt) node[above] {$p$};
\filldraw ({-.6-sqrt(3)/2},.8-.5) circle (1pt) node[left] {$q$};
\filldraw (-.6,.8) circle (1pt); 
\draw [->] (-1.7,1) node[left] {$c(p,q)$} -- (-.68,.82);
\draw (0,0) -- (-.6,.8+1) -- ({-.6-sqrt(3)/2},.8-.5) -- cycle;
\draw [<-] (.45,.8) -- (.7,.8) node[right] {$B(p,q)$};
\draw [->] (-1.45,-.3) node[below] {$\triangle(p,q)$} -- (-1.1,.2);
\draw[dashed] ({-1.8*(-2.3)/sqrt(.6^2+1.8^2)-.6/2},{-.6*(-2.3)/sqrt(.6^2+1.8^2)+1.8/2}) -- ({-1.8*(1.7)/sqrt(.6^2+1.8^2)-.6/2},{-.6*(1.7)/sqrt(.6^2+1.8^2)+1.8/2});
\draw[dashed] ({-.3*(-1.9)/sqrt((-.6-sqrt(3)/2)^2+.3^2)+(-.6-sqrt(3)/2)/2},{(-.6-sqrt(3)/2)*(-1.9)/sqrt((-.6-sqrt(3)/2)^2+.3^2)+.3/2}) -- ({-.3*(.9)/sqrt((-.6-sqrt(3)/2)^2+.3^2)+(-.6-sqrt(3)/2)/2},{(-.6-sqrt(3)/2)*(.9)/sqrt((-.6-sqrt(3)/2)^2+.3^2)+.3/2});
\end{tikzpicture}
\end{minipage}
\begin{minipage}{.49\textwidth}
\centering
\begin{tikzpicture}[scale=1.5]
\draw [->,line width=1] (0,0) -- (2,0); 
\draw [<-,line width=1] (-2,0) -- (0,0);
\draw [<->,line width=1] (0,-.8) -- (0,2); 
\draw [dashed] (-.6,.8) circle (1);
\filldraw (0,0) circle (1pt) node[below right] {$0$};
\filldraw (-.6,.8+1) circle (1pt) node[above] {${\bm\pi}(p)$};
\filldraw ({-.6-sqrt(3)/2},.8-.5) circle (1pt) node[label={[shift={(-.4,-.55)}]${\bm\pi}(q)$}] {};
\filldraw (-.6,.8) circle (1pt); 
\draw [->] (-1.7,1) node[left] {${\bm\pi}(c(p,q))$} -- (-.68,.82);
\draw (0,0) -- (-.6,.8+1) -- ({-.6-sqrt(3)/2},.8-.5) -- cycle;
\draw [<-] (.45,.8) -- (.7,.8) node[right] {${\bm\pi}(B(p,q))$};
\draw [->] (-1.45,-.3) node[below] {${\bm\pi}(\triangle(p,q))$} -- (-1.1,.2);
\draw[dashed] ({-1.8*(-2.3)/sqrt(.6^2+1.8^2)-.6/2},{-.6*(-2.3)/sqrt(.6^2+1.8^2)+1.8/2}) -- ({-1.8*(1.7)/sqrt(.6^2+1.8^2)-.6/2},{-.6*(1.7)/sqrt(.6^2+1.8^2)+1.8/2});
\draw[dashed] ({-.3*(-1.9)/sqrt((-.6-sqrt(3)/2)^2+.3^2)+(-.6-sqrt(3)/2)/2},{(-.6-sqrt(3)/2)*(-1.9)/sqrt((-.6-sqrt(3)/2)^2+.3^2)+.3/2}) -- ({-.3*(.9)/sqrt((-.6-sqrt(3)/2)^2+.3^2)+(-.6-sqrt(3)/2)/2},{(-.6-sqrt(3)/2)*(.9)/sqrt((-.6-sqrt(3)/2)^2+.3^2)+.3/2});
\end{tikzpicture}
\end{minipage}
\caption{Left: The triangle $\triangle(p,q)$, circumcenter $c(p,q)$, and ball $B(p,q)$ determined by $p$ and $q$ in $\mathcal{O}_i$.  The circumcircle $C(p,q)$ is the boundary of $B(p,q)$, and the straight dashed lines through $c(p,q)$ are the boundaries of the half-spaces $H(p)$ and $H(q)$.  Right: The image in $\mathbb{C}$ under ${\bm\pi}(\cdot)$.}
\label{circumcircle_etc}
\end{figure}
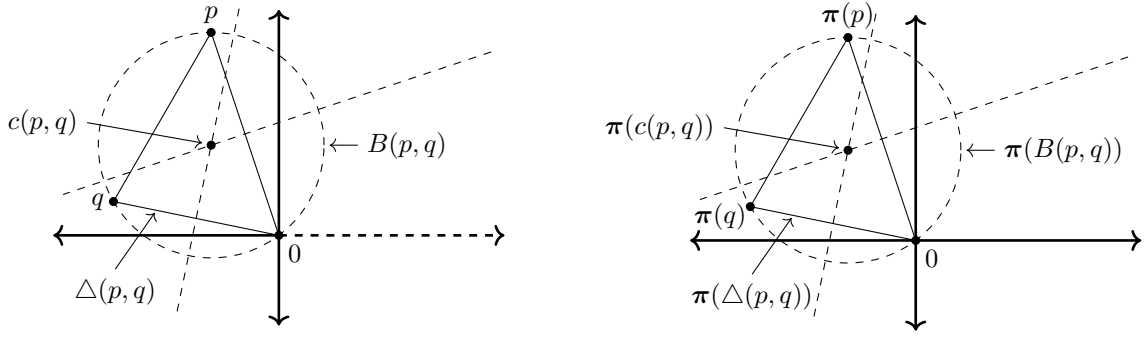

\begin{defn}\label{circumcenter_circumcircle_ball}
Let $p,q\in\mathcal{O}_i\backslash\{0\}$ be two points with distinct arguments in the same open $\pi$-sector of $\mathcal{O}_i$.  The \textit{triangle} determined by $p$ and $q$, denoted $\triangle(p,q)$, is the union of straight line segments from $0\in\mathcal{O}_i$ to $p$, from $p$ to $q$, and from $q$ to $0$.  Let 
\[c(p,q):=\partial H(p)\cap\partial H(q)\]
denote the \textit{circumcenter} determined by $p$ and $q$,
\[C(p,q):=\{z\in\mathcal{O}_i\ |\ d(z,c(p,q))=|c(p,q)|\}\]
denote the \textit{circumcircle} determined by $p$ and $q$, and
\[B(p,q):=\{z\in\mathcal{O}_i\ |\ d(z,c(p,q))<|c(p,q)|\}\]
denote the \textit{ball} determined by $p$ and $q$.
\end{defn}

Locally, and away from a non-removable singularity, the geometry on $\mathcal{O}_i$ is Euclidean, so the circumcenter, circumcircle and ball of Definition \mbox{\ref{circumcenter_circumcircle_ball}} may be viewed as isometric copies of their Euclidean namesakes in the plane.  To be more precise, Lemma 4 of \mbox{\cite{ESS}} implies that for any two points $p_1,p_2\in\mathcal{O}_i$ belonging to the same closed $\pi$-sector, the distance $d(p_1,p_2)$ between $p_1$ and $p_2$ in $\mathcal{O}_i$ equals the distance between their respective images ${\bm\pi}(p_1)$ and ${\bm\pi}(p_2)$ in the plane.  Hence on any such sector, ${\bm\pi}(\cdot)$ is an isometry.  Using Euclidean geometry, we find that ${\bm\pi}(c(p,q))$ is the center of the Euclidean circle ${\bm\pi}(C(p,q))$ of radius $|c(p,q)|$, and ${\bm\pi}(B(p,q))$ is the open ball whose boundary is ${\bm\pi}(C(p,q))$.  Moreover, ${\bm\pi}(C(p,q))$ is the circumcircle of the Euclidean triangle ${\bm\pi}(\triangle(p,q))$ with vertices $0,{\bm\pi}(p),{\bm\pi}(q)$, and $0,{\bm\pi}(p),{\bm\pi}(q)\in{\bm\pi}(C(p,q))$ implies $0,p,q\in C(p,q)=\partial B(p,q)$.

\begin{remark}
We use the same notation and terminology from Definition \mbox{\ref{circumcenter_circumcircle_ball}} for the analogous objects determined by two points in the Euclidean plane $\mathbb{C}$.
\end{remark}

\begin{defn}\label{P(O)}
Define the set of \textit{oppositely projected pairs} of points in $\mathcal{O}$ by
\[\mathbb{P}(\mathcal{O}):=\{\{p,p^-\}\subset\mathcal{O}\ |\ {\bm\pi}(p^-)=-{\bm\pi}(p)\neq 0\},\]
and for any subset $P\subset\mathbb{P}(\mathcal{O})$ of oppositely projected pairs, define the \textit{forgotten} version of $P$ by
\[P_F:=\bigcup_{\{p,p^-\}\in P}\{p,p^-\};\]
that is, the pairing inherent to $P$ is `forgotten' in the set $P_F\subset\mathcal{O}$.  For $r\in\mathbb{R}_+$ and $P\subset\mathbb{P}(\mathcal{O})$, let $rP:=\{\{rp,rp^-\}\ |\ \{p,p^-\}\in P\}$ and $rP_F:=(rP)_F$.
Call a subset $P\subset\mathbb{P}(\mathcal{O})$ \textit{limit-point free} if its forgotten version $P_F$ has no limit points in $\mathcal{O}$.
\end{defn}

Recall that ${\bm\pi}(\cdot)$ is a degree $2g-2+\kappa$ branched covering with a single branch point $0\in\mathbb{C}$.  Thus for any $p\in\mathcal{O}$ with ${\bm\pi}(p)\neq 0$, there are precisely $2g-2+\kappa$ points $p^-\in\mathcal{O}$ in the fiber above $-{\bm\pi}(p)$.  That is,
\[\big|\{p^-\ |\ \{p,p^-\}\in\mathbb{P}(\mathcal{O})\}\big|=2g-2+\kappa.\]
The subsets $P\subset\mathbb{P}(\mathcal{O})$ of interest to us do not have this multiplicity of elements:

\begin{defn}
Call a nonempty subset $P\subset\mathbb{P}(\mathcal{O})$ of oppositely projected pairs \textit{distinctive} if for any $p\in P_F$,
\[\big|\{p^-\ |\ \{p,p^-\}\in P\}\big|=1,\]
or, equivalently, for any $p\in P_F$, there is a unique point $p^-\in P_F$ satisfying ${\bm\pi}(p^-)=-{\bm\pi}(p)$.
\end{defn}

Now let $(p,q)\in\mathcal{O}^2$ be an ordered pair of two regular points with different arguments which belong to the same open $\pi$-sector of the same component of $\mathcal{O}$ (see the left-hand side of Figure \mbox{\ref{[p,q]_plot}}), and consider the difference ${\bm\pi}(q)-{\bm\pi}(p)$ of their images under ${\bm\pi(\cdot)}$ in $\mathbb{C}$.  As above, there are exactly $2g-2+\kappa$ points $u\in\mathcal{O}$ satisfying ${\bm\pi}(u)={\bm\pi}(q)-{\bm\pi}(p)$.  Each of the $2g-2+\kappa$ oppositely projected pairs $\{p,p^-\}\in\mathbb{P}(\mathcal{O})$ naturally announces a unique such point $u$, namely the $u$ lying nearest to $p^-$.  A distinctive set $P$ determines a unique point $p^-$ and hence a unique point $u$.  We fix special notation for this point $u$ in the following:

\begin{defn}\label{[p,q]}
Let $P,Q\subset\mathbb{P}(\mathcal{O})$ be distinctive subsets of oppositely projected pairs and $(\{p,p^-\},\{q,q^-\})\in P\times Q$.  If $p$ and $q$ belong to the same open $\pi$-sector of the same component of $\mathcal{O}$ and satisfy $\text{arg}(p)\neq\text{arg}(q)$, then we denote by $[p,q]$ the unique element of $\mathcal{O}$ for which
\begin{enumerate}
\item[(i)] ${\bm\pi}([p,q])={\bm\pi}(q)-{\bm\pi}(p)$, and

\item[(ii)] $[p,q]$ belongs to the same open $\pi$-sector of the same component of $\mathcal{O}$ as the point $p^-$.
\end{enumerate}
See Figure \mbox{\ref{[p,q]_plot}}.
\end{defn}

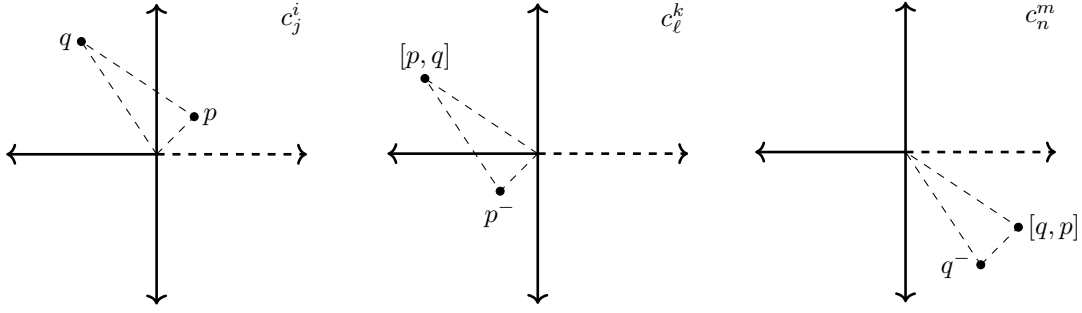
\begin{figure}[t]
\begin{minipage}{.3\textwidth}
\centering
\begin{tikzpicture}
\draw[->,line width=1,dashed] (0,0) -- (2,0);
\draw[->,line width=1] (0,0) -- (-2,0);
\draw[<->,line width=1] (0,-2) -- (0,2);
\draw[dashed] (0,0) -- (.5,.5) -- (-1,1.5) --cycle;
\filldraw (.5,.5) circle (1.5pt) node[right] {$p$};
\filldraw (-1,1.5) circle (1.5pt) node[left] {$q$};
\draw (1.8,1.8) circle (0pt) node {$c_j^i$};
\end{tikzpicture}
\end{minipage}
\begin{minipage}{.3\textwidth}
\centering
\begin{tikzpicture}
\draw[->,line width=1,dashed] (0,0) -- (2,0);
\draw[->,line width=1] (0,0) -- (-2,0);
\draw[<->,line width=1] (0,-2) -- (0,2);
\draw[dashed] (0,0) -- (-1-.5,1.5-.5) -- (-.5,-.5) -- cycle;
\filldraw (-.5,-.5) circle (1.5pt) node[below] {$p^-$};
\filldraw (-1-.5,1.5-.5) circle (1.5pt) node[above] {$[p,q]$};
\draw (1.8,1.8) circle (0pt) node {$c_\ell^k$};
\end{tikzpicture}
\end{minipage}
\begin{minipage}{.3\textwidth}
\centering
\begin{tikzpicture}
\draw[->,line width=1,dashed] (0,0) -- (2,0);
\draw[->,line width=1] (0,0) -- (-2,0);
\draw[<->,line width=1] (0,-2) -- (0,2);
\draw[dashed] (1,-1.5) -- (1.5,-1) -- (0,0) --cycle;
\filldraw (1,-1.5) circle (1.5pt) node[left] {$q^-$};
\filldraw (1.5,-1) circle (1.5pt) node[right] {$[q,p]$};
\draw (1.8,1.8) circle (0pt) node {$c_n^m$};
\end{tikzpicture}
\end{minipage}
\caption{The points $[p,q]$ and $[q,p]$, as determined by $\{p,p^-\}$ and $\{q,q^-\}$.}
\label{[p,q]_plot}
\end{figure}

\begin{remark}
The reader should note that the point $[p,q]$ crucially depends on the distinctive set $P$, as distinctiveness uniquely determines the point $p^-$ in Definition \mbox{\ref{[p,q]}}.  However, as $P$ shall be clear by context it is absent from the notation $[p,q]$.
\end{remark}

Notice from Definition \mbox{\ref{[p,q]}} that $[p,q]$ is defined if and only if $[q,p]$ is defined.  The following technical result is needed for Proposition \mbox{\ref{permute_permissible_triples}} below and is due to the fact that for distinctive sets, the points $p$ and $q$ naturally determine three isometric triangles in $\mathcal{O}$.  See again Figure \mbox{\ref{[p,q]_plot}}.

\begin{lem}\label{[p^bullet,[p,q]]}
Let $P,Q,U\subset\mathbb{P}(\mathcal{O})$ be distinctive subsets of oppositely projected pairs.  For any $(\{p,p^-\},\{q,q^-\})\in P\times Q$ for which $[p,q]\in\mathcal{O}$ is defined, we have that
$\{[p,q],[q,p]\}\in\mathbb{P}(\mathcal{O})$ is an oppositely projected pair.  Moreover, if $\{[p,q],[q,p]\}\in U$ then $[p^-,[p,q]]$ and $[[p,q],p^-]$ are defined and equal $q$ and $q^-$, respectively.
\end{lem}
\begin{proof}
Definition \ref{[p,q]} gives that 
\[{\bm\pi}([p,q])={\bm\pi}(q)-{\bm\pi}(p)=-{\bm\pi}([q,p]).\]
Since $p$ and $q$ belong to the same $\pi$-sector with $\text{arg}(p)\neq\text{arg}(q)$, we have ${\bm\pi}(q)-{\bm\pi}(p)\neq 0$, so $\{[p,q],[q,p]\}\in\mathbb{P}(\mathcal{O})$.

Condition (ii) of Definition \ref{[p,q]} guarantees that $p^-$ and $[p,q]$ belong to the same open $\pi$-sector of the same component.  Since $p$ and $q$ belong to the same open $\pi$-sector with different arguments, condition (i) guarantees that $\text{arg}(p^-)\neq\text{arg}([p,q])$, so $[p^-,[p,q]]$ and $[[p,q],p^-]$ are defined.  Next we show $[p^-,[p,q]]=q$.  Condition (ii) of Definition \ref{[p,q]} (applied to $[p^-,[p,q]]$) gives that $p$ and $[p^-,[p,q]]$ belong to the same open $\pi$-sector of the same component of $\mathcal{O}$; by assumption, the same is true of $p$ and $q$, so we find that $q$ and $[p^-,[p,q]]$ belong to the same open $2\pi$-sector.  As ${\bm\pi}(\cdot)$ is injective on open $2\pi$-sectors, it suffices to show that these latter two points are sent under this map to the same point in the plane.  We compute
\[{\bm\pi}([p^-,[p,q]])={\bm\pi}([p,q])-{\bm\pi}(p^-)={\bm\pi}(q)-{\bm\pi}(p)+{\bm\pi}(p)={\bm\pi}(q)\]
as desired.  The proof that $[[p,q],p^-]$ equals $q^-$ is similar.  
\end{proof}

The following definition and subsequent results will prove to be essential in \S\ref{Marked Voronoi staples determine permissible triples} and \S\ref{Finiteness of lattice surfaces with given Veech groups} below.

\begin{defn}\label{permissible_triples_defn}
Let $P,\ Q$ and $U$ be distinctive, limit-point free subsets of $\mathbb{P}(\mathcal{O})$ and $(\{p,p^-\},\{q,q^-\},\{u,u^-\})$ a triple of oppositely projected pairs in $P\times Q\times U$.  We call $(p,q,u)\in P_F\times Q_F\times U_F$ a \textit{permissible triple} if there exist \textit{permissible scalars} $(r,s,t)\in\mathbb{R}_+^3$ such that, for $(\{rp,rp^-\},\{sq,sq^-\},\{tu,tu^-\})\in rP\times sQ\times tU$,
\begin{enumerate}
\item[(i)] $[rp,sq]$ and $[sq,rp]$ are defined and equal $tu$ and $tu^-$, respectively, and

\item[(ii)] 
\[\big(B(rp,sq)\cup B(tu,rp^-)\cup B(sq^-,tu^-)\big)\cap \left(rP_F\cup sQ_F\cup tU_F\right)=\varnothing.\]
\end{enumerate}
(See Figure \mbox{\ref{permissible_triple_plot}}.)  For a collection ${\bf P}=\{P_i\}_{i\in I}$ of distinctive, limit-point free subsets of $\mathbb{P}(\mathcal{O})$, we let $\mathcal{P}({\bf P})$ denote the \textit{set of all permissible triples} arising from ${\bf P}$; that is,
\[\mathcal{P}({\bf P}):=\{(p_i,p_j,p_k)\in (P_i)_F\times (P_j)_F\times (P_k)_F\ |\ i,j,k\in I\ \text{and}\ (p_i,p_j,p_k)\ \text{is a permissible triple}\}.\]
\end{defn}

\begin{figure}[t]
\begin{minipage}{.3\textwidth}
\centering
\begin{tikzpicture}
\draw (1.8,1.8) circle (0pt) node {$c_j^i$};
\draw[->,line width=1,dashed] (0,0) -- (2.2,0);
\draw[->,line width=1] (0,0) -- (-2.2,0);
\draw[<->,line width=1] (0,-2.2) -- (0,2.2);
\draw[dashed] (0,1) circle (1);
\draw[->] (1.3,1) node[right] {\footnotesize $B(rp,sq)$} -- (.8,1);
\draw[dashed] (0,0) -- ({0+cos(315)},{1+sin(315)}) -- ({0+cos(120)},{1+sin(120)}) --cycle;
\filldraw (0,0) circle (2pt);
\filldraw[fill=cyan] ({0+cos(315)},{1+sin(315)}) circle (2pt) node[right] {\footnotesize $rp$};
\filldraw[fill=yellow] ({0+cos(120)},{1+sin(120)}) circle (2pt) node[left] {\footnotesize $sq$};
\filldraw[fill=red] (-1.8, .3) circle (2pt);
\filldraw[fill=red] (1.3, -.6) circle (2pt);
\filldraw[fill=red] (.8,1.8) circle (2pt);
\filldraw[fill=yellow] (1.6,.4) circle (2pt);
\filldraw[fill=yellow] (-.5,-1) circle (2pt);
\filldraw[fill=cyan] (-1.1,-1.3) circle (2pt);
\filldraw[fill=cyan] (1.7,-1.8) circle (2pt);
\filldraw[fill=cyan] (-1.3,1.2) circle (2pt);
\end{tikzpicture}
\end{minipage}
\begin{minipage}{.3\textwidth}
\centering
\begin{tikzpicture}
\draw (1.8,1.8) circle (0pt) node {$c_\ell^k$};
\draw[->,line width=1,dashed] (0,0) -- (2.2,0);
\draw[->,line width=1] (0,0) -- (-2.2,0);
\draw[<->,line width=1] (0,-2.2) -- (0,2.2);
\draw[dashed] ({-cos(315)},{-sin(315)}) circle (1);
\draw[->] (.6,.8) node[right] {\footnotesize $B(tu,rp^-)$} -- (.1,.8);
\draw[dashed] ({-cos(315)},{-1-sin(315)}) -- (0,0) -- ({cos(120)-cos(315)},{sin(120)-sin(315)}) --cycle;
\filldraw[fill=cyan] ({-cos(315)},{-1-sin(315)}) circle (2pt) node[below] {\footnotesize $rp^-$};
\filldraw (0,0) circle (2pt);
\filldraw[fill=red] ({cos(120)-cos(315)},{sin(120)-sin(315)}) circle (2pt) node[above left] {\footnotesize $tu$};
\filldraw[fill=yellow] (-1.8, .2) circle (2pt);
\filldraw[fill=red] (.5, -.4) circle (2pt);
\filldraw[fill=yellow] (1.5,1.3) circle (2pt);
\filldraw[fill=red] (-1.8,-.6) circle (2pt);
\filldraw[fill=cyan] (.6,.5) circle (2pt);
\filldraw[fill=yellow] (-1,-1.4) circle (2pt);
\filldraw[fill=red] (1.7,-1.5) circle (2pt);
\filldraw[fill=cyan] (1,-1.8) circle (2pt);
\end{tikzpicture}
\end{minipage}
\begin{minipage}{.3\textwidth}
\centering
\begin{tikzpicture}
\draw (1.8,1.8) circle (0pt) node {$c_n^m$};
\draw[->,line width=1,dashed] (0,0) -- (2.2,0);
\draw[->,line width=1] (0,0) -- (-2.2,0);
\draw[<->,line width=1] (0,-2.2) -- (0,2.2);
\draw[dashed] (0,0) arc (120:420:1);
\draw[->] (-.8,-.9) node[left] {\footnotesize $B(sq^-,tu^-)$} -- (-.3,-.9);
\draw[dashed] ({-cos(120)},{-1-sin(120)}) -- ({cos(315)-cos(120)},{sin(315)-sin(120)}) -- (0,0) --cycle;
\filldraw[fill=yellow] ({-cos(120)},{-1-sin(120)}) circle (2pt) node[below] {\footnotesize $sq^-$};
\filldraw[fill=red] ({cos(315)-cos(120)},{sin(315)-sin(120)}) circle (2pt) node[right] {\footnotesize $tu^-$};
\filldraw (0,0) circle (2pt);
\filldraw[fill=cyan] (-.9, .3) circle (2pt);
\filldraw[fill=cyan] (1.6, -.5) circle (2pt);
\filldraw[fill=yellow] (-.7,1.7) circle (2pt);
\filldraw[fill=yellow] (1.6,1.4) circle (2pt);
\filldraw[fill=red] (-1.4,-.1) circle (2pt);
\filldraw[fill=red] (.4,.8) circle (2pt);
\filldraw[fill=cyan] (-1,-1.9) circle (2pt);
\filldraw[fill=yellow] (-1.7,-1.5) circle (2pt);
\end{tikzpicture}
\end{minipage}
\caption{A permissible triple $(p,q,u)$ with permissible scalars $(r,s,t)$.  Points of $rP_F,\ sQ_F$ and $tU_F$ are in cyan, yellow, and red, respectively.  Condition (i) of Definition \ref{permissible_triples_defn} is met, as $[rp,sq]=tu$ and $[sq,rp]=tu^-$.  Condition (ii) is met since none of the open balls $B(rp,sq),\ B(tu,rp^-)$ or $B(sq^-,tu^-)$ contain any points of $rP_F,\ sQ_F$ or $tU_F$.}
\label{permissible_triple_plot}
\end{figure}
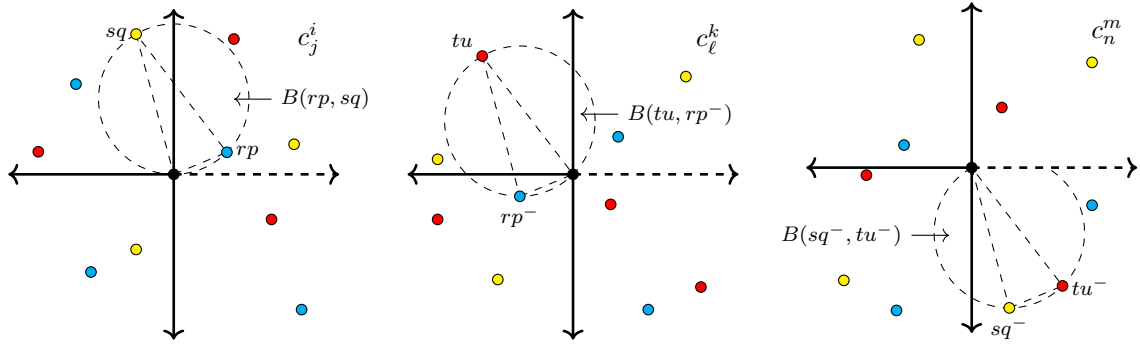

One immediate, but important, observation regarding permissible triples and permissible scalars is the following:

\begin{prop}\label{scalars_unique}
Let $(p,q,u)$ be a permissible triple with permissible scalars $(r,s,t)\in \mathbb{R}_+^3$.  The scalars $s$ and $t$ are uniquely determined by $p,q,u\in\mathcal{O}$ and $r\in\mathbb{R}_+$.  In particular, the set of all permissible scalars for $(p,q,u)$ is precisely the ray $\{(ar,as,at)\ |\ a\in\mathbb{R}_+\}$.
\end{prop}
\begin{proof}
By Definitions \ref{[p,q]} and \ref{permissible_triples_defn}, we have
\begin{equation}\label{tu=sq-rp}
t{\bm\pi}(u)={\bm\pi}([rp,sq])=s{\bm\pi}(q)-r{\bm\pi}(p).
\end{equation}
Rearranging, we find that $s$ and $t$ are solutions to
\[\begin{pmatrix}{\bm\pi}(q)_x & -{\bm\pi}(u)_x\\
{\bm\pi}(q)_y & -{\bm\pi}(u)_y\end{pmatrix}\begin{pmatrix} s\\ t\end{pmatrix}=\begin{pmatrix}r{\bm\pi}(p)_x\\ r{\bm\pi}(p)_y\end{pmatrix}.\]
The fact that $rp$ and $sq$ belong to the same open $\pi$-sector with different arguments implies that ${\bm\pi}(p)$ and ${\bm\pi}(q)$ are $\mathbb{R}$-linearly independent.  This, together with Equation (\ref{tu=sq-rp}), implies that also ${\bm\pi}(u)$ and ${\bm\pi}(q)$ are $\mathbb{R}$-linearly independent.  Hence the matrix on the left side of the previous equation is invertible, so $s$ and $t$ are uniquely determined.  

For the second statement, let $P_1, P_2\subset\mathbb{P}(O)$ be distinctive and $a\in\mathbb{R}_+$.  From Definition \mbox{\ref{[p,q]}}, one finds for $(\{p_1,p_1^-\},\{p_2,p_2^-\})\in P_1\times P_2$ and $(\{ap_1,ap_1^-\},\{ap_2,ap_2^-\})\in aP_1\times aP_2$ that if $[p_1,p_2]$ is defined then $[ap_1,ap_2]$ is defined and equals $a[p_1,p_2]$.  It follows then from Definition \mbox{\ref{permissible_triples_defn}} that if $(r,s,t)$ are permissible scalars for $(p,q,u)$, then so are $(a r,a s,a t)$ for any $a\in\mathbb{R}_+$.  On the other hand, if $(r_1,s_1,t_1)$ and $(r_2,s_2,t_2)$ are both permissible scalars for $(p,q,u)$, then so are $(1,s_1/r_1,t_1/r_1)$ and $(1,s_2/r_2,t_2/r_2)$.  The first statement of this proposition implies that $(r_1,s_1,t_1)$ and $(r_2,s_2,t_2)$ belong to a common ray.
\end{proof}

Note that the order of entries of a permissible triple $(p,q,u)$ is important in Definition \ref{permissible_triples_defn}; for instance, the first two entries $p$ and $q$ must necessarily belong to the same open $\pi$-sector of the same component of $\mathcal{O}$ for condition (i) to hold.  Nevertheless, this ordering does admit some flexibility.

\begin{prop}\label{permute_permissible_triples}
With notation as in Definition \ref{permissible_triples_defn}, the following statements are equivalent:
\begin{enumerate}
\item[(a)] $(p,q,u)$ is a permissible triple with permissible scalars $(r,s,t)$,
\item[(b)] $(q^-,u^-,p)$ is a permissible triple with permissible scalars $(s,t,r)$, and
\item[(c)] $(u,p^-,q^-)$ is a permissible triple with permissible scalars $(t,r,s)$.
\end{enumerate}
\end{prop}
\begin{proof}
We need only show $(a)$ implies $(b)$: the same argument will give $(b)$ implies $(c)$ and $(c)$ implies $(a)$.  Let $(p,q,u)$ be a permissible triple with permissible scalars $(r,s,t)$.  We claim that $(q^-,u^-,p)$ is a permissible triple with permissible scalars $(s,t,r)$.  
\begin{enumerate}
\item[(i)] We must show that $[sq^-,tu^-]$ and $[tu^-,sq^-]$ are defined and equal $rp$ and $rp^-$, respectively.  We have by assumption that $[sq,rp]$ is defined and equals $tu^-$.  By Lemma \mbox{\ref{[p^bullet,[p,q]]}}, both $[sq^-,[sq,rp]]=[sq^-,tu^-]$ and $[[sq,rp],sq^-]=[tu^-,sq^-]$ are defined.  By the same Lemma, these equal $rp$ and $rp^-$, respectively.
\item[(ii)] We have
\[\left(B(sq^-,tu^-)\cup B(rp,sq)\cup B(tu,rp^-)\right)\cap \left(sQ_F\cup tU_F\cup rP_F\right)=\varnothing\]
by assumption.
\end{enumerate}
\end{proof}

One might suspect from Definition \ref{permissible_triples_defn} that it is difficult for some $(p,q,u)\in P_F\times Q_F\times U_F$ to be a permissible triple.  Indeed, this suspicion is confirmed by the following:

\begin{lem}\label{permissible_triples_finite}
Let ${\bf P}=\{P_i\}_{i=1}^n$ be a finite collection of distinctive, limit-point free subsets of $\mathbb{P}(\mathcal{O})$.  Then the set $\mathcal{P}({\bf P})$ of all permissible triples arising from ${\bf P}$ is finite.
\end{lem}
\begin{proof}
It suffices to show that for any distinctive, limit-point free $P,Q,U\subset\mathbb{P}(\mathcal{O})$, the set of permissible triples $(p,q,u)\in P_F\times Q_F\times U_F$ is finite.  Suppose on the contrary that $\{(p_k,q_k,u_k)\}_{k\in\mathbb{N}}\subset P_F\times Q_F\times U_F$ is an infinite set of permissible triples with corresponding permissible scalars $\{(r_k,s_k,t_k)\}_{k\in\mathbb{N}}\subset\mathbb{R}_+^3$.  At least one of the sets $\{p_k\}_{k\in\mathbb{N}}, \{q_k\}_{k\in\mathbb{N}}$ or $\{u_k\}_{k\in\mathbb{N}}$ is infinite; by Proposition \ref{permute_permissible_triples}, we may assume that either $\{p_k\}_{k\in\mathbb{R}}$ or $\{p_k^-\}_{k\in\mathbb{N}}$ is infinite.  Since $P$ is distinctive, these sets have the same cardinality and hence are both infinite.  As the set of possible directions of points in each $\mathcal{O}_i$ is compact and there are only finitely many components $\mathcal{O}_i$ of $\mathcal{O}$, we may also assume that each $p_k$ (resp. $p_k^-$) belongs to some small sector of a fixed component of $\mathcal{O}$ and that $\lim_k\text{arg}({\bm\pi}(p_k))=\theta_P$ (resp. $\lim_k\text{arg}({\bm\pi}(p_k^-))=\theta_P^-$) exists.  Rotating each of $P, Q$ and $U$ by $-\theta_P$, assume for simplicity that $\theta_P=0$ (and hence $\theta_P^-=\pi$).

Rescaling each $(r_k,s_k,t_k)$ as necessary (see Proposition \mbox{\ref{scalars_unique})}, we further assume that $|r_kp_k|=1$ for all $k$.  Under these assumptions, we find that $\{r_k{\bm\pi}(p_k)\}_{k\in\mathbb{N}}\subset S^1$ with $r_k{\bm\pi}(p_k)\to 1\in\mathbb{C}$.  Also note that since $P$ is limit-point free, $|p_k|\to\infty$ and hence $r_k=1/|p_k|\to 0$.

Condition (ii) of Definition \ref{permissible_triples_defn} (together with the fact that ${\bm\pi}(\cdot)$ is an isometry on $\pi$-sectors) guarantees that in the plane, we have for each $k\in\mathbb{N}$ both
\begin{equation}\label{ball_int_seq_1}
B(r_k{\bm\pi}(p_k),s_k{\bm\pi}(q_k))\cap \{r_k{\bm\pi}(p_\ell)\}_{\ell\in\mathbb{N}}=\varnothing
\end{equation}
and
\[B(t_k{\bm\pi}(u_k),r_k{\bm\pi}(p_k^-))\cap \{r_k{\bm\pi}(p_\ell^-)\}_{\ell\in\mathbb{N}}=\varnothing.\]
This latter intersection may be rewritten
\begin{equation}\label{ball_int_seq_2}
B(t_k{\bm\pi}(u_k),-r_k{\bm\pi}(p_k))\cap \{-r_k{\bm\pi}(p_\ell)\}_{\ell\in\mathbb{N}}=\varnothing.
\end{equation}

For each $k$, let $c_k:=c(r_k{\bm\pi}(p_k),s_k{\bm\pi}(q_k))$ and $B_k:=B(r_k{\bm\pi}(p_k),s_k{\bm\pi}(q_k))$ be the circumcenter and ball determined by $r_k{\bm\pi}(p_k)$ and $s_k{\bm\pi}(q_k)$.  The circumcenter $c_k$ is the intersection of the perpendicular bisectors of the straight line segments from the origin to $r_k{\bm\pi}(p_k)$ and from the origin to $s_k{\bm\pi}(q_k)$.  Since $\text{arg}({\bm\pi}(p_k))\to 0$, for large $k$ the former perpendicular bisector does not intersect the negative real axis, so for all large $k$ we have $\text{arg}(c_k)\in [0,\pi)\cup(\pi,2\pi)$.  Passing to a subsequence, assume without loss of generality that $\text{arg}(c_k)\in[0,\pi)$ for all $k$.  

\begin{figure}[t]
\begin{tikzpicture}[scale=1.5]
\def\ptheta{10};
\def\s{.8};
\def\qtheta{130};
\def\btheta{\ptheta+58};
\draw[<->,line width=1] (-1.5,0) -- (1.8,0);
\draw[<->,line width=1] (0,-.7) -- (0,2);
\draw (1,-.07) -- (1,.07);
\draw (-1,-.07) -- (-1,.07);
\draw (-.07,1) -- (.07,1);
\filldraw (0,0) circle (1pt);
\filldraw ({cos(\ptheta)},{sin(\ptheta)}) circle (1pt) node[right] {$r_k{\bm\pi}(p_k)$};
\filldraw ({\s*cos(\qtheta)},{\s*sin(\qtheta)}) circle (1pt) node[left] {$s_k{\bm\pi}(q_k)$};
\draw[dashed] (0,0) -- ({cos(\ptheta)},{sin(\ptheta)}) -- ({\s*cos(\qtheta)},{\s*sin(\qtheta)}) -- cycle;
\filldraw ({-sin(\ptheta)*((cos(\ptheta)*cos(\qtheta)+sin(\ptheta)*sin(\qtheta)-\s)/(2*(sin(\ptheta)*cos(\qtheta)-cos(\ptheta)*sin(\qtheta))))+cos(\ptheta)/2},{cos(\ptheta)*((cos(\ptheta)*cos(\qtheta)+sin(\ptheta)*sin(\qtheta)-\s)/(2*(sin(\ptheta)*cos(\qtheta)-cos(\ptheta)*sin(\qtheta))))+sin(\ptheta)/2}) circle (1pt) node[left] {$c_k$};
\draw[dashed] ({-sin(\ptheta)*((cos(\ptheta)*cos(\qtheta)+sin(\ptheta)*sin(\qtheta)-\s)/(2*(sin(\ptheta)*cos(\qtheta)-cos(\ptheta)*sin(\qtheta))))+cos(\ptheta)/2},{cos(\ptheta)*((cos(\ptheta)*cos(\qtheta)+sin(\ptheta)*sin(\qtheta)-\s)/(2*(sin(\ptheta)*cos(\qtheta)-cos(\ptheta)*sin(\qtheta))))+sin(\ptheta)/2}) circle ({sqrt((-sin(\ptheta)*((cos(\ptheta)*cos(\qtheta)+sin(\ptheta)*sin(\qtheta)-\s)/(2*(sin(\ptheta)*cos(\qtheta)-cos(\ptheta)*sin(\qtheta))))+cos(\ptheta)/2)^2+(cos(\ptheta)*((cos(\ptheta)*cos(\qtheta)+sin(\ptheta)*sin(\qtheta)-\s)/(2*(sin(\ptheta)*cos(\qtheta)-cos(\ptheta)*sin(\qtheta))))+sin(\ptheta)/2)^2)});
\draw[dashed] (1/2,-.7) -- (1/2,2);
\filldraw (1/2,{sqrt(sqrt((-sin(\ptheta)*((cos(\ptheta)*cos(\qtheta)+sin(\ptheta)*sin(\qtheta)-\s)/(2*(sin(\ptheta)*cos(\qtheta)-cos(\ptheta)*sin(\qtheta))))+cos(\ptheta)/2)^2+(cos(\ptheta)*((cos(\ptheta)*cos(\qtheta)+sin(\ptheta)*sin(\qtheta)-\s)/(2*(sin(\ptheta)*cos(\qtheta)-cos(\ptheta)*sin(\qtheta))))+sin(\ptheta)/2)^2)^2-(1/2-(-sin(\ptheta)*((cos(\ptheta)*cos(\qtheta)+sin(\ptheta)*sin(\qtheta)-\s)/(2*(sin(\ptheta)*cos(\qtheta)-cos(\ptheta)*sin(\qtheta))))+cos(\ptheta)/2))^2)+cos(\ptheta)*((cos(\ptheta)*cos(\qtheta)+sin(\ptheta)*sin(\qtheta)-\s)/(2*(sin(\ptheta)*cos(\qtheta)-cos(\ptheta)*sin(\qtheta))))+sin(\ptheta)/2}) circle (1pt) node[above right] {$b_k^+$};
\filldraw (1/2,{1/2*tan(\btheta)}) circle (1pt) node[right] {$b$};
\filldraw[color=black!40,opacity=.5] (0,0) -- (1/2,{1/2*tan(\ptheta)}) arc (\ptheta:\btheta:1/2) -- cycle;
\filldraw ({1/3*cos((\ptheta+\btheta)/2)},{1/3*sin((\ptheta+\btheta)/2)}) circle (1pt);
\draw[->] ({1/3*cos((\ptheta+\btheta)/2)-.7},{1/3*sin((\ptheta+\btheta)/2)}) node[left] {$r_k{\bm\pi}(p_N)$} -- ({1/3*cos((\ptheta+\btheta)/2)-.1},{1/3*sin((\ptheta+\btheta)/2)});
\end{tikzpicture}
\caption{Illustration of the proof of Lemma \ref{permissible_triples_finite}.  The vertical dashed line represents $x=1/2$, and the shaded region is the subset $S_k$ of the ball $B_k=B(r_k{\bm\pi}(p_k),s_k{\bm\pi}(q_k))$.}
\label{permissible_triples_finite_fig}
\end{figure}
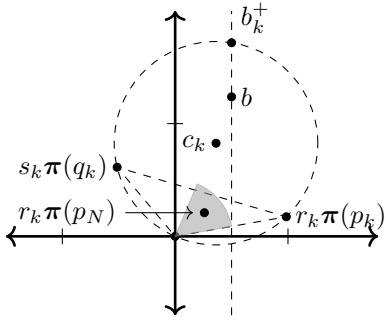

We consider two cases (see Figure \ref{permissible_triples_finite_fig}):
\begin{enumerate}
\item[(i)]  Suppose there is a subsequence for which ${\bm\pi}(p_k)_y\ge0$ for all $k$.  Since $\text{arg}({\bm\pi}(p_k))\to 0$, we may pass to a subsequence to assume $\text{arg}({\bm\pi}(p_{k+1}))\le\text{arg}({\bm\pi}(p_k))$ for all $k$.  For large enough $k$, the $x$-coordinate of $r_k{\bm\pi}(p_k)\in S^1$ is greater than $1/2$.  Since both $0$ and $r_k{\bm\pi}(p_k)$ belong to $C_k:=\partial B_k$, we find that the vertical line $x=1/2$ intersects $C_k$ at two distinct points; let $b_k^+$ denote the point of intersection on the upper-semicircle of $C_k$.  Note that $b_y:=\text{inf}_k\{(b_k^+)_y\}>0$; otherwise $\text{arg}(c_k)\in(\pi,2\pi)$ contrary to our assumption.  Set $b:=(1/2,b_y)$ and note that for large enough $k$, we have $\text{arg}({\bm\pi}(p_k))<\text{arg}(b)$.  Let 
\[S_k=\{p\in\mathbb{C}\ |\ |p|<1/2,\ \text{arg}(p)\in(\text{arg}({\bm\pi}(p_k)),\text{arg}(b))\}.\]
Note that $S_k$ is contained in the triangle with vertices $0,\ r_k{\bm\pi}(p_k)$ and $b_k^+$, which is contained in $B_k$; hence $S_k\subset B_k$ for each $k$.  Choose some $N\in\mathbb{N}$ large enough that $\text{arg}({\bm\pi}(p_N))\in[0,\text{arg}(b))$.  For each $k>N$, we have $\text{arg}({\bm\pi}(p_k))\le\text{arg}({\bm\pi}(p_N))$ by assumption, so $\text{arg}({\bm\pi}(p_N))\in[\text{arg}({\bm\pi}(p_k)),\text{arg}(b))$.  Taking $k$ large enough, we also have $|r_k{\bm\pi}(p_N)|=|{\bm\pi}(p_N)|/|{\bm\pi}(p_k)|<1/2$.  Hence $r_k{\bm\pi}(p_N)\in S_k\subset B_k=B(r_k{\bm\pi}(p_k),s_k{\bm\pi}(q_k))$, contradicting Equation \ref{ball_int_seq_1}.

\item[(ii)]  Suppose there is no subsequence for which ${\bm\pi}(p_k)_y\ge0$ for all $k$.  Then there is some subsequence for which ${\bm\pi}(p_k)_y<0$---and hence $-{\bm\pi}(p_k)_y>0$---for all $k$.  Notice that $c(t_k{\bm\pi}(u_k),-r_k{\bm\pi}(p_k))=c_k-r_k{\bm\pi}(p_k)$, so the argument of this circumcenter belongs to $(0,\pi)$.  An analogous proof (reflecting about the $x$-axis) to that of case (i) implies that for some fixed $N$ and large enough $k$, $-r_k{\bm\pi}(p_N)\in B(t_k{\bm\pi}(u_k),-r_k{\bm\pi}(p_k))$, which contradicts Equation \ref{ball_int_seq_2}.
\end{enumerate}

We conclude that the set of permissible triples must be finite.   
\end{proof}

\subsection{Marked segments and Voronoi staples}\label{Marked segments and Voronoi staples}

In this subsection we fix a stratum $\mathcal{H}(d_1,\dots,d_\kappa),$ its corresponding canonical surface $\mathcal{O}=\mathcal{O}(d_1,\dots,d_\kappa)$ and a translation surface $(X,\omega)\in\mathcal{H}(d_1,\dots,d_\kappa)$.  Label the singularities of $(X,\omega)$ as $\sigma_1,\dots,\sigma_\kappa$ so that $\sigma_i$ has cone angle $2\pi(d_i+1)$.  In a sufficiently small neighborhood of $\sigma_i$ define generalized polar coordinates in a fashion analogous to that on $\mathcal{O}_i$ (see \S\ref{Canonical surface} and the proof of Lemma 5 of \cite{ESS}).  

Let $s$ be a separatrix emanating from $\sigma_i$ of length $|s|>0$ and corresponding angle $\theta\in[0,2\pi(d_i+1))$, and denote by $\hat{s}$ the point of $\mathcal{O}_i$ with polar coordinates $(|s|,\theta)$.  If $|s|=0$, set $\hat{s}:=0\in\mathcal{O}_i$.  If $s$ is in fact a saddle connection, we call $\hat{s}$ the \textit{marked segment} determined by $s$; note that in this case, $\text{hol}(s)={\bm\pi}(\hat{s})\in\mathbb{C}$.  Denote by $s'$ the identical, but oppositely-oriented saddle connection to $s$.  Let $\mathcal{M}(X,\omega)$ be the set of all pairs $\{s,s'\}$ of oppositely-oriented saddle connections of $(X,\omega)$ and 
\[\mathcal{M}_F(X,\omega):=\bigcup_{\{s,s'\}\in\mathcal{M}(X,\omega)}\{s,s'\}\]
the set of all oriented saddle connections on $(X,\omega)$.  Let $\widehat{\mathcal{M}}(X,\omega)$ denote the set of all \textit{orientation-paired marked segments} $\{\hat{s},\hat{s}'\}$ determined by oppositely oriented saddle connections $\{s,s'\}\in\mathcal{M}(X,\omega)$. 

\begin{prop}\label{marked_segs_distinctive_and_lpf}
The set $\widehat{\mathcal{M}}(X,\omega)$ is a limit-point free subset of the set $\mathbb{P}(\mathcal{O})$ of oppositely projected pairs.  Moreover, if $\hat{s}$ and $\hat{t}$ are marked segments with identical arguments belonging to the same component of $\mathcal{O}$, then in fact $\{\hat{s},\hat{s}'\}=\{\hat{t},\hat{t}'\}$.  In particular, $\widehat{\mathcal{M}}(X,\omega)$ is distinctive.
\end{prop}
\begin{proof}
Suppose $\{\hat{s},\hat{s}'\}\in\widehat{\mathcal{M}}(X,\omega)$.  Since the saddle connections $s$ and $s'$ are identical but oppositely-oriented, it is clear that ${\bm\pi}(\hat{s}')=-{\bm\pi}(\hat{s})$, so $\{\hat{s},\hat{s}'\}\in\mathbb{P}(\mathcal{O})$ and $\widehat{\mathcal{M}}(X,\omega)\subset\mathbb{P}(\mathcal{O})$.  The image of $(\widehat{\mathcal{M}}(X,\omega))_F$ under ${\bm\pi}(\cdot)$ equals the image of $\mathcal{M}_F(X,\omega)$ under $\text{hol}(\cdot)$.  Since the set of holonomy vectors has no limit points (\S\ref{Metric, Lebesgue measure, saddle connections and holonomy vectors}) and ${\bm\pi}(\cdot)$ is a homeomorphism on sufficiently small neighborhoods of regular points, we have that $(\widehat{\mathcal{M}}(X,\omega))_F$ has no limit points and hence $\widehat{\mathcal{M}}(X,\omega)$ is limit-point free.

Now suppose $\hat{s}$ and $\hat{t}$ are marked segments with identical arguments in the same component $\mathcal{O}_i\subset\mathcal{O}$.  Then the underlying saddle connections $s$ and $t$ both emanate in the same direction from $\sigma_i\in\Sigma$.  If their lengths differ, say $|\hat{s}|<|\hat{t}|$, then $|s|<|t|$. This implies that $t$ has a singularity in its interior, which is impossible.  If $|\hat{s}|=|\hat{t}|$, then in fact $s=t$, and so $\{\hat{s},\hat{s}'\}=\{\hat{t},\hat{t}'\}$.  It follows that $\widehat{\mathcal{M}}(X,\omega)$ is distinctive.
\end{proof}

Let $\widehat{\mathcal{M}}_F(X,\omega):=(\widehat{\mathcal{M}}(X,\omega))_F$ denote the set of all marked segments determined by saddle connections on $(X,\omega)$.  The \textit{star domain} for $\sigma_i\in\Sigma$ is
\[\text{star}_i(X,\omega):=\{\hat{s}\ |\ s\ \text{is a separatrix on}\ (X,\omega)\ \text{emanating from $\sigma_i$}\}\subset\mathcal{O}_i.\]
Note that $\text{star}_i(X,\omega)$ consists of the union of closed rays emanating from $0\in\mathcal{O}_i$ which stop only when meeting a marked segment (and thus almost every such ray is infinite).  The \textit{star domain} for $(X,\omega)$ is
\[\text{star}(X,\omega)=\bigsqcup_{i=1}^\kappa \text{star}_i(X,\omega)\subset\mathcal{O}.\]

Define a map $\eta:\text{star}(X,\omega)\to(X,\omega),$ where for each point $p\in\text{star}_i(X,\omega)$, if $s$ is the separatrix for which $p=\hat{s}$, then $\eta(p)$ is the endpoint of $s$ on $(X,\omega)$.  In other words, if $p\in\text{star}(X,\omega)$ has polar coordinates $(|p|,\theta)$, then $\eta(p)$ is the endpoint of the separatrix of length $|p|$ emanating from $\sigma_i\in\Sigma$ with angle $\theta$.  

For each $x$ in the Voronoi 2-cell $\mathcal{C}_i:=\mathcal{C}_{\sigma_i}$, let $s_x$ be the unique length-minimizing separatrix from $\sigma_i$ to $x$.  Note that $\eta$ is injective---and thus invertible---on the set
\[\{\hat{s}_x\ |\ x\in\mathcal{C}_i\}\subset\text{star}_i(X,\omega):\]
if $\eta(\hat{s}_x)=\eta(\hat{s}_y)$ for $x,y\in\mathcal{C}_i$, then $x=y$ and $s_x=s_y$ by the aforementioned uniqueness of these separatrices.  Hence $\hat{s}_x=\hat{s}_y$.  Let $\iota_i:\mathcal{C}_i\to\{\hat{s}_x\ |\ x\in\mathcal{C}_i\}$ denote the corresponding inverse, namely $x\mapsto\hat{s}_x$ for each $x\in\mathcal{C}_i$, and define $\iota:\bigsqcup_{i=1}^\kappa \mathcal{C}_i\to\mathcal{O}$ by setting $\iota|_{\mathcal{C}_i}=\iota_i$ for each $i$;  see Figure \ref{voronoi_cells_figure}.

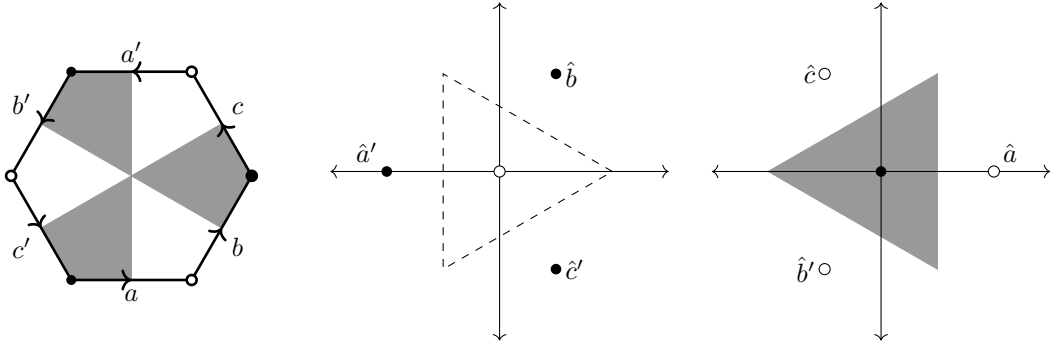
\begin{figure}[t]
\begin{minipage}{.28\textwidth}
\centering
\begin{tikzpicture}
\tikzmath{\scale=.8;}
\filldraw [color=black!40] (\scale*1,0) -- (\scale*1,\scale*1.732) -- (-\scale*.5,\scale*.5*1.732) -- (0,0) -- cycle;
\filldraw [color=black!40] (\scale*1,\scale*2*1.732) -- (\scale*1,\scale*1.732) -- (-\scale*.5,\scale*1.5*1.732) -- (0,\scale*2*1.732) -- cycle;
\filldraw [color=black!40] (\scale*2.5,\scale*.5*1.732) -- (\scale*1,\scale*1.732) -- (\scale*2.5,\scale*1.5*1.732) -- (\scale*3,\scale*1.732) -- cycle;
\draw [->-=.5,line width=1] (-\scale*1,\scale*1.732) -- (0,0) node[midway,below left] {$c'$} node[scale=.4,circle,fill=black] {};
\draw [->-=.5,line width=1] (0,\scale*3.464) -- (-\scale*1,\scale*1.732) node[midway,above left] {$b'$} node[draw,scale=.4,circle,fill=white] {};
\draw [->-=.5,line width=1] (\scale*2,\scale*3.464) -- (\scale*0,\scale*3.464) node[midway,above] {$a'$} node[scale=.4,circle,fill=black] {};
\draw [->-=.5,line width=1] (\scale*3,\scale*1.732) -- (\scale*2,\scale*3.464) node[midway,above right] {$c$} node[draw,scale=.4,circle,fill=white] {};
\draw [->-=.5,line width=1] (\scale*2,0) -- (\scale*3,\scale*1.732) node[midway,below right] {$b$} node[draw,scale=.4,circle,fill=black] {};
\draw [->-=.5,line width=1] (0,0) -- (\scale*2,0) node[midway,below] {$a$} node[draw,scale=.4,circle,fill=white] {};
\end{tikzpicture}
\end{minipage}
\begin{minipage}{.3\textwidth}
\centering
\begin{tikzpicture}[scale=1.5]
\draw[<->] (-1.5,0) -- (1.5,0);
\draw[<->] (0,-1.5) -- (0,1.5);
\draw [dashed] (1,0) -- (-1/2,.866025) -- (-1/2,-.866025) -- cycle;
\filldraw [fill=white] (0,0) circle (1.4pt);
\filldraw [fill=black] (1/2,.866025) circle (1.2pt) node[right] {$\hat{b}$};
\filldraw [fill=black] (-1,0) circle (1.2pt) node[above left] {$\hat{a}'$};
\filldraw [fill=black] (1/2,-.866025) circle (1.2pt) node[right] {$\hat{c}'$};
\end{tikzpicture}
\end{minipage}
\begin{minipage}{.3\textwidth}
\centering
\begin{tikzpicture}[scale=1.5]
\filldraw [color=black!40] (-1,0) -- (1/2,.866025) -- (1/2,-.866025) -- cycle;
\draw[<->] (-1.5,0) -- (1.5,0);
\draw[<->] (0,-1.5) -- (0,1.5);
\filldraw [fill=black] (0,0) circle (1.2pt);
\filldraw [fill=white] (-1/2,.866025) circle (1.4pt) node[left] {$\hat{c}$};
\filldraw [fill=white] (1,0) circle (1.4pt) node[above right] {$\hat{a}$};
\filldraw [fill=white] (-1/2,-.866025) circle (1.4pt) node[left] {$\hat{b}'$};
\end{tikzpicture}
\end{minipage}
\caption{Left: Voronoi decomposition of a translation surface subordinate to two removable singularities, $\sigma_1$ (white) and $\sigma_2$ (black).  The 2-cell $\mathcal{C}_1:=\mathcal{C}_{\sigma_1}$ (resp. $\mathcal{C}_2:=\mathcal{C}_{\sigma_2}$) is the open region in white (resp. gray).  Middle: The image $\iota(\mathcal{C}_1)$ in $\mathcal{O}_1$, along with three marked segments.  Right: The image $\iota(\mathcal{C}_2)$ in $\mathcal{O}_2$, along with three marked segments.}
\label{voronoi_cells_figure}
\end{figure}

The translation surface $(X,\omega)$ is isometric to the quotient space of $\bigsqcup_{i=1}^\kappa\overline{\mathcal{C}_i}$ under the equivalence relation defined by identifying shared edges of Voronoi 2-cells.  Proposition 7 of \cite{ESS} shows that in a similar fashion, $(X,\omega)$ may be recovered from the closure of the image under $\iota$ of its Voronoi 2-cells by identifying appropriate edges of the various $\overline{\iota(\mathcal{C}_i)}$.  We provide a brief overview of the method by which these edge identifications are made.  Recall that the half-space $H(p)$ determined by $p\in\mathcal{O}_i$ is convex, and thus so is any intersection of such half-spaces (see Definition \mbox{\ref{half-space}}).

\begin{defn}\label{convex_body}
For $S$ a subset of $\mathcal{O}$ with no limit points, the $\textit{convex body}$ of $\mathcal{O}_i$ subordinate to $S$ is defined by
\[\Omega_i(S):=\bigcap_{p\in S\cap\mathcal{O}_i}H(p).\]
The set of \textit{essential points} of $\Omega_i(S)$ is the (unique) minimal subset $\mathcal{E}_i(S)\subset S$ for which 
\[\Omega_i(S)=\bigcap_{p\in \mathcal{E}_i(S)}H(p).\]
When $S=\widehat{\mathcal{M}}_F(X,\omega)$ is the set of all marked segments of $(X,\omega)$, we use the suppressed notation $\Omega_i:=\Omega_i(\widehat{\mathcal{M}}_F(X,\omega))$ and $\mathcal{E}_i:=\mathcal{E}_i(\widehat{\mathcal{M}}_F(X,\omega))$.  
Call distinct points $p,q\in S$ \textit{adjacent} within $S$ if $p,q\neq 0$ belong to the same component $\mathcal{O}_i$, and either of the two open sectors centered at $0$ between $p$ and $q$ contains no points of $S$.
\end{defn}

Proposition 14 of \cite{ESS} shows that the convex body $\Omega_i$ is precisely the set $\overline{\iota(\mathcal{C}_i)}$.  Furthermore, Proposition 14 and Definition 15 of \cite{ESS} show that there is a subset $\widehat{\mathcal{S}}(X,\omega)$ of $\widehat{\mathcal{M}}(X,\omega)$ for which the union of essential points $\mathcal{E}:=\sqcup_{i=1}^\kappa\mathcal{E}_i$ equals $\widehat{\mathcal{S}}_F(X,\omega):=(\widehat{\mathcal{S}}(X,\omega))_F$, i.e. the essential points come equipped with a natural pairing.  Elements $\{\hat{s},\hat{s}'\}$ of $\widehat{\mathcal{S}}(X,\omega)$ are called \textit{marked Voronoi staples} and their underlying pairs of saddle connections $\{s,s'\}$ in $\mathcal{M}(X,\omega)$ are called \textit{Voronoi staples}.  It follows from the definition of $\mathcal{E}_i$ that each edge on the boundary of $\Omega_i=\overline{\iota(\mathcal{C}_i)}$ belongs to the boundary of a half-space $H(\hat{s})$ for some $\hat{s}\in\mathcal{E}_i\subset\widehat{\mathcal{S}}_F(X,\omega)$; conversely, the boundary of the half-space determined by each $\hat{s}\in\widehat{\mathcal{S}}_F(X,\omega)$ contains an edge on the boundary of some $\Omega_i=\overline{\iota(\mathcal{C}_i)}$.  Propositions 7 and 14 of \cite{ESS} show that the edges of the convex bodies of $\sqcup_{i=1}^\kappa \Omega_i$ corresponding to $\hat{s}$ and its orientation-paired $\hat{s}'$ are equal length, and that $(X,\omega)$ is isometric to the quotient space of $\sqcup_{i=1}^\kappa \Omega_i$ under the equivalence relation given by identifying these edges via translation.

Of import for our purposes is the observation that a translation surface is uniquely determined by its marked Voronoi staples $\widehat{\mathcal{S}}(X,\omega)$, and that $\widehat{\mathcal{S}}_F(X,\omega)$ is precisely the union of essential points $\mathcal{E}=\sqcup_{i=1}^\kappa\mathcal{E}_i$.

\begin{remark}
Under this identification of edges, the vertices of the various $\Omega_i$ are regarded as regular points on the resulting translation surface (they correspond to the Voronoi 0-cells of $(X,\omega)$).  The origins $0\in\Omega_i\subset\mathcal{O}_i$ become singularities of the resulting translation surface; we require this also if $\mathcal{O}_i=(\mathbb{C},dz)$, in which case the singularity is a marked point.
\end{remark}

\begin{eg}
In Figure \ref{voronoi_cells_figure}, the marked Voronoi staples are $\{\hat{a},\hat{a}'\},\ \{\hat{b},\hat{b}'\}$ and $\{\hat{c},\hat{c}'\}$.  The hexagonal translation surface on the left is recovered by identifying the edges of the convex bodies corresponding to these respective pairs.  
\end{eg}

We conclude this subsection a brief observation:

\begin{prop}\label{adjacent_vs_in_same_pi_sector}
Any two adjacent elements $\hat{s}$ and $\hat{t}$ of $\widehat{\mathcal{S}}_F(X,\omega)$ belong to the same open $\pi$-sector.
\end{prop}
\begin{proof}
If not, then $\partial H(\hat{s})\cap\partial H(\hat{t})=\varnothing$, and the corresponding convex body has infinite area.  This would imply that the closed translation surface $(X,\omega)$ has infinite area, which is a contradiction.  
\end{proof}

\section{Fanning groups, simulations and finiteness of lattices in strata}\label{Fanning groups, simulations and finiteness of lattices in strata}

\subsection{Fanning groups and lattices}\label{Fanning groups and lattices}

In this subsection we consider a set of directions $\Theta_\Gamma\subset S^1$ announced by a discrete subset $\Gamma\subset\mathrm{SL}_2\mathbb{R}$ and give an equivalent definition of lattice groups in terms of this set of directions (Lemma \ref{lattice_and_fanning}).  Abusing notation, we write $\Gamma$ for both a discrete subset of $\mathrm{SL}_2\mathbb{R}$ whose elements act as linear transformations on $\mathbb{R}^2$ and for its image as a discrete subset of $\text{PSL}_2\mathbb{R}$ whose elements act as M\"obius transformations on the closure of the upper-half plane $\mathbb{H}$.  Similarly, we denote elements of $\Gamma$ in both of these settings by the same notation $A\in\Gamma$.  When there is risk of confusion, we shall explicitly mention the setting in which to consider $\Gamma$.

\begin{defn}\label{fanning_def}
Let $\Gamma$ be a discrete subset of $\mathrm{SL}_2\mathbb{R}$ and $\Delta$ the open unit disk.  Set
\[\Theta_\Gamma:=S^1\backslash\bigcup_{A\in\Gamma} A\cdot\Delta,\]
where $A\cdot\Delta$ denotes the usual action of a matrix on $\mathbb{R}^2$ as a linear transformation.  A group $\Gamma$ for which $\Theta_\Gamma$ is finite is called a \textit{fanning group}.  We call $(X,\omega)$ a \textit{fanning surface} if its Veech group $\mathrm{SL}(X,\omega)$ is fanning. 
\end{defn}

As we shall see in Corollary \ref{directions_for_orbits} below, $\Theta_\Gamma$ contains the directions of certain short marked segments of any translation surface with Veech group $\Gamma$.  These directions will be instrumental in our construction of such surfaces.  

\begin{eg}\label{Gamma_0_eg}
Let $S:=(\begin{smallmatrix}0 & -1\\ 1 & 0\end{smallmatrix}),\ T:=(\begin{smallmatrix}1 & 1\\ 0 & 1\end{smallmatrix})$, and set $\Gamma:=\langle S,T^2\rangle$.  Figure \ref{fanning_group_Gamma_0} shows that $\Theta_{\Gamma}$ is contained in the set of eighth roots of unity and thus $\Gamma$ a fanning group.  One verifies that in fact $\Theta_{\Gamma}$ equals this set.
\end{eg}

\begin{figure}[t]
\centering
\begin{tikzpicture}[scale=1.6]
\draw[fill=black!30,fill opacity=.5,rotate around={22.5:(0,0)},dashed] (0,0) ellipse ({sqrt(2*sqrt(2)+3)} and {sqrt(3-2*sqrt(2))});
\draw[fill=black!30,fill opacity=.5,rotate around={-22.5:(0,0)},dashed] (0,0) ellipse ({sqrt(2*sqrt(2)+3)} and {sqrt(3-2*sqrt(2))});
\draw[fill=black!30,fill opacity=.5,rotate around={112.5:(0,0)},dashed] (0,0) ellipse ({sqrt(2*sqrt(2)+3)} and {sqrt(3-2*sqrt(2))});
\draw[fill=black!30,fill opacity=.5,rotate around={67.5:(0,0)},dashed] (0,0) ellipse ({sqrt(2*sqrt(2)+3)} and {sqrt(3-2*sqrt(2))});
\draw[<->,line width=1] (-2.5,0) -- (2.5,0);
\draw[<->,line width=1] (0,-2.5) -- (0,2.5);
\draw (0,0) circle (1);
\draw (1.3,.5) circle (0pt) node {\tiny $T^2\cdot\Delta$};
\draw (1.3,-.5) circle (0pt) node {\tiny $T^{-2}\cdot\Delta$};
\draw (-.45,1.2) circle (0pt) node {\tiny $ST^{2}\cdot\Delta$};
\draw (.45,1.2) circle (0pt) node {\tiny $ST^{-2}\cdot\Delta$};
\filldraw (1,0) circle (1pt);
\filldraw ({sqrt(2)/2},{sqrt(2)/2}) circle (1pt);
\filldraw (0,1) circle (1pt);
\filldraw ({-sqrt(2)/2},{sqrt(2)/2}) circle (1pt);
\filldraw (-1,0) circle (1pt);
\filldraw ({-sqrt(2)/2},{-sqrt(2)/2}) circle (1pt);
\filldraw (0,-1) circle (1pt);
\filldraw ({sqrt(2)/2},{-sqrt(2)/2}) circle (1pt);
\end{tikzpicture}
\caption{The set of eighth roots of unity contains $\Theta_{\Gamma}$ with $\Gamma=\langle S,T^2\rangle$ and $S$ and $T$ defined as in Example \ref{Gamma_0_eg}.}
\label{fanning_group_Gamma_0}
\end{figure}
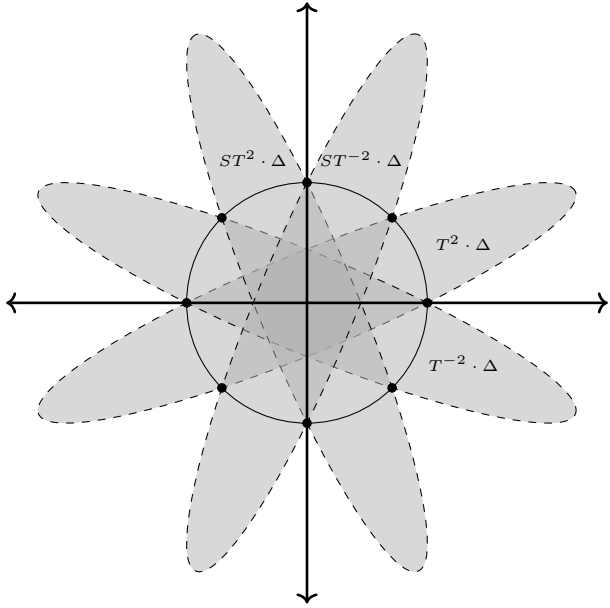

We introduce notation to help us relate lattices to fanning groups.  For $A\in\text{PSL}_2\mathbb{R}$, set
\[H(A):=\{z\in\mathbb{H}\ |\ d_\mathbb{H}(i,z)\le d_\mathbb{H}(Ai,z)\},\]
where $d_\mathbb{H}$ denotes the hyperbolic metric on $\mathbb{H}$, and for a discrete subset $\Gamma\subset\mathrm{PSL}_2\mathbb{R}$, define
\[D(\Gamma):=\bigcap_{A\in\Gamma} H(A)\]
(compare with Definitions \ref{half-space} and \ref{convex_body}).  Recall that if $\Gamma$ is a Fuchsian group which trivially stabilizes $i\in\mathbb{H}$---or, equivalently, $\Gamma\cap\text{PSO}_2\mathbb{R}=\{\text{Id}\}$---then $D(\Gamma)$ is the Dirichlet region centered at $i$ for $\Gamma$.  In this case $D(\Gamma)$ is a convex fundamental polygon for $\Gamma$; see, say, \cite{Beardon} and \cite{Katok}.  For any subset $S\subset\mathbb{H}$, let $\overline{S}$ and $\partial S$ denote the closure and boundary, respectively, of $S$ in $\hat{\mathbb{C}}=\mathbb{C}\cup\{\infty\}$.  Define
\[S(A):=\{x/y\ |\ (x,y)^T\in S^1\backslash A\cdot\Delta\}\]
and 
\[S_\Gamma:=\bigcap_{A\in\Gamma}S(A)=\{x/y\ |\ (x,y)^T\in\Theta_\Gamma\},\]
so that $S(A)$ and $S_\Gamma$ belong to $\partial\mathbb{H}=\mathbb{R}\cup\{\infty\}$.  The following proposition states that the reciprocals of the slopes of directions in $\Theta_\Gamma\subset S^1$ coincide with the points of $\overline{D(\Gamma)}$ in $\partial\mathbb{H}$.

\begin{prop}\label{S_Gamma}
For any discrete $\Gamma\subset\mathrm{SL}_2\mathbb{R}$,
\[S_\Gamma=\overline{D(\Gamma)}\cap\partial\mathbb{H}.\]
\end{prop} 
\begin{proof}
We first show that for any $A=(\begin{smallmatrix}a & b\\ c & d\end{smallmatrix})\in\mathrm{SL}_2\mathbb{R}$,
\[S(A)=\overline{H(A)}\cap\partial\mathbb{H}.\]
If $A\in\text{SO}_2\mathbb{R}$, then both sets are all of $\partial\mathbb{H}$, so suppose $A\notin\text{SO}_2\mathbb{R}$.  Assume that $A$ satisfies $Ai=i+2t_0$ for some $t_0\in\mathbb{R}_+$.  Then $H(A)$ consists of all points $\tau\in\mathbb{H}$ with $\text{Re}(\tau)\le t_0$, and 
\[\overline{H(A)}\cap\partial\mathbb{H}=\{t\in\mathbb{R}\ |\ t\le t_0\}\cup\{\infty\}.\]
From $(ai+b)/(ci+d)=i+2t_0$ we find that 
\[a=2ct_0+d\ \ \ \ \ \text{and}\ \ \ \ \ b=2dt_0-c,\]
and, moreover,
\[1=\text{det}(A)=ad-bc=c^2+d^2.\]
Notice that $(x,y)^T\in S^1\backslash A\cdot\Delta$ if and only if $(x,y)^T\in S^1$ and $A^{-1}(x,y)^T\notin\Delta$.  The latter requirement may be rewritten as $|A^{-1}(x,y)^T|\ge |(x,y)^T|$.  We claim that $(1,0)^T$ and $\frac1{\sqrt{t_0^2+1}}(t_0,1)^T$ satisfy this with equality.  We compute
\[|A^{-1}(1,0)^T|^2=\left|\begin{pmatrix}d & -b\\ -c & a\end{pmatrix}\begin{pmatrix}1\\ 0\end{pmatrix}\right|^2=\left|\begin{pmatrix}d\\ -c\end{pmatrix}\right|^2=c^2+d^2=1=|(1,0)^T|^2\]
and
\begin{align*}
|A^{-1}(t_0,1)^T|^2&=\left|\begin{pmatrix}d & -b\\ -c & a\end{pmatrix}\begin{pmatrix}t_0\\ 1\end{pmatrix}\right|^2\\
&=\left|\begin{pmatrix}dt_0-b\\ -ct_0+a\end{pmatrix}\right|^2\\
&=(dt_0-b)^2+(-ct_0+a)^2\\
&=(-dt_0+c)^2+(ct_0+d)^2\\
&=d^2t_0^2-2cdt_0+c^2+c^2t_0^2+2cdt_0+d^2\\
&=t_0^2+1\\
&=|(t_0,1)^T|^2,
\end{align*}
so the claim holds.  Note that $A\cdot\Delta$ is the open region whose boundary is the ellipse $A\cdot S^1$ centered at the origin, so $S^1\cap A\cdot\Delta$ consists of all points in $S^1$ with slope strictly between $0$ and $1/t_0$ (see, for instance, Figure \ref{fanning_group_Gamma_0}, setting $A=T^2$ and $t_0=1$).  It follows that 
\[S(A)=\{t\in\mathbb{R}\ |\ t\leq t_0\}\cup\{\infty\}=\overline{H(A)}\cap\partial\mathbb{H}.\]
Now consider general $A\in\Gamma\backslash\text{SO}_2\mathbb{R}$, and fix $B\in\text{SO}_2\mathbb{R}$ such that $BAi=i+2t_0$ for some $t_0\in\mathbb{R}_+$.  From the argument above, we have
\[S(BA)=\overline{H(BA)}\cap\partial\mathbb{H}.\]
We claim that $S(BA)=BS(A)$ and $H(BA)=BH(A)$.  By definition, $t\in S(BA)$ if and only if there is some $(x,y)^T\in S^1\backslash BA\cdot\Delta$ with $x/y=t$.  The latter inclusion is equivalent to $B^{-1}(x,y)^T\in B^{-1}(S^1\backslash BA\cdot\Delta)=S^1\backslash A\cdot\Delta$ since $B^{-1}\in\text{SO}_2\mathbb{R}$.  Writing $B^{-1}=(\begin{smallmatrix} e & f\\ g & h\end{smallmatrix})$, we find $B^{-1}(x,y)^T=(ex+fy,gx+hy)^T$.  But $B^{-1}t=B^{-1}(x/y)=(ex+fy)/(gx+hy)$ is the reciprocal of the slope of $B^{-1}(x,y)^T\in S^1\backslash A\cdot\Delta$, so $t\in S(BA)$ if and only if $B^{-1}t\in S(A)$, as desired.  For the second claim, note that $\tau\in H(BA)$ if and only if $d_{\mathbb{H}}(i,\tau)\le d_{\mathbb{H}}(BAi,\tau)$.  Since $B^{-1}$ is an isometry fixing $i$, this inequality is equivalent to $d_{\mathbb{H}}(i,B^{-1}\tau)\le d_{\mathbb{H}}(Ai,B^{-1}\tau)$, which is true if and only if $B^{-1}\tau\in H(A)$.

Next we show that
\[\overline{D(\Gamma)}=\bigcap_{A\in\Gamma}\overline{H(A)}.\]
The forward inclusion holds as the closure of an intersection of sets is always contained in the intersection of the closures of the sets.  For the reverse inclusion, suppose $\tau\in\bigcap_{A\in\Gamma}\overline{H(A)}\backslash\{i\}$, and let $\gamma$ be a geodesic segment in $\mathbb{H}$ without endpoints, for which $\overline{\gamma}$ has endpoints $i$ and $\tau$.  Since $\tau\in\overline{H(A)}$ for each $A\in\Gamma$ and each $H(A)$ is convex with $i$ in its interior, we have $\gamma\subset\bigcap_{A\in\Gamma}H(A)=D(\Gamma)$.  Hence $\tau\in\overline{\gamma}\subset\overline{D(\Gamma)}$, so the claim holds.

It follows that 
\begin{align*}
S_\Gamma&=\bigcap_{A\in\Gamma}S(A)=\bigcap_{A\in\Gamma}\left(\overline{H(A)}\cap\partial\mathbb{H}\right)=\left(\bigcap_{A\in\Gamma}\overline{H(A)}\right)\cap\partial\mathbb{H}=\overline{D(\Gamma)}\cap\partial\mathbb{H}.
\end{align*}
\end{proof}

With the aid of Proposition \ref{S_Gamma}, lattices may be characterized as the finitely generated discrete groups $\Gamma$ for which the set of directions $\Theta_\Gamma$ is finite.  Furthermore, for a lattice $\Gamma$, the set of directions $\Theta_\Gamma$ may be computed in finite time:

\begin{lem}\label{lattice_and_fanning}
A discrete subgroup $\Gamma\subset\mathrm{SL}_2\mathbb{R}$ is a lattice if and only if it is a finitely generated fanning group.  
In this case there is some finite subset $\Gamma_n\subset\Gamma$ for which $\Theta_\Gamma=\Theta_{\Gamma_n}$. 
\end{lem}
\begin{proof}
Let $\Lambda$ denote the image in $\text{PSL}_2\mathbb{R}$ of $\Gamma\cap\text{SO}_2\mathbb{R}$.  As $\Lambda$ contains only elliptic elements, it is a finite cyclic group of order, say, $n$ (Corollary 2.4.2 of \cite{Katok}).  Let $A_0\in\Lambda$ be a generator of this group.  Recall that if $n=1$, then $D(\Gamma)$ is the Dirichlet region centered at $i$ for $\Gamma$ and is thus a convex fundamental polygon.  If $n>1$, then $D(\Gamma)$ is no longer a fundamental domain: one finds that $A_0D(\Gamma)=D(\Gamma)$, and the interior of $D(\Gamma)$ contains $n$ points from each orbit $\Gamma\tau,\ \tau\in \mathbb{H}$.  Let $\gamma_0$ denote a geodesic segment beginning at $i$ and ending on $\partial D(\Gamma)$.  For each $0\le j<n$, the geodesic segment $\gamma_j:=A_0^j\gamma_0$ also begins and ends at $i$ and $\partial D(\Gamma)$.  Let $\mathcal{F}_0\subset D(\Gamma)$ denote the convex polygon bounded by $\gamma_0,\gamma_1$ and $\partial D(\Gamma)$.  A slight generalization of the proof of Theorem 3.2.2 of \cite{Katok} shows that $\mathcal{F}_0$ is a convex fundamental polygon for $\Gamma$.  In particular, 
\[\mu_\mathbb{H}(D(\Gamma))=n\mu_\mathbb{H}(\mathcal{F}_0)=n\mu_\mathbb{H}(\Gamma\backslash\mathbb{H}),\]
where $\mu_\mathbb{H}$ denotes hyperbolic area.  Hence $\Gamma$ is a lattice if and only if $\mu_\mathbb{H}(D(\Gamma))$ is finite.  

Suppose $\Gamma$ is a lattice.  Then $\Gamma$ is geometrically finite and hence finitely generated.  If $\Gamma$ is not a fanning group, then $\Theta_\Gamma$---and, consequently, $S_\Gamma$---is infinite.  By Proposition \ref{S_Gamma}, $\overline{D(\Gamma)}\cap\partial\mathbb{H}$ is infinite, and by Gauss-Bonnet, $D(\Gamma)$ has infinite hyperbolic area.  This is a contradiction, so $\Gamma$ is a finitely generated fanning group.

Now assume $\Gamma$ is a finitely generated fanning group.  Again by Proposition \ref{S_Gamma}, this implies that $\overline{D(\Gamma)}\cap\partial\mathbb{H}$ is finite, so $\mathcal{F}_0\subset D(\Gamma)$ has no free sides and at most finitely many cusps.  Since $\Gamma$ is finitely generated, it is geometrically finite and thus every convex fundamental polygon for $\Gamma$ has finitely many sides (Theorem 10.1.2 of \cite{Beardon}).  By Gauss-Bonnet, $\Gamma$ is a lattice, and the first statement is proven.

For the second statement, suppose $\Gamma$ is a lattice.  From the previous paragraph, we know that $\mathcal{F}_0$ has finitely many sides, and thus the same is true of $D(\Gamma)$.  But then there is some finite subset $\Gamma_n\subset\Gamma$ for which $D(\Gamma_n)=D(\Gamma)$, and by Proposition \mbox{\ref{S_Gamma}}, $S_\Gamma=S_{\Gamma_n}$, which implies $\Theta_\Gamma=\Theta_{\Gamma_n}$.
\end{proof}

\subsection{The group $\text{Aff}^+_\mathcal{O}(X,\omega)$ and its action on marked segments}\label{The group Aff^+_O(X,omega) and its action on marked segments}

Let $(X,\omega)\in\mathcal{H}(d_1,\dots,d_\kappa)$ and $\mathcal{O}=\mathcal{O}(d_1,\dots,d_\kappa).$  A central result of \cite{Edwards} (Theorem 18) and \cite{ESS} (Proposition 17) is that a matrix $A\in\mathrm{SL}_2\mathbb{R}$ belongs to $\mathrm{SL}(X,\omega)$ if and only if there is some $f\in\text{Aff}^+(\mathcal{O})$ with $\text{der}(f)=A$ satisfying $f(\widehat{\mathcal{S}}(X,\omega))\subset\widehat{\mathcal{M}}(X,\omega)$.  Theorem 17 of \cite{Edwards} shows that the same statement holds with the latter subset inclusion replaced by the equality $f(\widehat{\mathcal{M}}(X,\omega))=\widehat{\mathcal{M}}(X,\omega)$.  Note, in particular, that these statements require the affine automorphism $f$ to respect the orientation-pairing of marked segments---it is not enough that the set of marked segments $\widehat{\mathcal{M}}_F(X,\omega)$ is invariant under $f$ to conclude that $\text{der}(f)\in\mathrm{SL}(X,\omega)$ (see Example 21 of \cite{ESS}).  We introduce the following notation for the collection of such affine automorphisms of $\mathcal{O}$:

\begin{defn}\label{Aff_O^+(X,omega)}
Let 
\[\text{Aff}_\mathcal{O}^+(X,\omega):=\{f\in\text{Aff}^+(\mathcal{O})\ |\ f(\widehat{\mathcal{M}}(X,\omega))=\widehat{\mathcal{M}}(X,\omega)\}.\]
\end{defn}

Note that $\text{Aff}^+_\mathcal{O}(X,\omega)$ is a subgroup of $\text{Aff}^+(\mathcal{O})$.  From the comments above we see that the image of $\text{Aff}_\mathcal{O}^+(X,\omega)$ under $\text{der}(\cdot)$ is precisely the Veech group $\Gamma=\mathrm{SL}(X,\omega)$ (though $\text{der}(\cdot)$ restricted to $\text{Aff}_\mathcal{O}^+(X,\omega)$ need not be injective, so $\text{Aff}_\mathcal{O}^+(X,\omega)$ and $\mathrm{SL}(X,\omega)$ are in general non-isomorphic).  
The group $\text{Aff}_\mathcal{O}^+(X,\omega)$ acts on $\widehat{\mathcal{M}}(X,\omega)$ via
\begin{align*}
\text{Aff}_\mathcal{O}^+(X,\omega)\times\widehat{\mathcal{M}}(X,\omega)&\to\widehat{\mathcal{M}}(X,\omega)\\
(f,\{\hat{s},\hat{s}'\})&\mapsto \{f(\hat{s}),f(\hat{s}')\}.  
\end{align*}
Note in particular that $f(\hat{s}')=f(\hat{s})'$.  For any $\{\hat{s},\hat{s}'\}\in\widehat{\mathcal{M}}(X,\omega)$, let $[\{\hat{s},\hat{s}'\}]$ denote the orbit of $\{\hat{s},\hat{s}'\}$ under this action.  

Recall from \S\mbox{\ref{Marked segments and Voronoi staples}} that $(X,\omega)$ may be reconstructed from its marked Voronoi staples $\widehat{\mathcal{S}}(X,\omega)\subset \widehat{\mathcal{M}}(X,\omega)$.  The goal of Algorithm \mbox{\ref{the_algorithm}} is to construct---using only $\Gamma=\mathrm{SL}(X,\omega)$ and $d_1\le\dots\le d_\kappa$---increasing subsets of $\text{Aff}_\mathcal{O}^+(X,\omega)$-orbits $[\{\hat{s},\hat{s}'\}]$ whose union eventually contains $\widehat{\mathcal{S}}(X,\omega)$.  To this end, the two main obstacles we face are
\begin{enumerate}
\item[(i)] determining $\text{Aff}_\mathcal{O}^+(X,\omega)$, and

\item[(ii)] determining a representative $\{\hat{s},\hat{s}'\}$ of each of the $\text{Aff}_\mathcal{O}^+(X,\omega)$-orbits $[\{\hat{s},\hat{s}'\}]$ intersecting $\widehat{\mathcal{S}}(X,\omega)$
\end{enumerate}
using only the data $\Gamma$ and $d_1\le\dots\le d_\kappa$.
The first challenge is addressed by the following:  

\begin{lem}\label{aff^+_cosets}
Let $G$ be a set of generators of $\Gamma\le\mathrm{SL}_2\mathbb{R}$.  Then for any $(X,\omega)\in\mathcal{H}(d_1,\dots,d_\kappa)$ with $\mathrm{SL}(X,\omega)=\Gamma$, the group $\text{Aff}_\mathcal{O}^+(X,\omega)$ is generated by a subset of
\[G_\mathcal{O}:=\{\tau\circ f_A\ |\ \tau\in\text{Trans}(\mathcal{O}),\ f_A\in \text{Aff}_C^+(\mathcal{O}),\ A\in G\}.\]
In particular, if $G$ is finite then $\text{Aff}_\mathcal{O}^+(X,\omega)$ is generated by a subset of the finite set $G_\mathcal{O}$.
\end{lem}
\begin{proof}
For each $A\in G$, fix $g_A\in\text{Aff}_\mathcal{O}^+(X,\omega)$ with $\text{der}(g_A)=A$.  Recall that $g_A$ may be written uniquely as $g_A=\tau_A\circ f_A$, where $\tau_A\in\text{Trans}(\mathcal{O})$ and $f_A\in \text{Aff}_C^+(\mathcal{O})$, so the set of such $g_A$ is contained in $G_\mathcal{O}$.  Let $\Lambda$ be the subgroup of $\text{Aff}_\mathcal{O}^+(X,\omega)$ generated by these $g_A$.  Since $\text{Trans}(\mathcal{O})$ is closed under composition, it suffices to show that $\text{Aff}_\mathcal{O}^+(X,\omega)=\cup\ \tau\Lambda$, where the union is over some collection of $\tau\in\text{Trans}(\mathcal{O})$.

Now let $f\in \text{Aff}_\mathcal{O}^+(X,\omega)$ be arbitrary and $A:=\text{der}(f)\in\mathrm{SL}(X,\omega)$.  The matrix $A$ may be written as a product $A=A_1^{\delta_1}\cdots A_m^{\delta_m},\ \delta_j\in\{\pm 1\},$ of elements $A_j$ in the generating set $G$ and their inverses.  Set $g:=g_{A_1}^{\delta_1}\dots g_{A_m}^{\delta_m}\in\Lambda$, with each $g_{A_j}$ a chosen generator of $\Lambda$ as above.  Then $\text{der}(f\circ g^{-1})=\text{Id}$, so $\tau:=f\circ g^{-1}\in\text{Trans}(\mathcal{O})$ and $f\in\tau\Lambda$ as desired.  

The final statement of the Lemma follows immediately from finiteness of both $\text{Trans}(\mathcal{O})$ (Lemma \mbox{\ref{Trans(O)}}) and $G$.  
\end{proof}

For challenge (ii) mentioned above, we wish to determine the generalized polar coordinates of representatives of $\text{Aff}_\mathcal{O}^+(X,\omega)$-orbits.  In the proof of the following Lemma---which is also crucial for Algorithm \mbox{\ref{the_algorithm}}---it is shown that the directions of the shortest representatives $\{\hat{s},\hat{s}'\}$ of each orbit $[\{\hat{s},\hat{s}'\}]$ are announced by the set $\Theta_{\mathrm{SL}(X,\omega)}$.  The more delicate procedure of determining the lengths of such representatives is addressed in \S\mbox{\ref{Finiteness of lattice surfaces with given Veech groups}}.

\begin{lem}\label{finite_orbit_space}
If $(X,\omega)$ is a fanning surface, then the orbit space
\[\widehat{\mathcal{M}}(X,\omega)/\text{Aff}_\mathcal{O}^+(X,\omega)=\left\{[\{\hat{s},\hat{s}'\}]\ \big|\ \{\hat{s},\hat{s}'\}\in\widehat{\mathcal{M}}(X,\omega)\right\}\]
is finite.
\end{lem}
\begin{proof}
Let $[\{\hat{s},\hat{s}'\}]\in \widehat{\mathcal{M}}(X,\omega)/\text{Aff}_\mathcal{O}^+(X,\omega)$.  The image under ${\bm\pi}(\cdot)$ of $[\{\hat{s},\hat{s}'\}]$ consists of pairs $\{\pm v\}$ of holonomy vectors in $\mathbb{C}$, among which there is some pair of minimal length.  The preimage under ${\bm\pi}(\cdot)$ of this pair is a finite set in $\mathcal{O}$ containing a pair of orientation-paired marked segments $\{\hat{s},\hat{s}'\}$ of minimal length in $[\{\hat{s},\hat{s}'\}]$.  
Rescaling $(X,\omega)$ if necessary, we may assume $|\hat{s}|=|\hat{s}'|=1$.  We claim that 
\[\left\{{\bm\pi}(\hat{s}),{\bm\pi}(\hat{s}')\right\}\subset\Theta_{\mathrm{SL}(X,\omega)}.\]
Note for each $A\in\mathrm{SL}(X,\omega)$ that ${\bm\pi}(\hat{s})\notin S^1\cap A\cdot\Delta$: otherwise $A^{-1}\cdot{\bm\pi}(\hat{s})\in\Delta$, but the comments following Definition \ref{Aff_O^+(X,omega)} together with Lemma \ref{pi*f=A*pi} imply $A^{-1}\cdot{\bm\pi}(\hat{s})={\bm\pi}\circ f(\hat{s})$ for some $f\in\text{Aff}^+_\mathcal{O}(X,\omega)$ with $\text{der}(f)=A^{-1}$. Since $1>|{\bm\pi}\circ f(\hat{s})|=|f(\hat{s})|$, this contradicts the assumption that $\{\hat{s},\hat{s}'\}$ is a pair of minimal length in $[\{\hat{s},\hat{s}'\}]$.  The same is true for ${\bm\pi}(\hat{s}')$, so we find that
\[\left\{{\bm\pi}(\hat{s}),{\bm\pi}(\hat{s}')\right\}\subset \bigcap_{A\in\mathrm{SL}(X,\omega)}(S^1\backslash A\cdot\Delta)=S^1\backslash\bigcup_{A\in\mathrm{SL}(X,\omega)} A\cdot\Delta=\Theta_{\mathrm{SL}(X,\omega)}\]
as claimed.  

For each $v\in S^1$, let $r_v$ denote the infinite open ray emanating from $0\in\mathbb{C}$ in the direction of $v$.  Note that the preimage ${\bm\pi}^{-1}(r_v)\subset\mathcal{O}$ consists of $\sum_{i=1}^\kappa (d_i+1)=2g-2+\kappa$ infinite open rays emanating from the origins of the various components of $\mathcal{O}$, and each of these rays contains at most one marked segment of $(X,\omega)$ (Proposition \ref{marked_segs_distinctive_and_lpf}).  From the claim above, the minimal-length representatives of each $[\{\hat{s},\hat{s}'\}]\in\widehat{\mathcal{M}}(X,\omega)/\text{Aff}_\mathcal{O}^+(X,\omega)$ belong to one of the $2g-2+\kappa$ open rays of ${\bm\pi}^{-1}(r_v)$ for some $v\in\Theta_{\mathrm{SL}(X,\omega)}$.  Hence
\[\left|\widehat{\mathcal{M}}(X,\omega)/\text{Aff}_\mathcal{O}^+(X,\omega)\right|\le (2g-2+\kappa)\cdot |\Theta_{\mathrm{SL}(X,\omega)}|,\]
where $|\cdot|$ denotes cardinality.  Since $(X,\omega)$ is a fanning surface, the right-hand side is finite.
\end{proof}

The proof of Lemma \ref{finite_orbit_space} also gives the following:
\begin{cor}\label{directions_for_orbits}
Let $(X,\omega)$ be a fanning surface and $[\{\hat{s},\hat{s}'\}]\in\widehat{\mathcal{M}}(X,\omega)/\text{Aff}_\mathcal{O}^+(X,\omega)$.  If $\{\hat{s},\hat{s}'\}$ are orientation-paired marked segments of minimal length in $[\{\hat{s},\hat{s}'\}]$, then ${\bm\pi}(\hat{s})/|{\bm\pi}(\hat{s})|$ and ${\bm\pi}(\hat{s}')/|{\bm\pi}(\hat{s}')|$ belong to the finite set $\Theta_{\mathrm{SL}(X,\omega)}$.
\end{cor}

\subsection{Simulating normalized $\text{Aff}^+_\mathcal{O}(X,\omega)$-orbits}\label{Simulating orbits}

The set of orientation-paired marked segments $\widehat{\mathcal{M}}(X,\omega)$ (and the set of $\text{Aff}_\mathcal{O}^+(X,\omega)$-orbits into which it partitions) depends intrinsically on $(X,\omega)$ and its geometry.  The results of the previous subsection suggest that in the case of a fanning surface, much of this information is encoded in the Veech group $\mathrm{SL}(X,\omega)$.  This subsection further explores these ideas by introducing \textit{simulations} of $\text{Aff}_\mathcal{O}^+(X,\omega)$-orbits which are constructed via generators of a fanning group.  We again fix a stratum $\mathcal{H}(d_1,\dots,d_\kappa)$.

\begin{defn}\label{simulations_defn}
Let $\Gamma\le\mathrm{SL}_2\mathbb{R}$ be a fanning group generated by $G\subset\Gamma$ and set
\[\mathfrak{S}_G:=\{(H,\{p,p^-\})\in 2^{G_\mathcal{O}}\backslash \{\varnothing\}\times \mathbb{P}(\mathcal{O})\ |\ {\bm\pi}(p),{\bm\pi}(p^-)\in\Theta_\Gamma\},\]
where $G_\mathcal{O}$ is defined as in Lemma \mbox{\ref{aff^+_cosets}} and $2^{G_\mathcal{O}}$ is the power set of $G_\mathcal{O}$.  For each $\mathfrak{s}:=(H,\{p,p^-\})\in\mathfrak{S}_G$, define the \textit{stage $0$ simulation} determined by $\mathfrak{s}$ by $\text{sim}^0(\mathfrak{s}):=\{\{p,p^-\}\}\subset\mathbb{P}(\mathcal{O})$, and for $n\ge 1$ define the \textit{stage $n$ simulation} determined by $\mathfrak{s}$ recursively as
\[\text{sim}^n(\mathfrak{s}):=\bigcup_{f\in H\cup\{\text{Id}\}} f^{\pm 1}(\text{sim}^{n-1}(\mathfrak{s}))\subset\mathbb{P}(\mathcal{O}).\]
The \textit{simulation} determined by $\mathfrak{s}$ is the union $\text{sim}(\mathfrak{s}):=\bigcup_{n\ge 0}\text{sim}^n(\mathfrak{s})$
of all stage $n$ simulations determined by $\mathfrak{s}$.  Denote by
\[\text{Sims}_G:=\{\text{sim}(\mathfrak{s})\ |\ \text{sim}(\mathfrak{s})\ \text{is distinctive},\ \mathfrak{s}\in\mathfrak{S}_G\}\ \ \ \text{and}\ \ \ \text{Sims}^n_G:=\{\text{sim}^n(\mathfrak{s})\ |\ \text{sim}^n(\mathfrak{s})\ \text{is distinctive},\ \mathfrak{s}\in\mathfrak{S}_G\}\]
the \textit{set of all distinctive simulations} and \textit{set of all stage $n$ distinctive simulations}, respectively, determined by $G$.  For any $r\in\mathbb{R}_+$, the \textit{$r$-scaled (stage $n$)} simulation is $r\text{sim}(\mathfrak{s})$ ($r\text{sim}^n(\mathfrak{s})$).  Finally, set $\text{sim}_F(\mathfrak{s}):=(\text{sim}(\mathfrak{s}))_F\subset\mathcal{O}$ and $\text{sim}^n_F(\mathfrak{s}):=(\text{sim}^n(\mathfrak{s}))_F\subset\mathcal{O}$ (recall Definition \mbox{\ref{P(O)}}).
\end{defn}

Note that by construction, the stage $n$ simulation $\text{sim}^n(\mathfrak{s})$ is the set of images in $\mathcal{O}$ of the oppositely-projected pair $\{p,p^-\}$ under all compositions of at most $n$ elements of $H$ and their inverses, and the simulation $\text{sim}(\mathfrak{s})$ is simply the $\langle H\rangle$-orbit of $\{p,p^-\}$.  Furthermore, the distinctive simulations in $\text{Sims}_G$ and $\text{Sims}^n_G$ are limit-point free since $\Gamma$ is discrete.

Recall $S,\ T$ and $\Gamma$ from Example \ref{Gamma_0_eg}.

\begin{eg}\label{simulations_example}
Let $G:=\{S,T^2\}$, $\mathcal{O}=\mathcal{O}(2)$, and set
\[\mathfrak{s}_0:=(H,\{p_0,p_0^-\})\in\mathfrak{S}_{G},\]
where $H=\{\rho_1^2\circ f_{S}, f_{T^2}\}$, and $p_0,p_0^-\in\mathcal{O}_1=\mathcal{O}$ have unit length (this is required since ${\bm\pi}(p_0),{\bm\pi}(p_0^-)\in\Theta_\Gamma\subset S^1$) and arguments $\text{arg}(p_0)=0$ and $\text{arg}(p_0^-)=5\pi$.  The map $\rho_1^2\circ f_{S}$ acts as a counterclockwise rotation of $\mathcal{O}$ by an angle of $\pi/2+4\pi=9\pi/2$, and $f_{T^2}$ is a bijective horizontal shear on each $c_j^1$.  A subset of the simulation $\text{sim}(\mathfrak{s}_0)$ is shown in cyan in Figure \mbox{\ref{simulations_figure}}.  In red and yellow are subsets of $\text{sim}(\mathfrak{s}_1)$ and $\text{sim}(\mathfrak{s}_2)$, respectively, for
\[\mathfrak{s}_1:=(H,\{p_1,p_1^-\})\in\mathfrak{S}_G\]
and
\[\mathfrak{s}_2:=(H,\{p_2,p_2^-\})\in\mathfrak{S}_G,\]
where $\text{arg}(p_1)=4\pi,\ \text{arg}(p_1^-)=\pi,\ \text{arg}(p_2)=\pi/4$ and $\text{arg}(p_2^-)=13\pi/4$.
\end{eg}

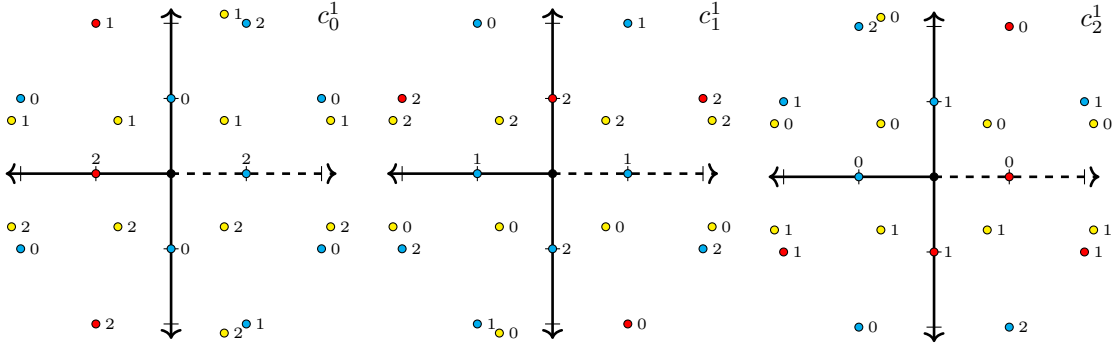
\begin{figure}[t]
\centering
\begin{minipage}{.3\textwidth}
\begin{tikzpicture}
\draw [->,dashed,line width=1] (0,0) -- (2.2,0);
\draw [<-,line width=1] (-2.2,0) -- (0,0);
\draw [<->,line width=1] (0,-2.2) -- (0,2.2);
\draw (2.1,2.1) circle (0pt) node {$c_0^1$};
\draw (1,-.1) -- (1,.1);
\draw (2,-.1) -- (2,.1);
\draw (-1,-.1) -- (-1,.1);
\draw (-2,-.1) -- (-2,.1);
\draw (-.1,1) -- (.1,1);
\draw (-.1,2) -- (.1,2);
\draw (-.1,-1) -- (.1,-1);
\draw (-.1,-2) -- (.1,-2);
\filldraw (0,0) circle (1.5pt);
\filldraw[fill=cyan] (1,0) circle (1.5pt) node[above] {\tiny $2$};
\filldraw[fill=cyan] (0,-1) circle (1.5pt) node[right] {\tiny $0$};
\filldraw[fill=cyan] (0,1) circle (1.5pt) node[right] {\tiny $0$};
\filldraw[fill=cyan] (-2,-1) circle (1.5pt) node[right] {\tiny $0$};
\filldraw[fill=cyan] (2,1) circle (1.5pt) node[right] {\tiny $0$};
\filldraw[fill=cyan] (2,-1) circle (1.5pt) node[right] {\tiny $0$};
\filldraw[fill=cyan] (-2,1) circle (1.5pt) node[right] {\tiny $0$};
\filldraw[fill=cyan] (1,2) circle (1.5pt) node[right] {\tiny $2$};
\filldraw[fill=cyan] (1,-2) circle (1.5pt) node[right] {\tiny $1$};
\filldraw[fill=yellow] ({sqrt(2)/2},{sqrt(2)/2}) circle (1.5pt) node[right] {\tiny $1$};
\filldraw[fill=yellow] ({-sqrt(2)/2},{sqrt(2)/2}) circle (1.5pt) node[right] {\tiny $1$};
\filldraw[fill=yellow] ({-sqrt(2)/2},{-sqrt(2)/2}) circle (1.5pt) node[right] {\tiny $2$};
\filldraw[fill=yellow] ({sqrt(2)/2},{-sqrt(2)/2}) circle (1.5pt) node[right] {\tiny $2$};
\filldraw[fill=yellow] ({3*sqrt(2)/2},{sqrt(2)/2}) circle (1.5pt) node[right] {\tiny $1$};
\filldraw[fill=yellow] ({3*sqrt(2)/2},{-sqrt(2)/2}) circle (1.5pt) node[right] {\tiny $2$};
\filldraw[fill=yellow] ({sqrt(2)/2},{-3*sqrt(2)/2}) circle (1.5pt) node[right] {\tiny $2$};
\filldraw[fill=yellow] ({sqrt(2)/2},{3*sqrt(2)/2}) circle (1.5pt) node[right] {\tiny $1$};
\filldraw[fill=yellow] ({-3*sqrt(2)/2},{-sqrt(2)/2}) circle (1.5pt) node[right] {\tiny $2$};
\filldraw[fill=yellow] ({-3*sqrt(2)/2},{sqrt(2)/2}) circle (1.5pt) node[right] {\tiny $1$};
\filldraw[fill=red] (-1,0) circle (1.5pt) node[above] {\tiny $2$};
\filldraw[fill=red] (-1,-2) circle (1.5pt) node[right] {\tiny $2$};
\filldraw[fill=red] (-1,2) circle (1.5pt) node[right] {\tiny $1$};
\end{tikzpicture}
\end{minipage}
\begin{minipage}{.3\textwidth}
\begin{tikzpicture}
\draw [->,dashed,line width=1] (0,0) -- (2.2,0);
\draw [<-,line width=1] (-2.2,0) -- (0,0);
\draw [<->,line width=1] (0,-2.2) -- (0,2.2); 
\draw (2.1,2.1) circle (0pt) node {$c_1^1$};
\draw (1,-.1) -- (1,.1);
\draw (2,-.1) -- (2,.1);
\draw (-1,-.1) -- (-1,.1);
\draw (-2,-.1) -- (-2,.1);
\draw (-.1,1) -- (.1,1);
\draw (-.1,2) -- (.1,2);
\draw (-.1,-1) -- (.1,-1);
\draw (-.1,-2) -- (.1,-2);
\filldraw (0,0) circle (1.5pt);
\filldraw[fill=cyan] (0,-1) circle (1.5pt) node[right] {\tiny $2$};
\filldraw[fill=cyan] (-2,-1) circle (1.5pt) node[right] {\tiny $2$};
\filldraw[fill=cyan] (2,-1) circle (1.5pt) node[right] {\tiny $2$};
\filldraw[fill=cyan] (-1,2) circle (1.5pt) node[right] {\tiny $0$};
\filldraw[fill=cyan] (1,0) circle (1.5pt) node[above] {\tiny $1$};
\filldraw[fill=cyan] (-1,0) circle (1.5pt) node[above] {\tiny $1$};
\filldraw[fill=cyan] (1,2) circle (1.5pt) node[right] {\tiny $1$};
\filldraw[fill=cyan] (-1,-2) circle (1.5pt) node[right] {\tiny $1$};
\filldraw[fill=yellow] ({sqrt(2)/2},{sqrt(2)/2}) circle (1.5pt) node[right] {\tiny $2$};
\filldraw[fill=yellow] ({-sqrt(2)/2},{sqrt(2)/2}) circle (1.5pt) node[right] {\tiny $2$};
\filldraw[fill=yellow] ({-sqrt(2)/2},{-sqrt(2)/2}) circle (1.5pt) node[right] {\tiny $0$};
\filldraw[fill=yellow] ({sqrt(2)/2},{-sqrt(2)/2}) circle (1.5pt) node[right] {\tiny $0$};
\filldraw[fill=yellow] ({3*sqrt(2)/2},{sqrt(2)/2}) circle (1.5pt) node[right] {\tiny $2$};
\filldraw[fill=yellow] ({3*sqrt(2)/2},{-sqrt(2)/2}) circle (1.5pt) node[right] {\tiny $0$};
\filldraw[fill=yellow] ({-3*sqrt(2)/2},{-sqrt(2)/2}) circle (1.5pt) node[right] {\tiny $0$};
\filldraw[fill=yellow] ({-3*sqrt(2)/2},{sqrt(2)/2}) circle (1.5pt) node[right] {\tiny $2$};
\filldraw[fill=yellow] ({-sqrt(2)/2},{-3*sqrt(2)/2}) circle (1.5pt) node[right] {\tiny $0$};
\filldraw[fill=red] (0,1) circle (1.5pt) node[right] {\tiny $2$};
\filldraw[fill=red] (2,1) circle (1.5pt) node[right] {\tiny $2$};
\filldraw[fill=red] (-2,1) circle (1.5pt) node[right] {\tiny $2$};
\filldraw[fill=red] (1,-2) circle (1.5pt) node[right] {\tiny $0$};
\end{tikzpicture}
\end{minipage}
\begin{minipage}{.3\textwidth}
\begin{tikzpicture}
\draw [->,dashed,line width=1] (0,0) -- (2.2,0);
\draw [<-,line width=1] (-2.2,0) -- (0,0);
\draw [<->,line width=1] (0,-2.2) -- (0,2.2);
\draw (2.1,2.1) circle (0pt) node {$c_2^1$};
\draw (1,-.1) -- (1,.1);
\draw (2,-.1) -- (2,.1);
\draw (-1,-.1) -- (-1,.1);
\draw (-2,-.1) -- (-2,.1);
\draw (-.1,1) -- (.1,1);
\draw (-.1,2) -- (.1,2);
\draw (-.1,-1) -- (.1,-1);
\draw (-.1,-2) -- (.1,-2);
\filldraw (0,0) circle (1.5pt);
\filldraw[fill=cyan] (-1,0) circle (1.5pt) node[above] {\tiny $0$};
\filldraw[fill=cyan] (0,1) circle (1.5pt) node[right] {\tiny $1$};
\filldraw[fill=cyan] (2,1) circle (1.5pt) node[right] {\tiny $1$};
\filldraw[fill=cyan] (-2,1) circle (1.5pt) node[right] {\tiny $1$};
\filldraw[fill=cyan] (-1,-2) circle (1.5pt) node[right] {\tiny $0$};
\filldraw[fill=cyan] (1,-2) circle (1.5pt) node[right] {\tiny $2$};
\filldraw[fill=cyan] (-1,2) circle (1.5pt) node[right] {\tiny $2$};
\filldraw[fill=yellow] ({sqrt(2)/2},{sqrt(2)/2}) circle (1.5pt) node[right] {\tiny $0$};
\filldraw[fill=yellow] ({-sqrt(2)/2},{sqrt(2)/2}) circle (1.5pt) node[right] {\tiny $0$};
\filldraw[fill=yellow] ({-sqrt(2)/2},{-sqrt(2)/2}) circle (1.5pt) node[right] {\tiny $1$};
\filldraw[fill=yellow] ({sqrt(2)/2},{-sqrt(2)/2}) circle (1.5pt) node[right] {\tiny $1$};
\filldraw[fill=yellow] ({3*sqrt(2)/2},{sqrt(2)/2}) circle (1.5pt) node[right] {\tiny $0$};
\filldraw[fill=yellow] ({3*sqrt(2)/2},{-sqrt(2)/2}) circle (1.5pt) node[right] {\tiny $1$};
\filldraw[fill=yellow] ({-3*sqrt(2)/2},{-sqrt(2)/2}) circle (1.5pt) node[right] {\tiny $1$};
\filldraw[fill=yellow] ({-3*sqrt(2)/2},{sqrt(2)/2}) circle (1.5pt) node[right] {\tiny $0$};
\filldraw[fill=yellow] ({-sqrt(2)/2},{3*sqrt(2)/2}) circle (1.5pt) node[right] {\tiny $0$};
\filldraw[fill=red] (1,0) circle (1.5pt) node[above] {\tiny $0$};
\filldraw[fill=red] (0,-1) circle (1.5pt) node[right] {\tiny $1$};
\filldraw[fill=red] (-2,-1) circle (1.5pt) node[right] {\tiny $1$};
\filldraw[fill=red] (2,-1) circle (1.5pt) node[right] {\tiny $1$};
\filldraw[fill=red] (1,2) circle (1.5pt) node[right] {\tiny $0$};
\end{tikzpicture}
\end{minipage}
\caption{Three (subsets) of the simulations $\text{sim}(\mathfrak{s}_0),\ \text{sim}(\mathfrak{s}_1),\ \text{sim}(\mathfrak{s}_2)\in\text{Sims}_{G}$ from Example \ref{simulations_example} are shown in cyan, red and yellow, respectively.  A label $j$ next to a point $p$ indicates to which $c_j^1$ the point $p^-$ belongs.}
\label{simulations_figure}
\end{figure}

Observe that Definition \mbox{\ref{simulations_defn}} and Lemmas \mbox{\ref{lattice_and_fanning}} and \mbox{\ref{aff^+_cosets}} immediately imply:

\begin{lem}\label{Sims_G_finite}
If $\Gamma$ is a lattice, then there exists a set of generators $G$ of $\Gamma$ for which the set of simulations $\text{Sims}_G$ is finite.  
\end{lem}

\begin{remark}
Recall from Lemma \mbox{\ref{lattice_and_fanning}} that for a lattice $\Gamma$, the set of directions $\Theta_\Gamma$ may be computed in finite time.  Indeed, letting $\Gamma_n$ be the subset of $\Gamma$ consisting of all words of length no greater than $n$ in the generators $G$ and their inverses, there is some $n\in\mathbb{N}$ for which $\Theta_{\Gamma_n}=\Theta_\Gamma$.  In general, $\Theta_{\Gamma_n}\supset\Theta_\Gamma$, and it may be difficult to determine when equality holds.  Thus in practice, the set $\Theta_\Gamma$ in the definition of $\mathfrak{S}_G$ (Definition \mbox{\ref{simulations_defn}}) is replaced by some finite $\Theta_{\Gamma_n}$ containing $\Theta_\Gamma$.  Such a replacement does not affect the statements of the results in the remainder of the paper, though it could increase the size of the---still finite---set of simulations $\text{Sims}_G$.
\end{remark}

\begin{defn}
Let $r\in\mathbb{R}_+$ be the minimal length of orientation-paired marked segments in $[\{\hat{s},\hat{s}'\}]\in\widehat{\mathcal{M}}(X,\omega)/\text{Aff}_\mathcal{O}^+(X,\omega)$.  For any marked segment $\hat{s}$ belonging to a pair in the orbit $[\{\hat{s},\hat{s}'\}]$, the \textit{normalized marked segment} corresponding to $\hat{s}$ is $n(\hat{s}):=(1/r)\hat{s}$.  Set $n(\hat{s})':=n(\hat{s}')$, and define the \textit{normalized $\text{Aff}^+_\mathcal{O}(X,\omega)$-orbit} corresponding to $[\{\hat{s},\hat{s}'\}]$ as
\[n([\{\hat{s},\hat{s}'\}]):=\big\{\{n(\hat{s}),n(\hat{s})'\}\ \big|\ {\{\hat{s},\hat{s}'\}\in[\{\hat{s},\hat{s}'\}]}\big\}.\]
\end{defn}

Thus the normalized $\text{Aff}^+_\mathcal{O}(X,\omega)$-orbit corresponding to $[\{\hat{s},\hat{s}'\}]$ takes all points of the orbit $[\{\hat{s},\hat{s}'\}]$ and rescales them (by a constant scalar $1/r$) so that the minimal length of a pair in the orbit is one.

\begin{thm}\label{marked_segments_as_scaled_simulations}
If $(X,\omega)$ is a fanning surface, then $\widehat{\mathcal{M}}(X,\omega)$ is a finite union of scaled simulations.  In particular, if $\Gamma\le\mathrm{SL}_2\mathbb{R}$ is a fanning group generated by $G\subset\Gamma$, then
\[\left\{n([\{\hat{s},\hat{s}'\}])\ \Big|\ [\{\hat{s},\hat{s}'\}]\in\widehat{\mathcal{M}}(X,\omega)/\text{Aff}_\mathcal{O}^+(X,\omega)\ \text{for some}\ (X,\omega)\ \text{with}\ \mathrm{SL}(X,\omega)=\Gamma\right\}\subset\text{Sims}_G.\]
\end{thm}
\begin{proof}
We will prove the second statement; the first follows from this together with Lemma \ref{finite_orbit_space}.  Let $(X,\omega)$ be a translation surface with $\mathrm{SL}(X,\omega)=\Gamma$, and let $G$ be a (finite or infinite) set of generators for $\Gamma$. By Lemma \mbox{\ref{aff^+_cosets}}, there is some subset $H\subset G_\mathcal{O}$ which generates $\text{Aff}^+_\mathcal{O}(X,\omega)$.  Now let $[\{\hat{s},\hat{s}'\}]\in\widehat{\mathcal{M}}(X,\omega)/\text{Aff}_\mathcal{O}^+(X,\omega)$ with $\{\hat{s},\hat{s}'\}$ a pair of minimal length in $[\{\hat{s},\hat{s}'\}]$.  Set $p:=n(\hat{s})$ and $p^-:=n(\hat{s})'$.  By Corollary \mbox{\ref{directions_for_orbits}}, ${\bm\pi}(p)={\bm\pi}(\hat{s})/|{\bm\pi}(\hat{s})|$ and ${\bm\pi}(p^-)={\bm\pi}(\hat{s}')/|{\bm\pi}(\hat{s}')|$ belong to the finite set $\Theta_\Gamma$, so $\mathfrak{s}:=(H,\{p,p^-\})\in\mathfrak{S}_G$.  Furthermore, $\text{sim}(\mathfrak{s})$ is the $\langle H\rangle=\text{Aff}^+_\mathcal{O}(X,\omega)$-orbit of $\{p,p^-\}=\{n(\hat{s}),n(\hat{s})'\}$, and this is precisely the set $n([\{\hat{s},\hat{s}'\}])$.  By Proposition \mbox{\ref{marked_segs_distinctive_and_lpf}}, $[\{\hat{s},\hat{s}'\}]$---and thus also $n([\{\hat{s},\hat{s}'\}])$---is distinctive; hence $n([\{\hat{s},\hat{s}'\}])\in \text{Sims}_G$.
\end{proof}

Theorem \ref{marked_segments_as_scaled_simulations} implies that for any fanning group $\Gamma$, the orientation-paired marked segments $\widehat{\mathcal{M}}(X,\omega)$ of any $(X,\omega)$ with Veech group $\Gamma$ are given by a finite collection of simulations, each of which is appropriately scaled.  Recall that the marked Voronoi staples $\widehat{\mathcal{S}}(X,\omega)$---from which $(X,\omega)$ may be reconstructed---are determined by $\widehat{\mathcal{M}}(X,\omega)$ (\S\ref{Marked segments and Voronoi staples}).  In \S\ref{Finiteness of lattice surfaces with given Veech groups} , we show that for any finite collection of simulations, there are at most finitely many scalars for which the union of scaled simulations agrees with the marked segments of a unit-area translation surface.  Together with Lemma \ref{Sims_G_finite}, this will prove Theorem \ref{finiteness_thm}.

\subsection{Marked Voronoi staples determine permissible triples}\label{Marked Voronoi staples determine permissible triples}

Let $(X,\omega)\in\mathcal{H}(d_1,\dots,d_\kappa)$ and $\mathcal{O}=\mathcal{O}(d_1,\dots,d_\kappa)$.  In this subsection we apply the results of \S\ref{Permissible triples} to orientation-paired marked segments and their $\text{Aff}_\mathcal{O}^+(X,\omega)$-orbits.  In particular, we shall see that adjacent elements of $\widehat{\mathcal{S}}_F(X,\omega)$ naturally determine permissible triples (Proposition \ref{vor_staples_det_perm_triples_prop}).  To this end, we begin with the following:

\begin{lem}\label{ball_contains_no_marked_segs}
For any two adjacent $\hat{s},\hat{t}\in\widehat{\mathcal{S}}_F(X,\omega)$, we have 
\[B(\hat{s},\hat{t})\cap\widehat{\mathcal{M}}_F(X,\omega)=\varnothing.\]
Moreover, $[\hat{s},\hat{t}]\in\widehat{\mathcal{M}}_F(X,\omega)$ with $[\hat{s},\hat{t}]'=[\hat{t},\hat{s}]$.
\end{lem}
\begin{proof}
Let $i$ be the index for which $\hat{s},\hat{t}\in\mathcal{O}_i$, and recall from Definition \ref{circumcenter_circumcircle_ball} that the circumcenter $c:=c(\hat{s},\hat{t})$ determined by $\hat{s},\hat{t}\in\mathcal{O}_i$ is the intersection of $\partial H(\hat{s})$ and $\partial H(\hat{t})$.  Since $\hat{s},\hat{t}\in\widehat{\mathcal{S}}_F(X,\omega)$ are essential points of $\Omega_i$, there are two (adjacent) edges of $\Omega_i$ contained in $\partial H(\hat{s})$ and $\partial H(\hat{t})$, respectively.  The intersection of these edges is precisely the circumcenter $c$, so $c\in\Omega_i$.  Suppose for the sake of contradiction that there is some marked segment $\hat{u}\in B(\hat{s},\hat{t})$.  Then 
\[d(\hat{u},c)<|c|=d(0,c),\]
and hence $c\notin H(\hat{u})$.  This contradicts the fact that $c\in\Omega_i$, so we have
\[B(\hat{s},\hat{t})\cap\widehat{\mathcal{M}}_F(X,\omega)=\varnothing.\]

This result---together with the convexity of $B(\hat{s},\hat{t})$ and the fact that $0\in\partial B(\hat{s},\hat{t})$---gives that $B(\hat{s},\hat{t})\subset\text{star}_i(X,\omega)$, so $\eta$ is defined on all of $B(\hat{s},\hat{t})$.  Let $\gamma$ be the straight line segment in $\mathcal{O}_i$ from $\hat{s}$ to $\hat{t}$, and let $\sigma_j,\sigma_k\in\Sigma$ be the singularities at which the oriented saddle connections $s$ and $t$ end, respectively.  Then, recycling notation, $u:=\eta(\gamma)$ is a saddle connection from $\sigma_j$ to $\sigma_k$, and $\hat{u}\in\widehat{\mathcal{M}}_F(X,\omega)$.  We claim that $\hat{u}=[\hat{s},\hat{t}]$.

By Proposition \ref{adjacent_vs_in_same_pi_sector}, $\hat{s}$ and $\hat{t}$ belong to the same open $\pi$-sector of $\mathcal{O}_i$.  We also have by Proposition \ref{marked_segs_distinctive_and_lpf} that $\text{arg}(\hat{s})\neq\text{arg}(\hat{t})$, so $[\hat{s},\hat{t}]$ is in fact defined.  It is clear from the definition of $\hat{u}$ that $\hat{s}'$ and $\hat{u}$ belong to the same open $\pi$-sector of the same component $\mathcal{O}_j$.  Furthermore, ${\bm\pi}(\hat{u})={\bm\pi}(\hat{t})-{\bm\pi}(\hat{s})$; see Figure \ref{[s,t]_a_marked_seg}.  By Definition \ref{[p,q]}, we have $\hat{u}=[\hat{s},\hat{t}]$.  Reversing the orientation of $\gamma$ in the argument above gives that $\hat{u}'=[\hat{t},\hat{s}]$.
\end{proof}

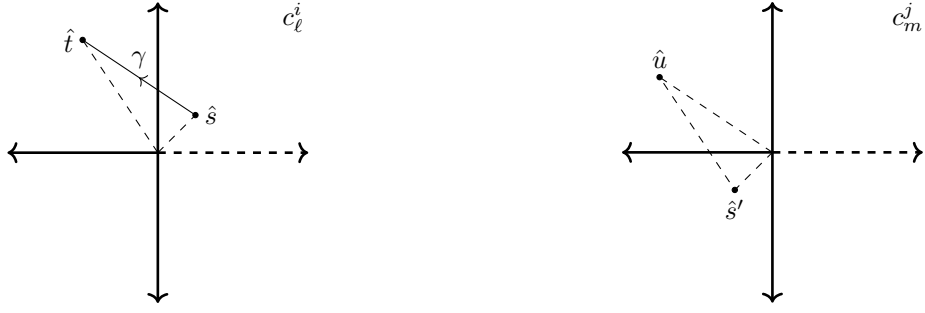
\begin{figure}[t]
\begin{minipage}{.49\textwidth}
\centering
\begin{tikzpicture}
\draw[->,line width=1,dashed] (0,0) -- (2,0);
\draw[->,line width=1] (0,0) -- (-2,0);
\draw[<->,line width=1] (0,-2) -- (0,2);
\draw[dashed] (0,0) -- (.5,.5);
\draw[dashed] (0,0) -- (-1,1.5);
\draw[->-=.5] (.5,.5) -- (-1,1.5) node[midway,above] {$\gamma$};
\filldraw (.5,.5) circle (1pt) node[right] {$\hat{s}$};
\filldraw (-1,1.5) circle (1pt) node[left] {$\hat{t}$};
\draw (1.8,1.8) circle (0pt) node {$c_\ell^i$};
\end{tikzpicture}
\end{minipage}
\begin{minipage}{.49\textwidth}
\centering
\begin{tikzpicture}
\draw[->,line width=1,dashed] (0,0) -- (2,0);
\draw[->,line width=1] (0,0) -- (-2,0);
\draw[<->,line width=1] (0,-2) -- (0,2);
\draw[dashed] (0,0) -- (-1-.5,1.5-.5) -- (-.5,-.5) -- cycle;
\filldraw (-.5,-.5) circle (1pt) node[below] {$\hat{s}'$};
\filldraw (-1-.5,1.5-.5) circle (1pt) node[above] {$\hat{u}$};
\draw (1.8,1.8) circle (0pt) node {$c_m^j$};
\end{tikzpicture}
\end{minipage}
\caption{Illustration of the proof of Lemma \ref{ball_contains_no_marked_segs}, showing that $\hat{u}=[\hat{s},\hat{t}]$ is a marked segment for adjacent $\hat{s},\hat{t}\in\widehat{\mathcal{S}}_F(X,\omega)$.}
\label{[s,t]_a_marked_seg}
\end{figure}

Recall that permissible triples are defined in terms of three distinctive, limit-point free subsets of $\mathbb{P}(\mathcal{O})$ (Definition \mbox{\ref{permissible_triples_defn}}).  We introduce notation for a set of---a priori non-permissible---triples determined by adjacent $\hat{s},\hat{t}\in\widehat{\mathcal{S}}_F(X,\omega)$ and subsequently prove that these are in fact permissible triples belonging to the set $\mathcal{P}(\text{Sim}_G)$, where $\text{SL}(X,\omega)=\langle G\rangle$.

\begin{defn}\label{set_of_permissible_triples}
For $(X,\omega)\in\mathcal{H}(d_1,\dots,d_\kappa)$, let
\[\mathcal{P}(X,\omega):=\{(n(\hat{s}),n(\hat{t}),n([\hat{s},\hat{t}]))\ |\ \hat{s},\ \hat{t}\ \text{are adjacent elements of}\ \widehat{\mathcal{S}}_F(X,\omega)\}.\]
\end{defn}

\begin{prop}\label{vor_staples_det_perm_triples_prop}
For any fanning group $\Gamma\le\mathrm{SL}_2\mathbb{R}$ generated by $G\subset\Gamma$ and any $(X,\omega)\in\mathcal{H}(d_1,\dots,d_\kappa)$ with Veech group $\mathrm{SL}(X,\omega)=\Gamma$, we have $\mathcal{P}(X,\omega)\subset\mathcal{P}(\text{Sims}_G)$.  Moreover, if $G$ is finite (i.e. $\Gamma$ is a lattice), then there exists some $N\in\mathbb{N}$ such that for any $(X,\omega)\in\mathcal{H}(d_1,\dots,d_\kappa)$ with Veech group $\mathrm{SL}(X,\omega)=\Gamma$, we have $\mathcal{P}(X,\omega)\subset\mathcal{P}(\text{Sims}_G)\subset\mathcal{P}(\text{Sims}^N_G)$.
\end{prop}
\begin{proof}
Let $(n(\hat{s}),n(\hat{t}),n([\hat{s},\hat{t}]))\in\mathcal{P}(X,\omega)$ for some $(X,\omega)$ with Veech group $\Gamma$.  By Theorem \ref{marked_segments_as_scaled_simulations}, there are $\mathfrak{s}_{\hat{s}},\mathfrak{s}_{\hat{t}},\mathfrak{s}_{[\hat{s},\hat{t}]}\in\mathfrak{S}_G$ for which 
\[n([\{\hat{s},\hat{s}'\}])=\text{sim}(\mathfrak{s}_{\hat{s}}),\ \ \ n([\{\hat{t},\hat{t}'\}])=\text{sim}(\mathfrak{s}_{\hat{t}}),\ \ \ \text{and}\ \ \ n([\{[\hat{s},\hat{t}],[\hat{s},\hat{t}]'\}])=\text{sim}(\mathfrak{s}_{[\hat{s},\hat{t}]}).\]
Hence 
\[(n(\hat{s}),n(\hat{t}),n([\hat{s},\hat{t}]))\in \text{sim}_F(\mathfrak{s}_{\hat{s}})\times\text{sim}_F(\mathfrak{s}_{\hat{t}})\times\text{sim}_F(\mathfrak{s}_{[\hat{s},\hat{t}]}),\]
and we must show that $(n(\hat{s}),n(\hat{t}),n([\hat{s},\hat{t}]))$ is a permissible triple.  The corresponding permissible scalars we shall use are the minimal lengths $(r_{\hat{s}}, r_{\hat{t}}, r_{[\hat{s},\hat{t}]})\in\mathbb{R}_+^3$ of pairs in the respective $\text{Aff}_\mathcal{O}^+(X,\omega)$-orbits $[\{\hat{s},\hat{s}'\}],\ [\{\hat{t},\hat{t}'\}]$ and $[\{[\hat{s},\hat{t}],[\hat{s},\hat{t}]'\}]$.  In particular, $r_{\hat{s}}n(\hat{s})=\hat{s},\ r_{\hat{t}}n(\hat{t})=\hat{t}$ and $r_{[\hat{s},\hat{t}]}n([\hat{s},\hat{t}])=[\hat{s},\hat{t}]$.  Moreover, $r_{\hat{s}}\text{sim}(\mathfrak{s}_{\hat{s}})=[\{\hat{s},\hat{s}'\}],\ r_{\hat{t}}\text{sim}(\mathfrak{s}_{\hat{t}})=[\{\hat{t},\hat{t}'\}]$ and $r_{[\hat{s},\hat{t}]}\text{sim}(\mathfrak{s}_{[\hat{s},\hat{t}]})=[\{[\hat{s},\hat{t}],[\hat{s},\hat{t}]'\}]$.  We now verify the two criteria of Definition \mbox{\ref{permissible_triples_defn}} on the triple $(n(\hat{s}),n(\hat{t}),n([\hat{s},\hat{t}]))$ and the scalars $(r_{\hat{s}}, r_{\hat{t}}, r_{[\hat{s},\hat{t}]})$.
\begin{enumerate}
\item[(i)]  From the observations above, $[r_{\hat{s}}n(\hat{s}),r_{\hat{t}}n(\hat{t})]=[\hat{s},\hat{t}]$ and $[r_{\hat{t}}n(\hat{t}),r_{\hat{s}}n(\hat{s})]=[\hat{t},\hat{s}]$.  Since $\hat{s}$ and $\hat{t}$ are adjacent elements of $\widehat{\mathcal{S}}_F(X,\omega)$, Lemma \ref{ball_contains_no_marked_segs} guarantees these are defined.  Moreover, 
\[[r_{\hat{s}}n(\hat{s}),r_{\hat{t}}n(\hat{t})]=[\hat{s},\hat{t}]=r_{[\hat{s},\hat{t}]}n([\hat{s},\hat{t}])\]
and 
\[[r_{\hat{t}}n(\hat{t}),r_{\hat{s}}n(\hat{s})]=[\hat{t},\hat{s}]=[\hat{s},\hat{t}]'=(r_{[\hat{s},\hat{t}]}n([\hat{s},\hat{t}]))'=r_{[\hat{s},\hat{t}]}n([\hat{s},\hat{t}])'.\]

\item[(ii)]  We must show that 
\[B(r_{\hat{s}}n(\hat{s}),r_{\hat{t}}n(\hat{t}))\cup B(r_{[\hat{s},\hat{t}]}n([\hat{s},\hat{t}]),r_{\hat{s}}n(\hat{s})')\cup B(r_{\hat{t}}n(\hat{t})',r_{[\hat{s},\hat{t}]}n([\hat{s},\hat{t}])')\]
does not intersect \[r_{\hat{s}}\text{sim}_F(\mathfrak{s}_{\hat{s}})\cup r_{\hat{t}}\text{sim}_F(\mathfrak{s}_{\hat{t}})\cup r_{[\hat{s},\hat{t}]}\text{sim}_F(\mathfrak{s}_{[\hat{s},\hat{t}]}),\]
or, equivalently,
\[\left(B(\hat{s},\hat{t})\cup B([\hat{s},\hat{t}],\hat{s}')\cup B(\hat{t}',[\hat{s},\hat{t}]')\right)\cap\left([\{\hat{s},\hat{s}'\}]_F\cup [\{\hat{t},\hat{t}'\}]_F\cup [\{[\hat{s},\hat{t}],[\hat{s},\hat{t}]'\}]_F\right)=\varnothing.\]
Since each of the $\text{Aff}_\mathcal{O}^+(X,\omega)$-orbits are contained in $\widehat{\mathcal{M}}(X,\omega)$, it suffices to show
\[\left(B(\hat{s},\hat{t})\cup B([\hat{s},\hat{t}],\hat{s}')\cup B(\hat{t}',[\hat{s},\hat{t}]')\right)\cap\widehat{\mathcal{M}}_F(X,\omega)=\varnothing.\]
Lemma \ref{ball_contains_no_marked_segs} gives that 
\[B(\hat{s},\hat{t})\cap\widehat{\mathcal{M}}_F(X,\omega)=\varnothing.\]
Now suppose for the sake of contradiction that
\[B([\hat{s},\hat{t}],\hat{s}')\cap\widehat{\mathcal{M}}_F(X,\omega)\neq\varnothing,\]
and let $j$ be the index for which $\hat{s}'\in\mathcal{O}_j$.  Since $\widehat{\mathcal{M}}_F(X,\omega)$ has no limit points, there are at most finitely many marked segments belonging to this intersection, and by Proposition \ref{marked_segs_distinctive_and_lpf}, none of these marked segments have the same argument as $\hat{s}'$.  Choose some $\hat{u}$ in the intersection with argument nearest that of $\hat{s}'$.  Then the straight-line segment from $\hat{s}'$ to $\hat{u}$ belongs to $\text{star}_j(X,\omega)$, and a similar argument to that in the proof of Lemma \ref{ball_contains_no_marked_segs} gives that $[\hat{s}',\hat{u}]\in\widehat{\mathcal{M}}(X,\omega)$.    By definition of $[\hat{s}',\hat{u}]$, this marked segment belongs to the same open $\pi$-sector as $\hat{s}$ and hence to $B(\hat{s},\hat{t})$, contradicting Lemma \ref{ball_contains_no_marked_segs} (see Figure \ref{marked_segs_and_balls}).  A similar argument gives
\[B(\hat{t}',[\hat{s},\hat{t}]')\cap\widehat{\mathcal{M}}_F(X,\omega)=\varnothing.\]
\end{enumerate}
Thus $(n(\hat{s}),n(\hat{t}),n([\hat{s},\hat{t}]))$ is a permissible triple, and $\mathcal{P}(X,\omega)\subset\mathcal{P}(\text{Sims}_G)$.

Now suppose that $G$ is finite.  Then the set $\text{Sims}_G$ of all simulations is finite, and by Lemma \mbox{\ref{permissible_triples_finite}}, the set $\mathcal{P}(\text{Sims}_G)$ is finite.  Hence for each $(p,q,u)\in \mathcal{P}(\text{Sims}_G)$ there is some $n$ for which $(p,q,u)\in \mathcal{P}(\text{Sims}^n_G)$.  Taking $N$ to be the maximum of these, we have $\mathcal{P}(X,\omega)\subset\mathcal{P}(\text{Sims}_G)\subset \mathcal{P}(\text{Sims}^N_G)$.
\end{proof}

\begin{figure}[t]
\begin{minipage}{.49\textwidth}
\centering
\begin{tikzpicture}[scale=1.2]
\draw[->,line width=1,dashed] (0,0) -- (2.2,0);
\draw[->,line width=1] (0,0) -- (-2.2,0);
\draw[<->,line width=1] (0,-1) -- (0,2.2);
\draw (2,2) circle (0pt) node {$c_\ell^i$};
\draw[dashed] ({cos(120)},{sin(120)}) circle (1);
\filldraw ({cos(120)+cos(330)},{sin(120)+sin(330)}) circle (1pt) node[right] {$\hat{s}$};
\filldraw ({cos(120)+cos(170)},{sin(120)+sin(170)}) circle (1pt) node[left] {$\hat{t}$};
\draw[dashed] (0,0) -- ({cos(120)+cos(330)},{sin(120)+sin(330)}) -- ({cos(120)+cos(170)},{sin(120)+sin(170)}) -- cycle;
\filldraw ({cos(120)+.5*cos(60)},{sin(120)+.5*sin(60)}) circle (1pt) node[left] {$[\hat{s}',\hat{u}]$};
\end{tikzpicture}
\end{minipage}
\begin{minipage}{.49\textwidth}
\centering
\begin{tikzpicture}[scale=1.2]
\draw[->,line width=1,dashed] (0,0) -- (2.2,0);
\draw[->,line width=1] (0,0) -- (-2.2,0);
\draw[<->,line width=1] (0,-1) -- (0,2.2);
\draw (2,2) circle (0pt) node {$c_m^j$};
\draw[dashed] ({-cos(330)},{-sin(330)}) circle (1);
\filldraw ({-cos(120)-cos(330)},{-sin(120)-sin(330)}) circle (1pt) node[below] {$\hat{s}'$};
\filldraw ({cos(170)-cos(330)},{sin(170)-sin(330)}) circle (1pt) node[left] {$[\hat{s},\hat{t}]$};
\draw[dashed] (0,0) -- ({-cos(120)-cos(330)},{-sin(120)-sin(330)}) -- ({cos(170)-cos(330)},{sin(170)-sin(330)}) -- cycle;
\filldraw ({-cos(330)+.5*cos(60)},{-sin(330)+.5*sin(60)}) circle (1pt) node[left] {$\hat{u}$};
\end{tikzpicture}
\end{minipage}
\caption{Illustration of the proof of Proposition \ref{vor_staples_det_perm_triples_prop}.}
\label{marked_segs_and_balls}
\end{figure}
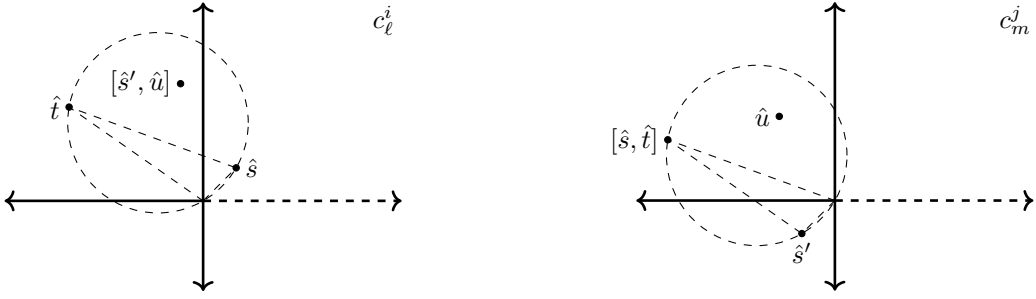

\subsection{Finiteness of lattice surfaces with given Veech groups}\label{Finiteness of lattice surfaces with given Veech groups}

Throughout this subsection, let $G$ be a set of generators of a fanning group $\Gamma$.  As noted at the end of \S\ref{Simulating orbits}, to prove Theorem \ref{finiteness_thm} it suffices to show that for any finite collection of simulations determined by $G$, there are at most finitely many scalars for which the union of scaled simulations coincides with the marked segments of a unit-area translation surface.  Propositions \ref{scalars_unique}, \ref{permute_permissible_triples} and $\ref{vor_staples_det_perm_triples_prop}$ suggest that for adjacent elements $\hat{s},\hat{t}\in\widehat{\mathcal{S}}_F(X,\omega)$ of a translation surface $(X,\omega)$ with Veech group $\Gamma$, the desired scalars for the simulations corresponding to $\hat{s},\hat{t}$ and $[\hat{s},\hat{t}]$ are uniquely determined by one another.  Hence if the elements of $\widehat{\mathcal{S}}_F(X,\omega)$ are all---loosely speaking---`transitively adjacent' to one another (or their orientation-paired inverses), then an arbitrary choice of one scalar uniquely determines all other scalars.  

We make these vague notions precise by introducing a \textit{permissible triples graph} corresponding to a set $P$ of permissible triples whose vertices are the distinctive, limit-point free sets corresponding to the triples and whose edges are defined between the vertices corresponding to entries of the same permissible triple in $P$.  The `transitive adjacency' alluded to above refers to the connectedness of such a graph.

\begin{defn}\label{simulation_graph}
Let ${\bf Q}=\{Q_i\}_{i\in I}$ be a collection of distinctive, limit-point free subsets of $\mathbb{P}(\mathcal{O})$.  For a subset $P\subset\mathcal{P}({\bf Q})$ of permissible triples, define the \textit{permissible triples graph} determined by $P$ to be the graph $\mathfrak{G}_P=(\mathfrak{V}_P,\mathfrak{E}_P)$ with vertex set
\[\mathfrak{V}_P=\{Q_i,Q_j,Q_k\ |\ \text{there exists some}\ (q_i,q_j,q_k)\in P\cap \left((Q_i)_F\times (Q_j)_F\times (Q_k)_F \right)\}\]
consisting of the sets in ${\bf Q}$ corresponding to the permissible triples in $P$, and edge set
\[\mathfrak{E}_P=\{\{Q_\ell,Q_m\}\ |\ \ell,m\in\{i,j,k\}\ \text{for some}\ (q_i,q_j,q_k)\in P\cap \left((Q_i)_F\times (Q_j)_F\times (Q_k)_F \right)\}\]
where an edge connects vertices $Q_\ell$ and $Q_m$ if there is some permissible triple in $P$ to which both $Q_\ell$ and $Q_m$ correspond.  We call $(r_i)_{i\in I}\in\mathbb{R}_+^{|I|}$ \textit{consistent scalars} for the graph $\mathfrak{G}_P$ if $(r_i,r_j,r_k)$ are permissible scalars for each permissible triple $(q_i,q_j,q_k)\in P\cap \left((Q_i)_F\times (Q_j)_F\times (Q_k)_F \right)$.  We call the graph $\mathfrak{G}_P$ \textit{consistent} if it admits consistent scalars.
\end{defn}

\begin{eg}\label{simulations_example2}
Consider again Example \ref{simulations_example} and Figure \ref{simulations_figure}; let $p\in\text{sim}(\mathfrak{s}_0)$ be the point with argument $3\pi/2$, $q\in\text{sim}(\mathfrak{s}_1)$ the point with argument $\pi$ and $u\in\text{sim}(\mathfrak{s}_2)$ the point with argument $3\pi/4$.  Figure \ref{simulations_figure2} shows that $(p,q,u)\in\mathcal{P}(\text{Sims}_G)$ is a permissible triple with permissible scalars $(1,1,\sqrt{2})$.  Let $P:=\{(p,q,u)\}$.  Then $\mathfrak{G}_P$ is the graph with vertex set $\mathfrak{V}_P=\{\text{sim}(\mathfrak{s}_0),\text{sim}(\mathfrak{s}_1),\text{sim}(\mathfrak{s}_2)\}$ and edge set 
\[\mathfrak{E}_P=\{\{\text{sim}(\mathfrak{s}_0),\text{sim}(\mathfrak{s}_1)\},\{\text{sim}(\mathfrak{s}_1),\text{sim}(\mathfrak{s}_2)\},\{\text{sim}(\mathfrak{s}_2),\text{sim}(\mathfrak{s}_0)\}\},\] 
and $(1,1,\sqrt{2})$ are consistent scalars for $\mathfrak{G}_P$.
\end{eg}

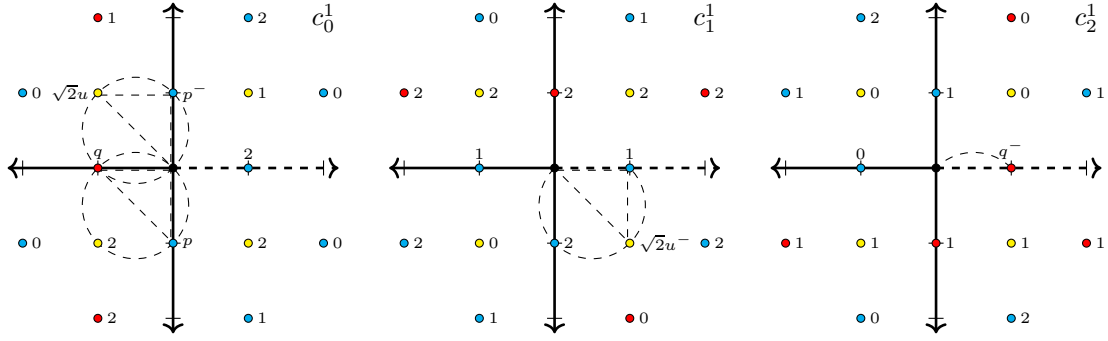
\begin{figure}[t]
\centering
\begin{minipage}{.3\textwidth}
\begin{tikzpicture}
\draw [->,dashed,line width=1] (0,0) -- (2.2,0);
\draw [<-,line width=1] (-2.2,0) -- (0,0);
\draw [<->,line width=1] (0,-2.2) -- (0,2.2); 
\draw (2,2) circle (0pt) node {$c_0^1$};
\draw (1,-.1) -- (1,.1);
\draw (2,-.1) -- (2,.1);
\draw (-1,-.1) -- (-1,.1);
\draw (-2,-.1) -- (-2,.1);
\draw (-.1,1) -- (.1,1);
\draw (-.1,2) -- (.1,2);
\draw (-.1,-1) -- (.1,-1);
\draw (-.1,-2) -- (.1,-2);
\filldraw (0,0) circle (1.5pt);
\draw[dashed] (-1/2,-1/2) circle ({sqrt(2)/2});
\draw[dashed] (-1/2,1/2) circle ({sqrt(2)/2});
\draw[dashed] (-0.03,-.03) -- (-.97,-.03) -- (-0.03,-.97) -- cycle;
\draw[dashed] (-0.03,.97) -- (-.97,.97) -- (-0.03,.03) -- cycle;
\filldraw[fill=cyan] (1,0) circle (1.5pt) node[above] {\tiny $2$};
\filldraw[fill=cyan] (0,-1) circle (1.5pt) node[right] {\tiny $p$};
\filldraw[fill=cyan] (0,1) circle (1.5pt) node[right] {\tiny $p^-$};
\filldraw[fill=cyan] (-2,-1) circle (1.5pt) node[right] {\tiny $0$};
\filldraw[fill=cyan] (2,1) circle (1.5pt) node[right] {\tiny $0$};
\filldraw[fill=cyan] (2,-1) circle (1.5pt) node[right] {\tiny $0$};
\filldraw[fill=cyan] (-2,1) circle (1.5pt) node[right] {\tiny $0$};
\filldraw[fill=cyan] (1,2) circle (1.5pt) node[right] {\tiny $2$};
\filldraw[fill=cyan] (1,-2) circle (1.5pt) node[right] {\tiny $1$};
\filldraw[fill=red] (-1,0) circle (1.5pt) node[above] {\tiny $q$};
\filldraw[fill=red] (-1,-2) circle (1.5pt) node[right] {\tiny $2$};
\filldraw[fill=red] (-1,2) circle (1.5pt) node[right] {\tiny $1$};
\filldraw[fill=yellow] (1,1) circle (1.5pt) node[right] {\tiny $1$};
\filldraw[fill=yellow] ({-1},{1}) circle (1.5pt) node[left] {\tiny $\sqrt2 u$};
\filldraw[fill=yellow] ({-1},{-1}) circle (1.5pt) node[right] {\tiny $2$};
\filldraw[fill=yellow] ({1},{-1}) circle (1.5pt) node[right] {\tiny $2$};
\end{tikzpicture}
\end{minipage}
\begin{minipage}{.3\textwidth}
\begin{tikzpicture}
\draw [->,dashed,line width=1] (0,0) -- (2.2,0);
\draw [<-,line width=1] (-2.2,0) -- (0,0);
\draw [<->,line width=1] (0,-2.2) -- (0,2.2);
\draw (2,2) circle (0pt) node {$c_1^1$};
\draw (1,-.1) -- (1,.1);
\draw (2,-.1) -- (2,.1);
\draw (-1,-.1) -- (-1,.1);
\draw (-2,-.1) -- (-2,.1);
\draw (-.1,1) -- (.1,1);
\draw (-.1,2) -- (.1,2);
\draw (-.1,-1) -- (.1,-1);
\draw (-.1,-2) -- (.1,-2);
\filldraw (0,0) circle (1.5pt);
\draw[dashed] (0,0) arc (135:405:{sqrt(2)/2});
\draw[dashed] (.97,-.03) -- (.03,-.03) -- (.97,-.97) -- cycle;
\filldraw[fill=cyan] (0,-1) circle (1.5pt) node[right] {\tiny $2$};
\filldraw[fill=cyan] (-2,-1) circle (1.5pt) node[right] {\tiny $2$};
\filldraw[fill=cyan] (2,-1) circle (1.5pt) node[right] {\tiny $2$};
\filldraw[fill=cyan] (-1,2) circle (1.5pt) node[right] {\tiny $0$};
\filldraw[fill=cyan] (1,0) circle (1.5pt) node[above] {\tiny $1$};
\filldraw[fill=cyan] (-1,0) circle (1.5pt) node[above] {\tiny $1$};
\filldraw[fill=cyan] (1,2) circle (1.5pt) node[right] {\tiny $1$};
\filldraw[fill=cyan] (-1,-2) circle (1.5pt) node[right] {\tiny $1$};
\filldraw[fill=red] (0,1) circle (1.5pt) node[right] {\tiny $2$};
\filldraw[fill=red] (2,1) circle (1.5pt) node[right] {\tiny $2$};
\filldraw[fill=red] (-2,1) circle (1.5pt) node[right] {\tiny $2$};
\filldraw[fill=red] (1,-2) circle (1.5pt) node[right] {\tiny $0$};
\filldraw[fill=yellow] ({1},{1}) circle (1.5pt) node[right] {\tiny $2$};
\filldraw[fill=yellow] ({-1},{1}) circle (1.5pt) node[right] {\tiny $2$};
\filldraw[fill=yellow] ({-1},{-1}) circle (1.5pt) node[right] {\tiny $0$};
\filldraw[fill=yellow] ({1},{-1}) circle (1.5pt) node[right] {\tiny $\sqrt{2}u^-$};
\end{tikzpicture}
\end{minipage}
\begin{minipage}{.3\textwidth}
\begin{tikzpicture}
\draw [->,dashed,line width=1] (0,0) -- (2.2,0);
\draw [<-,line width=1] (-2.2,0) -- (0,0);
\draw [<->,line width=1] (0,-2.2) -- (0,2.2);
\draw (2,2) circle (0pt) node {$c_2^1$};
\draw (1,-.1) -- (1,.1);
\draw (2,-.1) -- (2,.1);
\draw (-1,-.1) -- (-1,.1);
\draw (-2,-.1) -- (-2,.1);
\draw (-.1,1) -- (.1,1);
\draw (-.1,2) -- (.1,2);
\draw (-.1,-1) -- (.1,-1);
\draw (-.1,-2) -- (.1,-2);
\filldraw (0,0) circle (1.5pt);
\draw[dashed] (1,0) arc (45:135:{sqrt(2)/2});
\filldraw[fill=cyan] (-1,0) circle (1.5pt) node[above] {\tiny $0$};
\filldraw[fill=cyan] (0,1) circle (1.5pt) node[right] {\tiny $1$};
\filldraw[fill=cyan] (2,1) circle (1.5pt) node[right] {\tiny $1$};
\filldraw[fill=cyan] (-2,1) circle (1.5pt) node[right] {\tiny $1$};
\filldraw[fill=cyan] (-1,-2) circle (1.5pt) node[right] {\tiny $0$};
\filldraw[fill=cyan] (1,-2) circle (1.5pt) node[right] {\tiny $2$};
\filldraw[fill=cyan] (-1,2) circle (1.5pt) node[right] {\tiny $2$};
\filldraw[fill=red] (1,0) circle (1.5pt) node[above] {\tiny $q^-$};
\filldraw[fill=red] (0,-1) circle (1.5pt) node[right] {\tiny $1$};
\filldraw[fill=red] (-2,-1) circle (1.5pt) node[right] {\tiny $1$};
\filldraw[fill=red] (2,-1) circle (1.5pt) node[right] {\tiny $1$};
\filldraw[fill=red] (1,2) circle (1.5pt) node[right] {\tiny $0$};
\filldraw[fill=yellow] ({1},{1}) circle (1.5pt) node[right] {\tiny $0$};
\filldraw[fill=yellow] ({-1},{1}) circle (1.5pt) node[right] {\tiny $0$};
\filldraw[fill=yellow] ({-1},{-1}) circle (1.5pt) node[right] {\tiny $1$};
\filldraw[fill=yellow] ({1},{-1}) circle (1.5pt) node[right] {\tiny $1$};
\end{tikzpicture}
\end{minipage}
\caption{Subsets of the scaled simulations $\text{sim}(\mathfrak{s}_0),\ \text{sim}(\mathfrak{s}_1)$ and $\sqrt{2}\text{sim}(\mathfrak{s}_2)$, with $\mathfrak{s}_0,\ \mathfrak{s}_1\ \mathfrak{s}_2\in\mathfrak{S}_{G}$ as defined in Example \ref{simulations_example}, and balls showing that $(p,q,u)\in\mathcal{P}(\text{Sims}_{G})$ is a permissible triple with permissible scalars $(1,1,\sqrt{2})$.}
\label{simulations_figure2}
\end{figure}

\begin{prop}\label{sim_graph_for_P(X,omega)_consistent}
Let $(X,\omega)$ be a fanning surface with Veech group $\mathrm{SL}(X,\omega)=\Gamma$, and let $P:=\mathcal{P}(X,\omega)\subset\mathcal{P}(\text{Sims}_G)$.  Then the permissible triples graph $\mathfrak{G}_P$ is connected and consistent, and if $(r_\mathfrak{s})_{\text{sim}(\mathfrak{s})\in\mathfrak{V}_P}$ are consistent scalars for $\mathfrak{G}_P$, then any other consistent scalars are of the form $(a r_\mathfrak{s})_{\text{sim}(\mathfrak{s})\in\mathfrak{V}_P}$ with $a\in\mathbb{R}_+$.
\end{prop}
\begin{proof}
Theorem \ref{marked_segments_as_scaled_simulations} implies that the vertex set $\mathfrak{V}_P$ is finite, say $\mathfrak{V}_P=\{\text{sim}(\mathfrak{s}_i)\}_{i=1}^n$.  The same theorem also gives that
\[\bigcup_{i=1}^n r_i\text{sim}(\mathfrak{s}_i)\subset\widehat{\mathcal{M}}(X,\omega)\]
for some $(r_i)_{i=1}^n\in\mathbb{R}_+^n$.  In fact, each $r_i$ is the minimal length of a pair in the $\text{Aff}_\mathcal{O}^+(X,\omega)$-orbit for which $r_i\text{sim}(\mathfrak{s}_i)=[\{\hat{s},\hat{s}'\}]$.  The proof of Proposition \ref{vor_staples_det_perm_triples_prop} shows that the scalars $(r_i)_{i=1}^n\in\mathbb{R}_+^n$ are consistent, so the graph $\mathfrak{G}_P$ is consistent.

Now suppose that $(r_i')_{i=1}^n\in\mathbb{R}_+^n$ are also consistent scalars for $\mathfrak{G}_P$, and recall the definition of the edge set $\mathfrak{E}_P$.  Induction with Propositions \ref{scalars_unique} and \ref{permute_permissible_triples} implies that any of the scalars $r_i'$ corresponding to a simulation of a connected component of $\mathfrak{G}_P$ uniquely determines the remaining scalars $r_j'$ of the simulations of the component.  In particular, if $r_i'=a r_i$ for some $a\in\mathbb{R}_+$, then $r_j'=a r_j$ for each of these scalars.  Hence it suffices to show that $\mathfrak{G}_P$ is connected.  

Consider the subgraph $(V,E)$ of $(\mathfrak{V}_P,\mathfrak{E}_P)$ with vertex set $V$ consisting of all simulations in $\mathfrak{V}_P$ which contain a pair of normalized marked Voronoi staples and edge set $E$ consisting of edges $\{\text{sim}(\mathfrak{s}_i),\text{sim}(\mathfrak{s}_j)\}$, where $\text{sim}(\mathfrak{s}_i)$ and $\text{sim}(\mathfrak{s}_j)$ contain $\{n(\hat{s}),n(\hat{s})'\}$ and $\{n(\hat{t}),n(\hat{t})'\}$, respectively, for adjacent $\hat{s},\hat{t}\in\widehat{\mathcal{S}}_F(X,\omega)$.  Recall that $(X,\omega)$ is isometric to the quotient space of $\bigsqcup_{i=1}^\kappa \Omega_i$ under the equivalence relation given by identifying the two edges of convex bodies which correspond to a pair of marked Voronoi staples (\S\ref{Marked segments and Voronoi staples} above, and also \S3 of \cite{ESS}).  If $(V,E)$ were disconnected, then after this identification of edges of convex bodies the resulting surface---and hence $(X,\omega)$---would be disconnected.  This is a contradiction, so the subgraph $(V,E)$ is connected.  By definition of $P=\mathcal{P}(X,\omega)$ and $\mathfrak{V}_P$, any vertex $\text{sim}(\mathfrak{s}_i)\in\mathfrak{V}_P\backslash V$ must contain $\{n([\hat{s},\hat{t}]),n([\hat{s},\hat{t}])'\}$ for some adjacent $\hat{s},\hat{t}\in\widehat{\mathcal{S}}_F(X,\omega)$.  By definition of $\mathfrak{E}_P$, there are edges of $\mathfrak{G}_P$ from $\text{sim}(\mathfrak{s}_i)$ to the simulations containing $\{n(\hat{s}),n(\hat{s})'\}$ and $\{n(\hat{t}),n(\hat{t})'\}$, both of which belong to $V\subset \mathfrak{V}_P$.  Thus $\mathfrak{G}_P$ is connected, and the result follows.
\end{proof}

We are now ready to prove Theorem \ref{finiteness_thm}:

\begin{proof}[Proof of Theorem \ref{finiteness_thm}]
Let $\Gamma\le\mathrm{SL}_2\mathbb{R}$ be a lattice.  By Lemma \ref{lattice_and_fanning}, $\Gamma$ is a finitely generated fanning group.  Let $G$ be a finite set of generators of $\Gamma$ and fix $\mathcal{H}_1(d_1,\dots,d_\kappa)$.  We must show that the set of all $(X,\omega)\in\mathcal{H}_1(d_1,\dots,d_\kappa)$ for which $\mathrm{SL}(X,\omega)=\Gamma$ is finite.  By Proposition \ref{vor_staples_det_perm_triples_prop}, each of these translation surfaces belongs to the set
\[\{(X,\omega)\in\mathcal{H}_1(d_1,\dots,d_\kappa)\ |\ \mathcal{P}(X,\omega)\subset\mathcal{P}(\text{Sims}_G)\},\]
so it suffices to show that this latter set is finite.  By Lemmas \ref{permissible_triples_finite} and \ref{Sims_G_finite}, the set $\mathcal{P}(\text{Sims}_G)$ of all permissible triples is finite, so it suffices to show for any $P\subset\mathcal{P}(\text{Sims}_G)$ that the set
\[\{(X,\omega)\in\mathcal{H}_1(d_1,\dots,d_\kappa)\ |\ \mathcal{P}(X,\omega)=P\}\]
is finite.  Let $\mathfrak{V}_P=\{\text{sim}(\mathfrak{s}_i)\}_{i=1}^n$ be the vertex set of the permissible triples graph $\mathfrak{G}_P$.  As in the proof of Proposition \ref{sim_graph_for_P(X,omega)_consistent}, for each $(X,\omega)$ with $\mathcal{P}(X,\omega)=P$, there are consistent scalars $(r_i)_{i=1}^n\in\mathbb{R}_+^n$ for which 
\[\bigcup_{i=1}^n r_i\text{sim}(\mathfrak{s}_i)\subset\widehat{\mathcal{M}}(X,\omega).\]
By definition of $\mathfrak{G}_P$ and its vertex set $\mathfrak{V}_P$, we have that the set of marked Voronoi staples, $\widehat{\mathcal{S}}(X,\omega)$---from which $(X,\omega)$ may be recovered---is contained in this union of scaled simulations.  Again from Proposition \ref{sim_graph_for_P(X,omega)_consistent}, if $(Y,\eta)\in\mathcal{H}_1(d_1,\dots,d_\kappa)$ also satisfies $\mathcal{P}(Y,\eta)=P$, then its corresponding consistent scalars are of the form $(a r_i)_{i=1}^n$ for some $a\in\mathbb{R}_+^n$.  From \S\ref{Marked segments and Voronoi staples}, we have
\[\bigsqcup_{j=1}^\kappa \Omega_j(\widehat{\mathcal{S}}_F(X,\omega))=\bigsqcup_{j=1}^\kappa \Omega_j\left(\bigcup_{i=1}^n r_i\text{sim}_F(\mathfrak{s}_i)\right)\]
and
\[\bigsqcup_{j=1}^\kappa \Omega_j(\widehat{\mathcal{S}}_F(Y,\eta))=\bigsqcup_{j=1}^\kappa \Omega_j\left(\bigcup_{i=1}^n a r_i\text{sim}_F(\mathfrak{s}_i)\right)=a \left(\bigsqcup_{j=1}^\kappa \Omega_j\left(\bigcup_{i=1}^n r_i\text{sim}_F(\mathfrak{s}_i)\right)\right).\]
Since $(X,\omega)$ and $(Y,\eta)$ are isometric to the quotient spaces of these respective sets under an equivalence relation identifying edges, and since both translation surfaces have area one, we conclude that $a=1$ and thus $(X,\omega)=(Y,\eta)$. 
\end{proof}

\section{The algorithm}\label{The algorithm}

The constructive results of the previous sections suggest an algorithm (Algorithm \ref{the_algorithm}) which returns all translation surfaces with a given lattice Veech group $\Gamma$ in any given stratum $\mathcal{H}_1(d_1,\dots,d_\kappa)$.  We highlight the main steps of this algorithm here, providing necessary references and further commentary in italics.  Recall that $G$ is a given finite set of generators of the lattice $\Gamma\le\mathrm{SL}_2\mathbb{R}$.

\begin{enumerate}
\item[(1)]  Compute a finite set $\Theta_{\Gamma_m}\subset S^1$ containing $\Theta_\Gamma$, and let 
\[\mathfrak{S}_G:=\{(H,\{p,p^-\})\in 2^{G_\mathcal{O}}\backslash\{\varnothing\}\times \mathbb{P}(\mathcal{O})\ |\ {\bm\pi}(p),{\bm\pi}(p^-)\in\Theta_{\Gamma_m}\}.\]
(\textit{See Definition \mbox{\ref{simulations_defn}} and the remark following Lemma \mbox{\ref{Sims_G_finite}}.})

\item[(2)]  For each $n\in\mathbb{N}$:
\begin{enumerate}
\item[(a)]  Compute the set $\text{Sims}^n_G=\{\text{sim}^n(\mathfrak{s})\ |\ \mathfrak{s}\in\mathfrak{S}_G\}$ of all stage $n$ distinctive simulations determined by $G$. \\
(\textit{See Defintion \mbox{\ref{simulations_defn}}.})

\item[(b)] Compute the set $\mathcal{P}(\text{Sims}^n_G)$ of all permissible triples arising from $\text{Sims}^n_G$ along with their respective rays of permissible scalars. \\
(\textit{See Definition \mbox{\ref{permissible_triples_defn}} and Proposition \mbox{\ref{scalars_unique}}.  Recall from Definition \mbox{\ref{set_of_permissible_triples}} and Proposition \mbox{\ref{vor_staples_det_perm_triples_prop}} that there is some $N\in\mathbb{N}$ so that for any $(Y,\eta)\in\mathcal{H}_1(d_1,\dots,d_\kappa)$ with $\mathrm{SL}(Y,\eta)=\Gamma$, the set $\mathcal{P}(Y,\eta)$ determined by the marked Voronoi staples $\widehat{\mathcal{S}}(Y,\eta)$ is contained in the finite set of permissible triples $\mathcal{P}(\text{Sims}^N_G)$.})

\item[(c)] For each $P\subset\mathcal{P}(\text{Sims}^n_G)\backslash\{\varnothing\}$:\\
(\textit{For $n=N$ and any $(Y,\eta)$ as above, this loop will set $P=\mathcal{P}(Y,\eta)$ for some $P$.}) 
\begin{enumerate}
\item[(i)] Construct the permissible triples graph $\mathfrak{G}_P=(\mathfrak{V}_P,\mathfrak{E}_P)$.  If $\mathfrak{G}_P$ is connected and consistent with consistent scalars $(r_{\mathfrak{s}})_{\text{sim}^n(\mathfrak{s})\in\mathfrak{V}^n_P}$, then set
\[S:=\bigcup_{\text{sim}^n(\mathfrak{s})\in\mathfrak{V}_P}r_{\mathfrak{s}}\text{sim}^n(\mathfrak{s}),\]
construct the convex bodies $\Omega_i(S_F)$ and determine the corresponding essential points $\mathcal{E}_i(S_F)$ for each $i\in\{1,\dots,\kappa\}$.\\
(\textit{See Definition \mbox{\ref{simulation_graph}} and \S\mbox{\ref{Marked segments and Voronoi staples}}.})

\item[(ii)] If each $\Omega_i(S_F)$ is compact, the essential points $\sqcup_{i=1}^\kappa\mathcal{E}_i(S_F)$ come in pairs $\{p,p^-\}$ determined by the simulations $\text{sim}^n(\mathfrak{s})\in\mathfrak{V}_P$, and the edges of the convex bodies determined by $p$ and $p^-$ are equal length, then construct the translation surface $(X,\omega)$ by identifying these edges via translation.\\
(\textit{See \S\mbox{\ref{Marked segments and Voronoi staples}}.  When $P=\mathcal{P}(Y,\eta)$, the proof of Proposition \mbox{\ref{sim_graph_for_P(X,omega)_consistent}} shows that for some $a\in\mathbb{R}_+$, the set $aS$ contains the set of marked Voronoi staples $\widehat{\mathcal{S}}(Y,\eta)$, and thus the union of scaled essential points $\sqcup_{i=1}^\kappa a\mathcal{E}_i(S_F)$ is precisely $\widehat{\mathcal{S}}_F(Y,\eta)$.  Hence $(X,\omega)$ is a scaled version of $(Y,\eta)$.})

\item[(iii)] Rescale $(X,\omega)$ if necessary so that it has area one, and verify that $\mathrm{SL}(X,\omega)=\Gamma$.\\
(\textit{Computing $\text{SL}(X,\omega)$ is done using one of the algorithms mentioned in \S\mbox{\ref{Introduction}}.})
\end{enumerate}
\end{enumerate}
\end{enumerate}

\begin{eg}\label{simulations_example3}
Examples \ref{simulations_example} and \ref{simulations_example2} illustrate (parts of) Steps 1, 2.a, 2.b and 2.c.i of the algorithm.  Figure \ref{simulations_figure3} illustrates the construction of convex bodies and corresponding essential points in Step 2.c.i as well as the resulting translation surface $(X,\omega)$ constructed in Step 2.c.ii.  Notice that the essential points come only from $\text{sim}(\mathfrak{s}_0)$ and $\text{sim}(\mathfrak{s}_1)$; however, $\text{sim}(\mathfrak{s}_2)$ was necessary to determine the permissible scalars $(1,1,\sqrt{2})$ (Example \ref{simulations_example2}).  Rescaling to unit-area, one verifies that indeed $\mathrm{SL}(X,\omega)=\langle S,T^2\rangle$.
\end{eg}

\begin{figure}[t]
\centering
\begin{minipage}{.24\textwidth}
\begin{tikzpicture}
\filldraw[dashed,fill=black!30] (1/2,1/2) -- (-1/2,1/2) -- (-1/2,-1/2) -- (1/2,-1/2) -- cycle;
\draw [->,dashed,line width=1] (0,0) -- (1.5,0);
\draw [<-,line width=1] (-1.5,0) -- (0,0);
\draw [<->,line width=1] (0,-1.5) -- (0,1.5);
\draw (1.4,1.4) circle (0pt) node {$c_0^1$};
\draw (1,-.1) -- (1,.1);
\draw (-1,-.1) -- (-1,.1);
\draw (-.1,1) -- (.1,1);
\draw (-.1,-1) -- (.1,-1);
\filldraw (0,0) circle (1.5pt);
\draw (0,0) circle (.1);
\filldraw[fill=cyan] (1,0) circle (1.5pt) node[above] {\tiny $2$};
\filldraw[fill=cyan] (0,-1) circle (1.5pt) node[right] {\tiny $0$};
\filldraw[fill=cyan] (0,1) circle (1.5pt) node[right] {\tiny $0$};
\filldraw[fill=red] (-1,0) circle (1.5pt) node[above] {\tiny $2$};
\filldraw[fill=yellow] (1,1) circle (1.5pt) node[right] {\tiny $1$};
\filldraw[fill=yellow] ({-1},{1}) circle (1.5pt) node[right] {\tiny $1$};
\filldraw[fill=yellow] ({-1},{-1}) circle (1.5pt) node[right] {\tiny $2$};
\filldraw[fill=yellow] ({1},{-1}) circle (1.5pt) node[right] {\tiny $2$};
\end{tikzpicture}
\end{minipage}
\begin{minipage}{.24\textwidth}
\begin{tikzpicture}
\filldraw[dashed,fill=black!30] (1/2,1/2) -- (-1/2,1/2) -- (-1/2,-1/2) -- (1/2,-1/2) -- cycle;
\draw [->,dashed,line width=1] (0,0) -- (1.5,0);
\draw [<-,line width=1] (-1.5,0) -- (0,0);
\draw [<->,line width=1] (0,-1.5) -- (0,1.5);
\draw (1.4,1.4) circle (0pt) node {$c_1^1$};
\draw (1,-.1) -- (1,.1);
\draw (-1,-.1) -- (-1,.1);
\draw (-.1,1) -- (.1,1);
\draw (-.1,-1) -- (.1,-1);
\filldraw (0,0) circle (1.5pt);
\draw (0,0) circle (.1);
\draw (0,0) circle (.15);
\filldraw[fill=cyan] (0,-1) circle (1.5pt) node[right] {\tiny $2$};
\filldraw[fill=cyan] (1,0) circle (1.5pt) node[above] {\tiny $1$};
\filldraw[fill=cyan] (-1,0) circle (1.5pt) node[above] {\tiny $1$};
\filldraw[fill=red] (0,1) circle (1.5pt) node[right] {\tiny $2$};
\filldraw[fill=yellow] ({1},{1}) circle (1.5pt) node[right] {\tiny $2$};
\filldraw[fill=yellow] ({-1},{1}) circle (1.5pt) node[right] {\tiny $2$};
\filldraw[fill=yellow] ({-1},{-1}) circle (1.5pt) node[right] {\tiny $0$};
\filldraw[fill=yellow] ({1},{-1}) circle (1.5pt) node[right] {\tiny $0$};
\end{tikzpicture}
\end{minipage}
\begin{minipage}{.24\textwidth}
\begin{tikzpicture}
\filldraw[dashed,fill=black!30] (1/2,1/2) -- (-1/2,1/2) -- (-1/2,-1/2) -- (1/2,-1/2) -- cycle;
\draw [->,dashed,line width=1] (0,0) -- (1.5,0);
\draw [<-,line width=1] (-1.5,0) -- (0,0);
\draw [<->,line width=1] (0,-1.5) -- (0,1.5);
\draw (1.4,1.4) circle (0pt) node {$c_2^1$};
\draw (1,-.1) -- (1,.1);
\draw (-1,-.1) -- (-1,.1);
\draw (-.1,1) -- (.1,1);
\draw (-.1,-1) -- (.1,-1);
\filldraw (0,0) circle (1.5pt);
\draw (0,0) circle (.1);
\draw (0,0) circle (.15);
\draw (0,0) circle (.2);
\filldraw[fill=cyan] (-1,0) circle (1.5pt) node[above] {\tiny $0$};
\filldraw[fill=cyan] (0,1) circle (1.5pt) node[right] {\tiny $1$};
\filldraw[fill=red] (1,0) circle (1.5pt) node[above] {\tiny $0$};
\filldraw[fill=red] (0,-1) circle (1.5pt) node[right] {\tiny $1$};
\filldraw[fill=yellow] ({1},{1}) circle (1.5pt) node[right] {\tiny $0$};
\filldraw[fill=yellow] ({-1},{1}) circle (1.5pt) node[right] {\tiny $0$};
\filldraw[fill=yellow] ({-1},{-1}) circle (1.5pt) node[right] {\tiny $1$};
\filldraw[fill=yellow] ({1},{-1}) circle (1.5pt) node[right] {\tiny $1$};
\end{tikzpicture}
\end{minipage}
\begin{minipage}{.24\textwidth}
\begin{tikzpicture}
\filldraw[fill=black!30] (-1,-1) -- (1,-1) -- (1,0) -- (0,0) -- (0,1) -- (-1,1) -- cycle;
\filldraw (-1,-1) circle (1.5pt);
\filldraw (0,-1) circle (1.5pt);
\filldraw (1,-1) circle (1.5pt);
\filldraw (1,0) circle (1.5pt);
\filldraw (0,0) circle (1.5pt);
\filldraw (0,1) circle (1.5pt);
\filldraw (-1,1) circle (1.5pt);
\filldraw (-1,0) circle (1.5pt);
\draw[dashed] (-1/2,-1) -- (-1/2,1);
\draw[dashed] (1/2,-1) -- (1/2,0);
\draw[dashed] (-1,-1/2) -- (1,-1/2);
\draw[dashed] (-1,1/2) -- (0,1/2);
\draw (-.9,-1) arc (0:90:.1);
\draw (1,-.9) arc (90:180:.1);
\draw (.9,0) arc (180:270:.1);
\draw (-1,-.1) arc (270:450:.1);
\draw (-.85,0) arc (0:90:.15);
\draw (0,.1) arc (90:360:.1);
\draw (0,.15) arc (90:360:.15);
\draw (.1,-1) arc (0:180:.1);
\draw (.15,-1) arc (0:180:.15);
\draw (.2,-1) arc (0:180:.2);
\draw (-.1,1) arc (180:270:.1);
\draw (-.15,1) arc (180:270:.15);
\draw (-.2,1) arc (180:270:.2);
\draw (-1,.9) arc (270:360:.1);
\draw (-1,.85) arc (270:360:.15);
\draw (-1,.8) arc (270:360:.2);
\end{tikzpicture}
\end{minipage}
\caption{The convex body and resulting translation surface $(X,\omega)$ (with opposite sides identified) using the scaled simulations $\text{sim}(\mathfrak{s}_0),\ \text{sim}(\mathfrak{s}_1)$ and $\sqrt{2}\text{sim}(\mathfrak{s}_2)$ from Examples \ref{simulations_example} and \ref{simulations_example2}.  The arcs about the singularity of $\mathcal{O}$ and of $(X,\omega)$ indicate arguments of generalized polar coordinates.}
\label{simulations_figure3}
\end{figure}
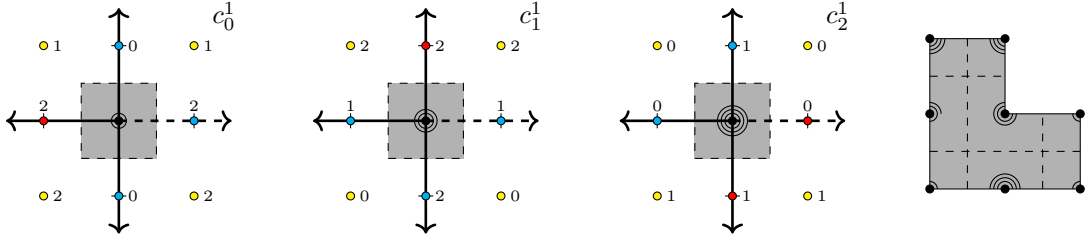

A number of improvements to Algorithm \mbox{\ref{the_algorithm}} can be made for computational efficiency.  We mention three in particular:
\begin{itemize}
\item  It is optimal to choose a minimal set of generators $G$ for $\Gamma$, as the size of $\mathfrak{S}_G$---and hence of $\text{Sims}_G$---depends on $|G|$.

\item  Recall from Proposition \mbox{\ref{marked_segs_distinctive_and_lpf}} that if $\hat{s},\hat{t}\in\widehat{\mathcal{M}}_F(X,\omega)$ are two marked segments with identical arguments in the same component $\mathcal{O}_i\subset \mathcal{O}$, then in fact $\{\hat{s},\hat{s}'\}=\{\hat{t},\hat{t}'\}$.  This requirement restricts the permissible triples graphs that one needs to consider in step 2.c.

\item In Step 2.b, one need only consider permissible triples $(p,q,u)\in\text{sim}_F^n(\mathfrak{r})\times \text{sim}_F^n(\mathfrak{s})\times \text{sim}_F^n(\mathfrak{t})$ where $\mathfrak{r},\ \mathfrak{s}$ and $\mathfrak{t}$ are defined using the same subset $H\subset G_\mathcal{O}$.  A similar restriction applies to the subsets $P\subset\mathcal{P}(\text{Sims}^n_G)$ considered in Step 2.c.
\end{itemize}

While Algorithm \mbox{\ref{the_algorithm}} does produce the finite set of lattice surfaces asserted by Theorem \mbox{\ref{finiteness_thm}} in finite time, it does not give a halting criterion to determine when the set of returned translation surfaces is exhaustive.  An explicit test to determine the $N\in\mathbb{N}$ for which $\mathcal{P}(\text{Sims}_G)\subset\mathcal{P}(\text{Sims}^N_G)$ would give such a halting criterion (see Proposition \mbox{\ref{vor_staples_det_perm_triples_prop}}).  In particular, Algorithm \mbox{\ref{the_algorithm}} together with this halting criterion would give a general procedure to determine whether or not any given lattice group is realized as a Veech group in any given stratum.  However, even without this criterion, the ideas of this paper may still be used in special cases to show that certain lattice groups are not realized in certain strata; see \S\mbox{\ref{The modular group in minimal strata}}.

As the set of all strata is countable, Algorithm \mbox{\ref{the_algorithm}}---if allowed to run indefinitely---is easily adapted to produce all unit-area translation surfaces with a given lattice Veech group in \textit{any} stratum.  A full implementation could also provide experimental evidence for lattices which are never realized as Veech groups, though---even with the aforementioned stratum-wise halting criterion---a criterion to definitively assert that a lattice is not a Veech group in any stratum seems to the author to be elusive.

We also mention that as the set of generators for any fanning group is at most countable, Algorithm \mbox{\ref{the_algorithm}} may be adapted to construct surfaces $(X,\omega)$ with given fanning Veech groups in given strata.  However, for countably infinite generating sets $G$, the set $\mathfrak{S}_G$ is uncountable, so a construction of such surfaces may not be exhaustive.

\section{The modular group in minimal strata}\label{The modular group in minimal strata}

Here we give an example showing how the ideas of the previous sections may be used---in certain cases---to give obstructions for lattices being realized as Veech groups in strata.\footnote{Due to time constraints, we have only considered the group $\mathrm{SL}_2\mathbb{Z}$ in minimal strata.  However, the author suspects that similar methods can be applied more generally to obtain further restrictions for different groups and strata.}  In particular, we show that the square torus is the only translation surface with Veech group $\mathrm{SL}_2\mathbb{Z}$ in the collection of all minimal strata $\mathcal{H}(2g-2),\ g>0$.  Throughout this section, fix $g>0$ and $(X,\omega)\in\mathcal{H}(2g-2)$.  Let $\sigma:=\sigma_1$ be the sole singularity of $(X,\omega)$, $\Omega:=\Omega_1$ the convex body subordinate to $\widehat{\mathcal{M}}_F(X,\omega)$ and $\mathcal{O}=\mathcal{O}(2g-2)$ the canonical surface associated to $\mathcal{H}(2g-2)$.  We begin with two lemmas.  

\begin{lem}\label{num_vor_staples}
Let $\widehat{\mathcal{S}}(X,\omega)=\{\{\hat{s}_i,\hat{s}_i'\}\}_{i=1}^n$ be the marked Voronoi staples of $(X,\omega)$.  If the counterclockwise angle measured from $\hat{s}_i$ to $\hat{s}_i'$ equals that measured from $\hat{s}_i'$ to $\hat{s}_i$ for each $i$, then $n=2g$ for $n$ even and $n=2g+1$ for $n$ odd.
\end{lem}
\begin{proof}
Since $\widehat{\mathcal{S}}_F(X,\omega)$ is the set of all essential points of $\Omega$, the convex body $\Omega$ has $2n$ edges.  Decompose $\Omega$ into $2n$ closed triangular regions with disjoint interiors, where each triangle has one vertex at the origin $0\in\mathcal{O}$, and remaining two vertices at the endpoints of an edge of $\Omega$.  The sum of the interior angles of all triangles in this decomposition is $2n\pi$.  The origin $0\in\mathcal{O}$ is a singularity of cone angle $2(2g-1)\pi$, so the sum of the interior angles of $\Omega$ is 
\[2n\pi-2(2g-1)\pi=2(n-2g+1)\pi.\]

Recall that $(X,\omega)$ is isometric to the quotient space of $\Omega$ under the equivalence relation given by identifying edges determined by the marked Voronoi staples $\widehat{\mathcal{S}}(X,\omega)$, and the vertices of $\Omega$ under this identification correspond to Voronoi 0-cells of $(X,\omega)$ (\S\ref{Marked segments and Voronoi staples}).  In particular, the sum of interior angles of $\Omega$ which are identified as a single point must equal $2\pi$.  Our assumptions on the counterclockwise angles between $\hat{s}_i$ and $\hat{s}_i'$ and between $\hat{s}_i'$ and $\hat{s}_i$ imply that for $n$ even, the vertices of $\Omega$ become a single point under the identification, while for $n$ odd, the vertices of $\Omega$ become two distinct points under the identification (the argument here is analogous to that which shows the translation surface constructed by identifying opposite edges of a regular $2m$-gon has one singularity when $m$ is even and two singularities when $m$ is odd).  In particular, for $n$ even,
\[2(n-2g+1)\pi=2\pi\]
implies $n=2g$, while for $n$ odd, 
\[2(n-2g+1)\pi=4\pi\]
implies $n=2g+1$.
\end{proof}

\begin{lem}\label{P=S_F(X,omega)}
Suppose that the minimal length of a saddle connection on $(X,\omega)$ is one.  If the set
\[P:=\{p\in\mathcal{O}\ |\ |p|=1,\ \text{arg}({\bm\pi}(p))\in\{0,\pi/2,\pi,3\pi/2\}\}\]
belongs to the set of marked segments $\widehat{\mathcal{M}}_F(X,\omega)$, then in fact $P=\widehat{\mathcal{S}}_F(X,\omega)$.
\end{lem}
\begin{proof}
It suffices to show that $P$ is the set of essential points of the convex body $\Omega$, i.e. that (i) $\Omega=\cap_{p\in P}H(p)$ and (ii) for any $q\in P$, $\Omega\subsetneq\cap_{p\in P\backslash\{q\}} H(p)$.  Condition (ii) is immediate from the fact that $\Omega$ is compact while $\cap_{p\in P\backslash\{q\}} H(p)$ is non-compact for any $q\in P$.  Now
\[\Omega=\bigcap_{\hat{s}\in\widehat{\mathcal{M}}_F(X,\omega)}H(\hat{s})=\left(\bigcap_{p\in P}H(p)\right)\cap\left(\bigcap_{\hat{s}\in\widehat{\mathcal{M}}_F(X,\omega)\backslash P}H(\hat{s})\right),\]
so it suffices to show for each marked segment $\hat{s}\in\widehat{\mathcal{M}}_F(X,\omega)\backslash P$, that $\cap_{p\in P}H(p)\subset H(\hat{s})$.  

We claim that for each $p\in P$, the ball centered at $p$ of radius one contains no marked segments other than $p$, i.e.
\[B_1(p)\cap\widehat{\mathcal{M}}_F(X,\omega)\backslash\{p\}=\varnothing.\]
Suppose on the contrary that the intersection is non-empty.  As in the proof of Proposition \ref{vor_staples_det_perm_triples_prop}, we can choose some $\hat{u}$ in the intersection with argument nearest that of $p$ and find that $[p,\hat{u}]$ is a marked segment of length strictly less than one.  Its corresponding saddle connection also has length strictly less than one, contrary to our assumptions, so the claim holds.

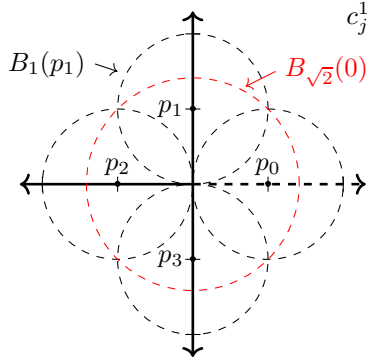
\begin{figure}[t]
\centering
\begin{tikzpicture}[scale=1]
\draw [->,dashed,line width=1] (0,0) -- (2.3,0); 
\draw [<-,line width=1] (-2.3,0) -- (0,0);
\draw [<->,line width=1] (0,-2.3) -- (0,2.3); 
\draw [dashed] (1,0) circle (1);
\filldraw [dashed] (1,0) circle (1pt) node[above] {$p_0$};
\draw [dashed] (0,1) circle (1);
\filldraw [dashed] (0,1) circle (1pt) node[left] {$p_1$};
\draw [->] (-1.3,1.6) node[left] {$B_{1}(p_1)$} -- (-1,1.5);
\draw [dashed] (-1,0) circle (1);
\filldraw [dashed] (-1,0) circle (1pt) node[above] {$p_2$};
\draw [dashed] (0,-1) circle (1);
\filldraw [dashed] (0,-1) circle (1pt) node[left] {$p_3$};
\draw [dashed,color=red] (0,0) circle ({sqrt(2)});
\draw [->,color=red] (1.1,1.5) node[right,color=red] {$B_{\sqrt 2}(0)$} -- (.7,1.3);
\draw (1,-.1) -- (1,.1);
\draw (2,-.1) -- (2,.1);
\draw (-1,-.1) -- (-1,.1);
\draw (-2,-.1) -- (-2,.1);
\draw (-.1,1) -- (.1,1);
\draw (-.1,2) -- (.1,2);
\draw (-.1,-1) -- (.1,-1);
\draw (-.1,-2) -- (.1,-2);
\draw (2.2,2.2) circle (0pt) node {$c_j^1$};
\end{tikzpicture}
\caption{The ball $B_{\sqrt 2}(0)$ is contained in the union of balls $B_1(p),\ p\in P$.  Note that this figure shows only a portion of the ball $B_{\sqrt 2}(0)$, as this ball is centered at the singularity $0\in\mathcal{O}$.  Similarly, only the upper half of $B_1(p_0)$ is shown.}
\label{ball_contained_in_balls}
\end{figure}

Next, we claim that the ball centered at the origin of radius $\sqrt 2$ is contained in the union of balls of radius one centered at the various $p\in P$,
\[B_{\sqrt 2}(0)\subset\bigcup_{p\in P} B_1(p).\]
This is clear from elementary Euclidean geometry, restricting to each $2\pi$-sector $c_j^1$; see Figure \ref{ball_contained_in_balls}.  In particular, the previous two claims imply that 
\begin{equation}\label{ball_sqrt_2}
B_{\sqrt{2}}(0)\cap \left(\widehat{\mathcal{M}}_F(X,\omega)\backslash P\right)=\varnothing.
\end{equation}

Notice that the maximum distance from the origin to any point in $\cap_{p\in P}H(p)$ is $\sqrt{2}/2$ (this distance being realized at the vertices of this intersection).  Let $\hat{s}\in\widehat{\mathcal{M}}_F(X,\omega)\backslash P$ and let $q$ belong to the complement of $H(\hat{s})$.  Certainly $|q|>|\hat{s}|/2$.  Equation \ref{ball_sqrt_2} implies $|\hat{s}|\ge\sqrt2$, so $|q|>\sqrt{2}/2$.  Hence $q\notin\cap_{p\in P}H(p)$, so the complement of $H(\hat{s})$ does not intersect $\cap_{p\in P}H(p)$.  Thus $\cap_{p\in P}H(p)\subset H(\hat{s})$ as desired.  
\end{proof}

With these two lemmas, we now prove the main result of this section.

\begin{proof}[Proof of Theorem \ref{square_torus_thm}]
Let $S=(\begin{smallmatrix}0 & -1\\ 1 & 0\end{smallmatrix})$ and $T=(\begin{smallmatrix}1 & 1\\ 0 & 1\end{smallmatrix})$, fix $\tau_S,\tau_T\in\text{Trans}(\mathcal{O})$, and set $g_S=\tau_S\circ f_S$ and $g_T=\tau_T\circ f_T$ where $f_S,f_T\in\text{Aff}_{\text{C}}^+(\mathcal{O})$.  We first show that all elements of
\[P:=\{p\in\mathcal{O}\ |\ |p|=1,\ \text{arg}({\bm\pi}(p))\in\{0,\pi/2,\pi,3\pi/2\}\}\]
belong to the same $\langle g_S,g_T\rangle$-orbit.  Note that $\text{Trans}(\mathcal{O})=\langle\rho_1\rangle$; let $n\in\mathbb{Z}_{2g-1}$ be such that $\tau_T=\rho_1^n$.  Set
\[P':=\{p\in\mathcal{O}\ |\ |p|=\sqrt 2,\ \text{arg}({\bm\pi}(p))\in\{\pi/4,3\pi/4,5\pi/4,7\pi/4\}\}.\]
Beginning with $p_0\in P\cap c_0^1,\ \text{arg}(p_0)=0$ and sweeping counterclockwise, label the elements of $P\cup P'$ as $p_0,p_1,\dots,p_{16g-9}$; see Figure \ref{P_and_P'_points}.  
\begin{figure}[b]
\centering
\begin{tikzpicture}
\draw[<-,line width=1] (-2.2,0) -- (0,0);
\draw[->,line width=1,dashed] (0,0) -- (2.2,0);
\draw[<->,line width=1] (0,-2.2) -- (0,2.2);
\filldraw (2,2) circle (0pt) node {$c_j^1$};
\filldraw (1,0) circle (1pt) node[above right] {$p_{8j}$};
\filldraw (1,1) circle (1pt) node[above right] {$p_{8j+1}$};
\filldraw (0,1) circle (1pt) node[above right] {$p_{8j+2}$};
\filldraw (-1,1) circle (1pt) node[above left] {$p_{8j+3}$};
\filldraw (-1,0) circle (1pt) node[above left] {$p_{8j+4}$};
\filldraw (-1,-1) circle (1pt) node[below left] {$p_{8j+5}$};
\filldraw (0,-1) circle (1pt) node[below right] {$p_{8j+6}$};
\filldraw (1,-1) circle (1pt) node[below right] {$p_{8j+7}$};
\end{tikzpicture}
\caption{Labelling of the points $p_0,\dots,p_{16g-9}$ of $P\cup P'$ for $j\in\mathbb{Z}_{2g-1}$.}
\label{P_and_P'_points}
\end{figure}
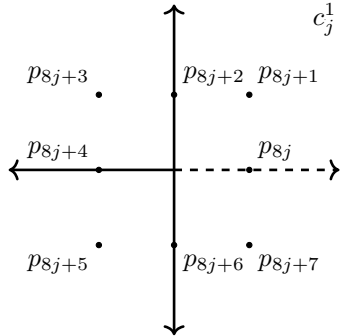
For each $k\in\mathbb{Z}_{16g-8}$, we have
\[f_S(p_k)=p_{k+2},\]
and for $k\not\equiv 1$ (mod $4$),
\[f_T(p_k)=\begin{cases}
p_k &\text{if}\ k\equiv 0\ (\text{mod 4})\\
p_{k-1} &\text{if}\ k\equiv 2,3\ (\text{mod 4})
\end{cases}.\]
Moreover, for each $k\in\mathbb{Z}_{16g-8}$,
\[\tau_T(p_k)=p_{k+8n}.\]
Set $f:=g_T^{-1}\circ g_S^{-1}\circ g_T^{-1}\circ g_S$ and $h:=g_S^{-1}\circ g_T^{-1}\circ g_S\circ g_T^{-1}$.  Using the relations above---together with Lemma \ref{centralizer_lemma} and the fact that $\text{Trans}(\mathcal{O})=\langle\rho_1\rangle$ is abelian---we find for each $k\in\mathbb{Z}_{16g-8}$ that
\begin{align*}
f(p_{4k})&=g_T^{-1}\circ g_S^{-1}\circ g_T^{-1}\circ g_S(p_{4k})\\
&=(f_T^{-1}\circ\tau_T^{-1})\circ (f_S^{-1}\circ\tau_S^{-1})\circ (f_T^{-1}\circ\tau_T^{-1})\circ (\tau_S\circ f_S)(p_{4k})\\
&=\tau_T^{-2}\circ f_T^{-1}\circ f_S^{-1}\circ f_T^{-1}\circ f_S(p_{4k})\\
&=\tau_T^{-2}\circ f_T^{-1}\circ f_S^{-1}\circ f_T^{-1}(p_{4k+2})\\
&=\tau_T^{-2}\circ f_T^{-1}\circ f_S^{-1}(p_{4k+3})\\
&=\tau_T^{-2}\circ f_T^{-1}(p_{4k+1})\\
&=\tau_T^{-2}(p_{4k+2})\\
&=p_{4k-16n+2}
\end{align*}
and
\begin{align*}
h(p_{4k+2})&=g_S^{-1}\circ g_T^{-1}\circ g_S\circ g_T^{-1}(p_{4k+2})\\
&=(f_S^{-1}\circ\tau_S^{-1})\circ (f_T^{-1}\circ\tau_T^{-1})\circ (\tau_S\circ f_S)\circ (f_T^{-1}\circ\tau_T^{-1})(p_{4k+2})\\
&=\tau_T^{-2}\circ f_S^{-1}\circ f_T^{-1}\circ f_S\circ f_T^{-1}(p_{4k+2})\\
&=\tau_T^{-2}\circ f_S^{-1}\circ f_T^{-1}\circ f_S(p_{4k+3})\\
&=\tau_T^{-2}\circ f_S^{-1}\circ f_T^{-1}(p_{4k+5})\\
&=\tau_T^{-2}\circ f_S^{-1}(p_{4k+6})\\
&=\tau_T^{-2}(p_{4k+4})\\
&=p_{4k-16n+4}.
\end{align*}
Then
\begin{align*}
g_T^8\circ h\circ f\circ h\circ f(p_0)&=g_T^8\circ h\circ f\circ h(p_{-16n+2})\\
&=g_T^8\circ h\circ f(p_{-32n+4})\\
&=g_T^8\circ h(p_{-48n+6})\\
&=g_T^8(p_{-64n+8})\\
&=\tau_T^8\circ f_T^8(p_{-64n+8})\\
&=\tau_T^8(p_{-64n+8})\\
&=p_8.
\end{align*}
Iterating this latter composition shows that each $p_{8k}$ belongs to the same $\langle g_S,g_T\rangle$-orbit.  Injectivity of $g_S$ implies that the union of the images under $g_S,\ g_S^2$ and $g_S^3$ of these $p_{8k}$ is all of $P$, as claimed.  

Now suppose $(X,\omega)\in\mathcal{H}(2g-2)$ with $\mathrm{SL}(X,\omega)=\mathrm{SL}_2\mathbb{Z}=\langle S,T\rangle$ for some $g>0$.  Normalize $(X,\omega)$ if necessary so that its shortest saddle connection has unit length.  One computes
\[S^1\backslash (T\cdot\Delta\cup T^{-1}\cdot\Delta\cup ST\cdot\Delta\cup ST^{-1}\cdot\Delta)=\{(\pm 1,0),(0,\pm 1)\}.\]
This set contains $\Theta_{\mathrm{SL}_2\mathbb{Z}}=\Theta_{\mathrm{SL}(X,\omega)}$, so Corollary \ref{directions_for_orbits} implies that the shortest saddle connection of $(X,\omega)$ is horizontal or vertical.  Since the orthogonal matrix $S$ belongs to $\mathrm{SL}(X,\omega)$ and any affine automorphism of $(X,\omega)$ sends saddle connections to saddle connections, $(X,\omega)$ has both horizontal and vertical saddle connections of minimal length one.

Recycling notation, let $f\in\text{Aff}^+(\mathcal{O})$ and $\{p,p^-\}\in\mathbb{P}(\mathcal{O})$.  Lemma \ref{pi*f=A*pi} together with the fact that $f$ is a self-homeomorphism of $\mathcal{O}=\mathcal{O}(2g-2)$ implies that the counterclockwise angle measured from $p$ to $p^-$ equals that measured from $f(p)$ to $f(p^-)$.  In particular, this is true of any $f\in\text{Aff}_\mathcal{O}^+(X,\omega)$ and horizontal, unit-length $\{\hat{s},\hat{s}'\}\in\widehat{\mathcal{M}}(X,\omega)$.  Since $S,T\in\mathrm{SL}(X,\omega)$, there exist $g_S,g_T\in\text{Aff}_\mathcal{O}^+(X,\omega)$ with $\text{der}(g_S)=S,\ \text{der}(g_T)=T$ (\S\ref{The group Aff^+_O(X,omega) and its action on marked segments}).  The argument above shows that the set $P$ belongs to the $\langle g_S,g_T\rangle$-orbit of $\hat{s}$.  Since this orbit belongs to $\widehat{\mathcal{M}}_F(X,\omega)$, Lemma \ref{P=S_F(X,omega)} implies that $P=\widehat{\mathcal{S}}_F(X,\omega)$.  Moreover, $\hat{s}'$ belongs to this orbit, so the observation above implies that the counterclockwise angle measured from $\hat{s}$ to $\hat{s}'$ equals that measured from $\hat{s}'$ to $\hat{s}$, and the same is true of any pair of marked Voronoi staples in $\widehat{\mathcal{S}}(X,\omega)$.  Since 
\[|\widehat{\mathcal{S}}(X,\omega)|=|\widehat{\mathcal{S}}_F(X,\omega)|/2=|P|/2=2(2g-1)\]
is even, Lemma \ref{num_vor_staples} implies that $2(2g-1)=2g$, which is only true for $g=1$.  When $g=1$, the convex body $\Omega$ is the unit-square centered at $0\in\mathcal{O}=(\mathbb{C},dz)$, and the reconstruction of $(X,\omega)$ from $\Omega$ gives the square torus.
\end{proof}

\bigskip

{\flushleft \textbf{Acknowledgments.}} The author thanks Tom Schmidt for many helpful discussions and suggestions, including---but not limited to---conjecturing the reverse implication of the first statement of Lemma \ref{lattice_and_fanning}.  Thanks also to Aaron Calderon and Sunrose Shrestha for pointing out the appearance of Theorem \ref{finiteness_thm} in \cite{Smillie-Weiss10_fin}; the independent proof given here was written prior to the author's knowledge of the original result of Smillie and Weiss.  The author also thanks the anonymous referee whose suggestions greatly improved the exposition of this paper.

{\flushleft \textbf{Funding.}}  The author was supported in part by the Department of Mathematics at Oregon State University and the Mathematical Institute at Utrecht University.  This work is part of project number 613.009.135 of the research programme Mathematics Clusters which is (partly) financed by the Dutch Research Council (NWO).

{\flushleft \textbf{Conflicts of interest.}}  The author has no relevant financial or non-financial interests to disclose.

{\flushleft \textbf{Data availability statement.}} Data sharing not applicable to this article as no datasets were generated or analysed during the current study.\\

This version of the article has been accepted for publication, after peer review but is not the Version of Record and does not reflect post-acceptance improvements, or any corrections. The Version of Record is available online at: \url{http://dx.doi.org/10.1007/s10711-023-00818-7}.



\begin{thebibliography}{9}



\bibitem[Be]{Beardon}
A. F. Beardon.
\newblock {\em The geometry of discrete groups.}
\newblock Springer New York (1983).  

\bibitem[BM]{BM}
I. I. Bouw and M. M\"oller.
\newblock {\em Teichm\"uller curves, triangle groups, and Lyapunov exponents.}
\newblock Ann. of Math., Vol 2 No. 1 (2010), 139--185.


\bibitem[Bo]{Bowman}
J. P. Bowman.
\newblock {\em Teichm\"uller geodesics, Delaunay triangulations, and Veech groups.}
\newblock Teichm\"uller Theory and Moduli Problems, Ramanujan Math. Society Lecture Notes Series, Vol. 10 (2010), 113--129.

\bibitem[BrJ]{BrJ}
S. A. Broughton and C. Judge.
\newblock {\em Ellipses in translation surfaces.}
\newblock Geom. Dedicata, Vol. 157 No. 1 (2012), 111--151.

\bibitem[DPU]{DaPaU}
D. Davis, I. Pasquinelli, and C. Ulcigrai.
\newblock {\em Cutting sequences on Bouw-M\"oller surfaces: an $\mathcal{S}$-adic characterization.}
\newblock Ann. Sci. de l'\'Ec. Norm. Sup\'er., Vol. 52 No. 4 (2019), 927--1023.

\bibitem[Ed]{Edwards}
B. Edwards.
\newblock {\em A new algorithm for computing the Veech group of a translation surface.}
\newblock Oregon State University, PhD Dissertation (2017).

\bibitem[ESS]{ESS}
B. Edwards, S. Sanderson, and T. A. Schmidt.
\newblock {\em Canonical translation surfaces for computing Veech groups.}
\newblock Geom. Dedicata, Vol. 216 No. 5 Paper No. 60 (2022), 20 pp.

\bibitem[EMM]{EMM}
A. Eskin, M. Mirzakhani, and A. Mohammadi.
\newblock {\em Isolation, equidistribution, and orbit closures for the $\mathrm{SL}(2,\mathbb{R})$ action on moduli space.}
\newblock Ann. of Math., Vol. 182 No. 2 (2015), 673--721.

\bibitem[Fr]{Freidinger}
M. Freidinger.
\newblock {\em Stabilisatorgruppen in $\text{Aut}(F_2)$ und Veechgruppen von \"Uberlagerungen.}
\newblock Diplome Thesis, Universit\"at Karlsruhe (2008).

\bibitem[GJ]{GJ}
E. Gutkin and C. Judge.
\newblock {\em Affine mappings on translation surfaces: geometry and arithmetic.}
\newblock Duke Math. J., Vol. 103 No. 2 (2000), 191--213.

\bibitem[Ho]{Ho}
W. P. Hooper.
\newblock {\em Grid graphs and lattice surfaces.}
\newblock Int. Math. Res. Not. IMRN, No. 12 (2013), 2657--2698.

\bibitem[HL]{HL}
P. Hubert and E. Lanneau.
\newblock {\em Veech groups without parabolic elements.}
\newblock Duke Math. J., Vol. 133 No. 2 (2006), 335--346.

\bibitem[HMSZ]{HMSZ}
P. Hubert, H. Masur, T. A. Schmidt, and A. Zorich.
\newblock {\em Problems on billiards, flat surfaces and translation surfaces.}
\newblock in Problems on mapping class groups and related topics, ed. B. Farb, Proc. Symp. Pure Math. 74, AMS (2006), 233--243.

\bibitem[HS]{HS}
P. Hubert and T. A. Schmidt
\newblock {\em Infinitely generated Veech groups.}
\newblock Duke Math. J., Vol. 123 No. 1 (2004), 49--69.

\bibitem[Ka]{Katok}
S. Katok.
\newblock {\em Fuchsian groups.}
\newblock University of Chicago Press (1992).

\bibitem[KS]{KS}
R. Kenyon and J. Smillie.
\newblock {\em Billiards on rational-angled triangles.}
\newblock Comment. Math. Helv., Vol. 75 No.1 (2000), 65--108.

\bibitem[MS]{Masur-Smillie}
H. Masur and J. Smillie.
\newblock {\em Hausdorff dimension of sets of nonergodic measured foliations.}
\newblock Ann. of Math., Vol. 134 No. 3 (1991), 455--543.

\bibitem[Mc]{Mc03}
C. McMullen.
\newblock {\em Teichm\"uller geodesics of infinite complexity.}
\newblock Acta Math., Vol. 191 No. 2 (2003), 191--223.

\bibitem[Mc2]{McMullen}
\rule{.5in}{.5pt}
\newblock {\em Rigidity of Teichm\"uller curves.}
\newblock Math. Res. Lett., Vol. 16 No. 4 (2009), 647--649.

\bibitem[Mu]{Mukamel}
R. Mukamel.
\newblock {\em Fundamental domains and generators for lattice Veech groups.}
\newblock Comment. Math. Helv., Vol. 92 No. 1 (2017), 57--83.

\bibitem[Ra]{Ratner}
M. Ratner.
\newblock {\em Raghunathan's topological conjecture and distributions of unipotent flows.}
\newblock Duke Math. J., Vol. 63 No. 1 (1991), 235--280.

\bibitem[Sa]{SageMath}
The Sage Developers.
\newblock{\em SageMath, the Sage Mathematics Software System (Version 9.4).}
\newblock 2021, \url{https://www.sagemath.org}.

\bibitem[Sch]{Schmithusen}
G. Schmith\"usen.
\newblock {\em An algorithm for finding the Veech group of an origami.}
\newblock Exper. Math., Vol. 13 No. 4 (2004), 459--472.

\bibitem[SmW]{Smillie-Weiss04}
J. Smillie and B. Weiss.
\newblock{\em Minimal sets for flows on moduli space.}
\newblock Israel J. Math. Vol., 142 No. 1 (2004), 249--260.

\bibitem[SmW2]{Smillie-Weiss10}
\rule{.5in}{.5pt}
\newblock{\em Characterizations of lattice surfaces.}
\newblock Invent. Math., Vol. 180 No. 3 (2010), 535--557.

\bibitem[SmW3]{Smillie-Weiss10_fin}
\rule{.5in}{.5pt}
\newblock{\em Finiteness results for flat surfaces: large cusps and short geodesics.}
\newblock Comment. Math. Helv., Vol. 85 No. 2 (2010), 313--336.

\bibitem[Ve]{Veech89}
W. A. Veech.
\newblock {\em Teichm\"uller curves in moduli space, Eisenstein series and an application to triangular billiards.}
\newblock Invent. Math., Vol. 97 (1989), 553--583.

\bibitem[Ve2]{Veech95}
\rule{.5in}{.5pt}
\newblock {\em Geometric realizations of hyperelliptic curves.}
\newblock Algorithms, fractals, and dynamics (Okayama/Kyoto, 1992) (1995), 217--226.

\bibitem[Ve3]{Veech11}
\rule{.5in}{.5pt}
\newblock {\em Bicuspid $F$-structures and Hecke groups.}
\newblock Proc. London Math. Soc., Vol. 103 No. 4 (2011) 710--745.

\bibitem[Vo]{Vorobets}
Y. B. Vorobets.
\newblock {\em Planar structures and billiards in rational polygons: The Veech alternative.}
\newblock Russian Math. Surveys, Vol. 51 No. 5 (1996), 779--817.

\bibitem[Wr]{Wright}
A. Wright.
\newblock {\em Translation surfaces and their orbit closures: An introduction for a broad audience.}
\newblock EMS Surv. Math. Sci., Vol. 2 No. 1 (2015), 63--108.

\bibitem[Zo]{Zorich}
A. Zorich.
\newblock {\em Flat Surfaces.}
\newblock Frontiers in number theory, physics, and geometry, Vol. 1 (2006), 439--586.
\end{thebibliography}
\end{document}